\documentclass[a4paper,10pt]{article}
\usepackage[utf8]{inputenc}
\usepackage{LatexDefinitions}
\usepackage{LatexDefinitions}
\usepackage{amssymb}
\usepackage{caption}
\usepackage{dsfont}

\title{Entropic transfer operators for stochastic systems}
\author{Hancheng Bi, Clément Sarrazin, Bernhard Schmitzer, Thilo D. Stier}
\begin{document}
\maketitle
\begin{abstract}
  Dynamical systems can be analyzed via their Frobenius--Perron transfer operator and its estimation from data is an active field of research. 
  Recently entropic transfer operators have been introduced to estimate the operator of deterministic systems. The approach is based on the regularizing properties of entropic optimal transport plans.
  In this article we generalize the method to stochastic and non-stationary systems and give a quantitative convergence analysis of the empirical operator as the available samples increase. We introduce a way to extend the operator's eigenfunctions to previously unseen samples, such that they can be efficiently included into a spectral embedding.
  The practicality and numerical scalability of the method are demonstrated on a real-world fluid dynamics experiment.
\end{abstract}

\section{Introduction} \label{sec:intro}
\subsection{Motivation and related work} \label{sec:motivation}
\paragraph{Dynamical systems and transfer operators.}
A time-discrete stochastic dynamical system can be described by a state space $\SpX$ and a transition kernel $(\kappa_x)_{x \in \SpX}$ where the probability measure $\kappa_x \in \prob(\SpX)$ gives the conditional distribution for the state $x_{t+1} \in \SpX$ at time $t+1$, conditioned on the observation that $x_t=x$.
Time-continuous systems can be captured in this description by integrating the corresponding stochastic differential equation over a (small) time interval $\tau>0$.
Dynamical systems are a versatile modelling tool and can be used to describe population dynamics \cite{ZhaoDynSys2017}, molecular dynamics \cite{SchuetteMilestoning2011}, chemical reaction networks \cite{SingerPNAS2009}, fluid dynamics \cite{Koltai_2020}, meteorology \cite{Froyland2021Climate}, and many other phenomena.

Systems of interest are often chaotic, stochastic, and high-dimensional.
Therefore, even if individual trajectories $(x_t)_t$ can be measured or simulated with high precision, it is difficult to obtain a structured understanding of the system's behaviour by direct inspection of such data, due to the sensitivity on the starting point, stochasticity, and high dimension.
Instead, one usually looks for a coarse-grained description, e.g.~by identifying metastable states, or low-dimensional effective coordinates that capture the dynamics on slower time scales.
One ansatz for obtaining such descriptions is via the \emph{transfer operator} $T : \prob(\SpX) \to \prob(\SpX)$, that describes how an ensemble of particles evolves. For instance, if the particles at discrete time $t$ are distributed according to some probability measure $\mu_t \in \prob(\SpX)$, then at time $t+1$ they will be distributed according to $\mu_{t+1} \assign T \mu_t$. For a transition kernel $(\kappa_x)_x$ as mentioned above, $T$ would formally be characterized by
\begin{align*}
  \int_{\SpX} \phi(y)\, \diff \mu_{t+1}(y) = \int_{\SpX} \left[ \int_{\SpX} \phi(y) \, \diff \kappa_x(y) \right] \diff \mu_t(x)
\end{align*}
for continuous test functions $\phi \in \Cont(\SpX)$.
$T$ is a linear operator so we can make use of methods of functional analysis to study it. Often, the restriction of $T$ to subsets of $\prob(\SpX)$ is analyzed, such as probability measures with densities in $\Leb^p(\mu)$ with respect to some reference measure $\mu$. The adjoint of $T$ is called the \emph{Koopman operator} $K=T^*$.\footnote{This adjoint is conventionally taken with respect to the pairing $(L^1(\mu),L^\infty(\mu))$ where $T$ is interpreted as operator on $L^1(\mu)$ \cite[Section 3]{LasotaMackey}, but depending on context, other pairings such as $(L^2(\mu),L^2(\mu))$ or Radon measures and bounded continuous functions may be appropriate.}
Spectral analysis of $T$, e.g.~by eigendecomposition can yield information about the long-term behaviour of the underlying dynamical system, and a corresponding coarse approximate description through spectral embedding \cite{Coifman2006}.

\paragraph{Approximation by compact operators and from discrete data.}
Often, an analytic description of $T$ is not available or not tractable for direct analysis. This difficulty may be exacerbated when $T$ is not a compact operator.
It is therefore important to find compact approximations $T^\veps$ of $T$, where $\veps>0$ is some regularization parameter. A common strategy is to consider a small stochastic perturbation of $T^\veps$, e.g.~by composing it with a small blur step \cite{DelJun1999,Froyland2013Coherent}.
Subsequently one seeks a discrete approximation $T_N^\veps$ of $T^\veps$ based on empirical or simulated data, where $N$ denotes the amount of available samples $(x_i,y_i)_{i=1}^N$.

Usually the $x_i$ are assumed to be independently identically distributed (i.i.d.) random variables with law $\mu \in \prob(\SpX)$ and $y_i$ are corresponding states observed after one time step, i.e.~the law of $y_i$ conditioned on $x_i=x$ is given by $\kappa_{x}$.
One can then interpret the pairs $(x_i,y_i)$ as i.i.d.~random variables with law $\pi \in \prob(\SpX \times \SpX)$ where $(\kappa_{x})_x$ is the disintegration of $\pi$ with respect to its first marginal (which equals $\mu$). The second marginal of $\pi$, which is the law of $y_i$ when not conditioned on $x_i$, is given by $\nu \assign T \mu$.
Specifying the measure $\pi$ is enough to fix $T$ as an operator from $\Leb^p(\mu)$ to $\Leb^p(\nu)$ (see \Cref{prop:inducedTO}).
Alternatively, one could consider data gathered from a single long trajectory of random variables $(z_t)_{t=1}^{N+1}$ where the law of $z_{t+1}$ conditioned on $z_t=x$ is given by $\kappa_x$.
When the system is ergodic, the i.i.d.~assumption in the former option is still a good approximation for the latter case if one interprets $(z_{t},z_{t+1})$ as a pair $(x_t,y_t)$ for $t=1,\ldots,N$.

In Ulam's celebrated and prototypical method \cite{Ulam60,li1976finite}, a finite partition $(\SpX_i)_i$ of $\SpX$ is introduced and the observed transitions of data points between partition cells offer a discrete approximation of the Markov kernel $(\kappa_x)_x$ at the level of the cells. The literature on approximating and analyzing $T$ from data is vast and we refer to \cite{DelJun1999,Koopmanism2012,Froyland2013Coherent,KluKolSch2016,TransferOperatorApproxReview18} or the monograph \cite{eisner2015operator} and references therein for exemplary starting points and \cite{BittracherTransitionManifolds2018,BittracherKernelTransitionManifolds21,BittracherRC2021,YangOTDynSys2021} for snapshots of recent developments.

\paragraph{Entropic transfer operators.}
In \cite{EntropicTransfer22} \emph{entropic transfer operators} were introduced. These can be interpreted as a partition-free variant of Ulam's method that works directly on the sample point cloud and mitigates discretization artefacts by a blurring kernel that is generated via entropic optimal transport.
\cite{EntropicTransfer22} considers time-discrete \emph{deterministic} dynamical systems where the state $x_{t+1}$ is given as $F(x_t)$ for a continuous evolution map $F : \SpX \to \SpX$. 
The transfer operator $T$ is then given by the push-forward $T\mu= F_{\#} \mu$, which is linear but not generally a compact operator. Given an invariant measure $\mu$ of $T$, i.e.~$\mu=T \mu$, \cite{EntropicTransfer22} then considers the restriction of $T$ to $\Leb^2(\mu)$ and constructs a compact approximation $T^\veps \assign G^\veps_{\mu\mu} T$ by composing $T$ with a transfer operator $G^\veps_{\mu\mu}$ induced by the entropic transport plan of $\mu$ onto itself (see Sections \ref{subsec:Opt_Trans} and \ref{sec:transfer} below for details) where $\veps$ denotes the strength of entropic regularization.
It was shown that $G_{\mu\mu}^\veps$ can be thought of as an operator that introduces blur at a length scale $\sqrt{\veps}$ while preserving the invariant measure $\mu$ (unlike the more naive blur used, for instance, in \cite{DelJun1999,Froyland2013Coherent}), and thus $T^\veps$ can be interpreted as a compactification of $T$ that preserves $\mu$ and the dynamics on length scales above $\sqrt{\veps}$.
Then an approximation $T_N^\veps : \Leb^2(\mu_N) \to \Leb^2(\mu_N)$ is introduced where $\mu_N\assign \tfrac1N \sum_i \delta_{x_i}$ is the random empirical measure associated with the random variables $(x_i)_i$.
For a suitable extension of $T_N^\veps$ from $\Leb^2(\mu_N)$ to $\Leb^2(\mu)$ qualitative convergence towards $T^\veps$ in Hilbert--Schmidt norm is then shown for fixed $\veps>0$ as $N \to \infty$.

In this article we expand the construction and analysis given in \cite{EntropicTransfer22} in several ways. In particular we modify the definition of $T^\veps$ and $T_N^\veps$ such that convergence also holds when the original system $T$ is stochastic, i.e.~not induced by a deterministic map $F$ but by a more general transition kernel $(\kappa_x)_x$.

\subsection{Contribution and outline} \label{sec:contribution}
Throughout the rest of Section \ref{sec:intro} we collect necessary notation and concepts on entropic optimal transport and transfer operators.

\paragraph{Introduction of double-blurred entropic transfer operator.}
In Section \ref{subsec:main_defs} we consider a measure $\pi \in \prob(\SpX \times \SpX)$ with marginals $\mu$ and $\nu$ and its induced transfer operator $T : \Leb^p(\mu) \to \Leb^p(\nu)$.
We then introduce a compact approximation
$$T^\veps \assign G_{\nu\nu}^\veps T G_{\mu\mu}^\veps  \;:\; \Leb^2(\mu) \to \Leb^2(\nu)$$
where $G_{\nu\nu}^\veps$ and $G_{\mu\mu}^\veps$ are entropic transport blur operators.
In a manner similar to \cite{EntropicTransfer22}, $T^\veps$ can be thought of as compact approximation of $T$ that preserves dynamics on length scales above $\sqrt{\veps}$ and is thus more amenable to interpretation or estimation from data.

Given data in the form of observed transitions $(x_i,y_i)_{i=1}^N$, we then define the empirical approximation
$$T^\veps_N \assign G_{\nu_N\nu_N}^\veps T_N G_{\mu_N\mu_N}^\veps \;:\; \Leb^2(\mu_N) \to \Leb^2(\nu_N).$$
Here $G_{\nu_N\nu_N}^\veps$ and $G_{\mu_N\mu_N}^\veps$ are empirical entropic transport blur operators and $T_N : \Leb^2(\mu_N)\allowbreak \to \Leb^2(\nu_N)$ is the operator induced by the $(x_i,y_i)_i$.
$T^\veps_N$ can be constructed and analyzed numerically, e.g.~its dominant singular values and vectors can be computed. Therefore the main question of this article is, how spectral analysis of $T^\veps_N$ relates to the regularized full operator $T^\veps$ or even $T$ itself (when the latter is already compact).

To this end we extend $T^\veps_N$ to an operator $T^{A,\veps}_N : \Leb^2(\mu) \to \Leb^2(\nu)$ by isometrically embedding $\Leb^2(\mu_N)$ and $\Leb^2(\nu_N)$ into $\Leb^2(\mu)$ and $\Leb^2(\nu)$ via piecewise constant functions. Therefore, $T^{A,\veps}_N$ has the same non-zero spectrum as $T^\veps_N$ (cf.~Section \ref{subsec:SpectralConvergence}). Note that only $T^{\veps}_N$ will be relevant for numerical methods and that $T^{A,\veps}_N$ is merely introduced for theoretical analysis. The main theoretical contribution of this article is then to quantify the (probabilistic) convergence of $T^{A,\veps}_N$ to $T^\veps$ in Hilbert--Schmidt norm, as $N \to \infty$. For this we introduce two additional auxiliary operators $T^{B,\veps}_N$ and $T^{C,\veps}_N$. All defined operators and their relations are summarized in Figure \ref{fig:operator_overview}.

Compared to the original definition in \cite{EntropicTransfer22} here the definition of $T^\veps$ no longer assumes that $\mu$ is an invariant measure (i.e.~$\mu \neq \nu$ in general) and two blurring steps are applied (similar to \cite{Froyland2013Coherent}). This is needed to ensure that the extension $T_N^{A,\veps}$ still converges to $T^\veps$ as $N \to \infty$ when $T$ is not induced by a deterministic continuous map $F$ (see Section \ref{subsec:exampleNonConv} for an example where a single blur operation is not sufficient).

\paragraph{Quantitative probabilistic convergence analysis of $T^{A,\veps}_N$ to $T^\veps$.}
The main mathematical contribution of this article is the quantitative analysis of the convergence of the extended empirical regularized operator $T_N^{A,\veps}$ to the true regularized $T^\veps$, extending the qualitative approach of \cite{EntropicTransfer22}. Some preliminary results are established in Section \ref{subsec:preliminary}.
The convergence itself is developed throughout Section \ref{subsec:quant_cvg} in three steps.
\textbf{\Cref{prop:cross_operators}} shows that $\| T_N^{A,\veps}-T_N^{B,\veps} \|_{\HS} \to 0$ as $N \to \infty$ with a dimension-dependent rate which is related to the sample complexity of unregularized optimal transport. This is expected, since the kernel $t^{A,\veps}_\veps$ of $T_N^{A,\veps}$ turns out to be a piecewise constant approximation of the kernel $t_N^{B,\veps}$ of $T_N^{B,\veps}$.
\textbf{\Cref{prop:prob_bias_bound}} then shows $\| T_N^{B,\veps}-T_N^{C,\veps} \|_{\HS} \to 0$ as $N\to \infty$ with parametric rate where the effective dimension (see \Cref{ass:min_ball_mass} for details) of the measures $\mu$ and $\nu$ enters in the constant.
The key step is to control the discrepancy between the kernels of $G_{\mu_N\mu_N}^\veps$ and $G_{\mu\mu}^\veps$ (and likewise for $\nu^N$ and $\nu$) with results on the sample complexity of entropic optimal transport \cite{luise2019sinkhorn}.
\textbf{\Cref{prop:prob_var_bound}} shows that $\| T_N^{C,\veps}-T^{\veps} \|_{\HS} \to 0$ as $N\to\infty$ with almost parametric rate (where the dimension enters again in the constant). For this the discrepancy between the true $\pi$ and its empirical approximation $\pi_N \assign \tfrac1N \sum_i \delta_{(x_i,y_i)}$ is accounted for with a concentration inequality that leverages the regularity of the kernel of $G_{\mu\mu}^\veps$.

\paragraph{Further convergence results.}
Sections \ref{subsec:cv_Teps2T}, \ref{subsec:stationary}, and \ref{subsec:SpectralConvergence} collect further convergence results with practical relevance for data analysis.
Section \ref{subsec:cv_Teps2T} addresses the convergence of $T^\veps$ to $T$ as $\veps \to 0$ under the assumption that $T$ has a kernel with Hölder-type continuity. This serves to illustrate that $T^{\veps}_N$ may not only approximate $T^\veps$ as $N \to \infty$ for fixed $\veps>0$, but also potentially $T$ directly in some joint limit $N\to \infty$ and $\veps \to 0$, if $T$ is sufficiently regular. Section \ref{subsec:stationary} gives an adjusted definition of $T_N^{\veps}$ (and all related operators) for the stationary setting where $\mu=\nu$. All prior results canonically carry over to this setting. Section \ref{subsec:SpectralConvergence} collects several results on the convergence of eigen- and singular value decompositions, which ensure that analysis of $T^\veps_N$ ultimately reveals properties of $T^\veps$ or $T$.

\paragraph{Out-of-sample embedding.}
Section \ref{subsec:out_of_sampling} then shows that the regularity of entropic optimal transport can be used to construct an extension of eigen- and singular functions of $T^\veps_N$ to the whole space $\SpX$, which can be used to obtain out-of-sample embeddings for new samples, when a spectral decomposition has previously been computed on a smaller subset of samples.
This is reminiscent of the Nystr\"om approximation of kernel matrices and related subsampling methods for (kernel) PCA \cite{WilliamsSeeger2000,Achlioptas2001,Coifman2006b}. However, our extension is based directly on the regularity of the entropic transport kernel and does not rely on pseudo inverses.

\paragraph{Examples and numerical experiments.}
Finally, \Cref{sec:numerics} gathers (numerical) examples. Section \ref{subsec:exampleNonConv} underscores the importance of double blurring when working with stochastic systems.
Section \ref{subsec:workflow} describes the algorithmic workflow for numerically analyzing a new empirical dataset. Section \ref{subsec:num_1d} illustrates the convergence behaviour of entropic transfer operators on the simple synthetic example of a stochastic shift on a torus.
A numerical comparison with Ulam's method on a synthetic example is given in Section \ref{sec:Ulam}.
Section \ref{sec:convection} analyses a dataset from fluid dynamics that was previously examined by a combination of diffusion maps and Ulam's method in \cite{Koltai_2020}. This demonstrates that entropic transfer operators are a robust and transparent method (with only a single parameter) that can scale to large datasets by the use of contemporary GPU hardware and suitable software \cite{Keops}.

\paragraph{Relation to \cite{beier2024transfer}.}
In this article we construct an approximation of the transfer operator $T$ from observed transitions $(x_i,y_i)_{i=1}^N$. In \cite{beier2024transfer} a variant of the problem is considered where points are observed in $N$ batches of $M$ particles per batch, i.e.~for each $i \in \{1,\ldots,N\}$ one obtains $M$ point pairs $(x_{i,j},y_{i,j})_{j=1}^M$, but the association between the points is not observed, instead $y_{i,j}$ is obtained as the evolution of $x_{i,\sigma_i(j)}$ for some unknown random permutation $\sigma_i$. It is then shown that one can still recover an approximation $\widehat{T}^\veps_N$ of the transfer operator by taking an ansatz $\widehat{T}^\veps_N=G_{\nu_N\nu_N}^\veps Q G_{\mu_N\mu_N}^\veps$ and optimizing a suitable approximate log-likelihood with respect to $Q$. The blur operators $G_{\nu_N\nu_N}^\veps$ and $G_{\mu_N\mu_N}^\veps$ serve to limit the bandwidth of the approximation and thus control the variance of the estimator. This ansatz is similar to the form $T^\veps_N  \assign G_{\nu_N\nu_N}^\veps T_N G_{\mu_N\mu_N}^\veps$ that we consider here.

The main objectives of \cite{beier2024transfer} are to show qualitative convergence (under suitable assumptions) of maximizers $\widehat{T}^\veps_N$ of the approximate likelihood to the true operator $T$ as $N \to \infty$, and to devise a numerical algorithm for optimizing over $Q$. In contrast, in the present article the `middle' operator $T_N$ is directly observed and quantitative convergence $T_N^\veps \to T^\veps$ is established.

\subsection{Setting and notation}
Throughout this article, let $(\SpX,d)$ be a compact metric space. We equip $\SpX \times \SpX$ with the metric $((x,y),(x',y')) \mapsto \sqrt{d(x,x')^2 + d(y,y')^2}$.
At some points we will assume in addition that $\SpX$ is a subset of $\R^d$ equipped with the Euclidean distance metric. This will be mentioned explicitly.
For a compact metric space $\SpZ$, denote by $\Cont(\SpZ)$ the Banach space of continuous real-valued functions on $\SpZ$, equipped with the supremum norm. We identify its dual space with the space of (finite) Radon measures on $\SpZ$, denoted by $\meas(\SpZ)$.
The subsets of non-negative and probability measures are denoted by $\measp(\SpZ)$ and $\prob(\SpZ)$ respectively.
For $\sigma, \tau \in \meas(\SpZ)$ we denote by $\KL(\sigma|\tau)$ the Kullback--Leibler divergence of $\sigma$ with respect to $\tau$, i.e.
\begin{align*}
  \KL(\sigma|\tau) \assign \begin{cases}
    \displaystyle \int_Z \varphi\left(\RadNik{\sigma}{\tau}\right)\,\diff \tau & \tn{if } \sigma, \tau \geq 0,\;\sigma \ll \tau, \\
    \displaystyle +\infty & \tn{otherwise,}
  \end{cases}
  \quad \tn{where} \quad
  \varphi : \R \ni s \mapsto \begin{cases}
    s \log s - s + 1 & \tn{for } s>0, \\
    1 & \tn{for } s=0, \\
    +\infty & \tn{otherwise.}
  \end{cases}
\end{align*}
For two compact metric spaces $\SpZ_1$, $\SpZ_2$,  $\mu\in\meas(\SpZ_1)$, and a measurable function $f : \SpZ_1\to\SpZ_2$, the push forward measure $f_\#\mu\in\meas(\SpZ_2)$ is defined by the relation
\begin{equation*}
  \int_{\SpZ_2} h\, \diff f_\#\mu  = \int_{\SpZ_1} h\circ f\, \diff\mu 
\end{equation*}
for any $h\in\Cont(\SpZ_2)$.
Let $P^1_{\SpX} : \SpX^2\ni(x,y)\mapsto x$ be the projection onto the first component and $P^2_{\SpX} : \SpX^2\ni(x,y)\mapsto y$ onto the second.
For two Borel measurable functions $f,g$ on $\SpX$, we can construct the function 
\begin{align*}
  f\oplus g:\SpX^2\ni(x,y)\mapsto f(x)+g(y).
\end{align*}
For any function $f : \SpX \to \R$ we denote by $\norm{f}_{\infty} := \sup_{x\in\SpX} \abs{f(x)}$.
Let $B(x,r)\assign \{x'\in\SpX\mid d(x,x')< r\}$ denote the open ball centered at $x\in\SpX$ with radius $r > 0$. 
For normed spaces $U, V$ and a linear operator $T : U \to V$ we denote by $\opNorm{T}$ the induced operator norm. 
When $U$ and $V$ are Hilbert spaces, we denote by $\norm{T}_{\HS}$ the Hilbert--Schmidt norm. 
In this case one has $\opNorm{T} \leq \norm{T}_{\HS}$.
For positive real valued functions $A ,B: \Theta \mapsto \R_+$ defined on some space $\Theta$, we use $A(\theta) \lesssim B(\theta)$ to indicate there exists some positive constant $C$ independent of $\theta$ such that $A(\theta) \leq C\cdot B(\theta)$ for all $\theta \in \Theta$.
We mention explicitly, which parameters are not part of $\theta$ in such instances.

\subsection{Optimal transport and entropic regularization} \label{subsec:Opt_Trans}
The following proposition collects a few standard properties of entropic optimal transport. Proofs can be found, for instance, in \cite{NS2021},
see also \cite{BoSch2020,FSVATP2018,santambrogio2015optimal}.
\begin{proposition}[Entropic optimal transport] \label{prop:EOT_intro}
  For $\mu, \nu \in \prob(\SpX)$, some cost function $c \in \Cont(\SpX \times \SpX)$, and a regularization parameter $\veps > 0$, the corresponding primal entropic optimal transport problem is given by
  \begin{equation} \label{eq:def_entrOT_primal}
    I_{c}^\veps(\mu,\nu) \assign \inf\left\{\int_{\SpX^2}c(x,y)\,\diff\pi(x,y)+ \veps \KL(\pi\mid\mu\otimes\nu) ~\middle|~\pi\in\Pi(\mu,\nu)\right\}
  \end{equation}
  where
  \begin{equation} \label{eq:def_trans_plan}
    \Pi(\mu,\nu) \assign \left\{\pi\in\mathcal{P}(\SpX\times \SpX) \middle| P_{\SpX\#}^1\pi = \mu , \ =P_{\SpX\#}^2\pi = \nu\right\}
  \end{equation}
  is the set of transport plans between $\mu$ and $\nu$ (recall that $P_{\SpX}^i$, $i=1,2$, are the projections from $\SpX \times \SpX$ to the first and second coordinate).
  The dual problem is given by
  \begin{align} \label{eq:def_entrOT_dual}
    \sup \left\{ \int_{\SpX} \alpha \, \diff \mu + \int_{\SpX} \beta \, \diff \nu - \veps \int_{{\SpX}^2} \left[\exp([\alpha \oplus \beta-c]/\veps)-1\right]\,\diff \mu \otimes \nu ~\middle|~ \alpha,\beta \in \Cont({\SpX}) \right\}.
  \end{align}
  Problem \eqref{eq:def_entrOT_primal} has a unique minimizer $\pi$, maximizers in \eqref{eq:def_entrOT_dual} exist. For any pair of maximizers $(\alpha,\beta)$ of \eqref{eq:def_entrOT_dual} one has
  \begin{equation} \label{eq:EntropicPD}
    \pi=\exp([(\alpha \oplus \beta) - c]/\veps) \cdot \mu \otimes \nu
  \end{equation}
  and
  \begin{equation} \label{eq:Sinkhorn}
    \begin{aligned}
      \alpha(x) & = -\veps \log\left(\int_{\SpX} \exp([\beta(y)-c(x,y)]/\veps)\,\diff \nu(y) \right), \\
      \beta(y) & = -\veps \log\left(\int_{\SpX} \exp([\alpha(x)-c(x,y)]/\veps)\,\diff \mu(x) \right)
    \end{aligned}
  \end{equation}
  for $\mu$-almost all $x$ and $\nu$-almost all $y$. In particular, for any two dual maximizers $(\alpha_1,\beta_1)$ and $(\alpha_2,\beta_2)$ their outer sums $\alpha_i \oplus \beta_i$, $i=1,2$, are the same $(\mu\otimes\nu)$-almost everywhere. 
  Furthermore, given one solution $(\alpha, \beta)$ of \eqref{eq:def_entrOT_dual}, the set of all solutions is given by shifts $(\alpha + t, \beta - t)$ for $t \in \R$ almost everywhere.
\end{proposition}
\begin{remark} \label{rem:extended_duals}
  Equations \eqref{eq:Sinkhorn} can be evaluated at any $x, y \in \SpX$, even beyond the support of $\mu$ and $\nu$, allowing us to extend dual maximizers $(\alpha,\beta)$ to continuous functions on $\SpX$. Via \eqref{eq:Sinkhorn} the extensions inherit the modulus of continuity of $c$ (for example, the extensions of $\alpha$ and $\beta$ are Lipschitz continuous if $c$ is Lipschitz continuous). For any such extended potentials $(\alpha,\beta)$, the sum $\alpha\oplus\beta$ does not depend on the specific choice of maximizers $(\alpha,\beta)$, now everywhere on $\SpX\times\SpX$.
\end{remark}

In the specific case of `self-transport', i.e.~$\mu=\nu$, a favoured dual solution will be useful later to ensure stronger bounds on such entropic potentials:
\begin{proposition} \label{prop:self_transport}
  If $\mu = \nu$ and $c$ is symmetric, there exists a unique $\bar{\alpha}$ such that $(\bar{\alpha},\bar{\alpha})$ is a solution to \eqref{eq:def_entrOT_dual} and satisfies \eqref{eq:Sinkhorn} on the whole space $\SpX$, i.e.
  \begin{equation} \label{eq:self_Sinkhorn}
    \bar{\alpha}(x) = -\veps \log\left(\int_{\SpX} \exp([\bar{\alpha}(y)-c(x,y)]/\veps)\,\diff \nu(y) \right).
  \end{equation}
  This function is the ($\mu$-almost everywhere unique) solution to the symmetrized problem
  \begin{align} \label{eq:def_entrOT_dual_self}
    \sup \left\{ 2\int_{\SpX} \alpha \, \diff \mu - \veps \int_{{\SpX}^2} \left[\exp([\alpha \oplus \alpha-c]/\veps)-1\right]\,\diff \mu \otimes \mu ~\middle|~ \alpha\in \Cont({\SpX}) \right\}.
  \end{align}
  and we will refer to it as the optimal self-transport potential for $\mu$.
\end{proposition}
\begin{proof}
  Clearly for $\mu=\nu$, \eqref{eq:def_entrOT_dual} $\geq$ \eqref{eq:def_entrOT_dual_self}.
  We show the other inequality by construction.
  Let $(\alpha, \beta)$ be some solution to \eqref{eq:def_entrOT_dual}. 
  By symmetry $(\beta, \alpha)$ is also a solution and therefore, there exists a constant $t\in\R$ such that $\beta = \alpha - t$ (at least $\mu$-almost everywhere, by Proposition \ref{prop:EOT_intro}). 
  It then follows that $(\alpha - t/2,\alpha - t/2)$ is also a solution, yielding the corresponding $\bar{\alpha}$, $\mu$-almost everywhere, which can then be extended to the whole space using \eqref{eq:self_Sinkhorn}. 
  Uniqueness of a solution $\bar{\alpha}$ of \eqref{eq:def_entrOT_dual_self} $\mu$-a.e.~follows from the fact that $\bar{\alpha}\oplus\bar{\alpha}$ is unique $(\mu\otimes\mu)$-a.e.~and therefore $\bar{\alpha}(x)=\frac{1}{2}(\Bar{\alpha}\oplus\Bar{\alpha})(x,x)$ is uniquely defined for $\mu$-a.e.~$x\in\SpX$.
\end{proof}
\noindent Throughout this article we make the following assumption on the cost function $c$.
\begin{assumption} \label{ass:lip_c}
  The cost $c$ on $\SpX \times \SpX$ is symmetric, $c\geq 0$ and $c(x,x)=0$ for any $x\in\SpX$.
  Furthermore, $c$ is uniformly Lipschitz-continuous in each argument, i.e.~there is some $\Lip(c) > 0$ s.t.~for all $x, x^{\prime}, y \in \SpX$ we have
  \begin{align*}
    \abs{c(x, y) - c(x^{\prime},y)} \leq \Lip(c) \, d(x, x^{\prime}).
  \end{align*}
\end{assumption}

\noindent In this article the following function will play a fundamental role as a data-adapted smoothing kernel.
\begin{definition}[Entropic transport kernel] \label{def:entOT_kernel}
  For $\mu,\nu\in\mathcal{P}(\SpX)$, $\veps>0$, we define the entropic transport kernel from $\mu$ to $\nu$ as
  \begin{equation} \label{eq:def_etk}
    k^\veps_{\mu\nu} \assign \exp([(\alpha \oplus \beta) - c]/\veps)
  \end{equation}
  where $\alpha$ and $\beta$ are dual maximizers of \eqref{eq:def_entrOT_dual} that satisfy \eqref{eq:Sinkhorn} on the full space $\SpX \times \SpX$, which implies that $\alpha \oplus \beta$ is unique, and therefore that $k^\veps_{\mu\nu}$ is well-defined and lies in $\Cont(\SpX \times \SpX)$.
\end{definition}

Setting $\veps=0$ in \eqref{eq:def_entrOT_primal}, one obtains the \emph{unregularized} optimal transport problem. 
In this setting, minimal $\pi$ still exist by standard compactness continuity arguments, although they are no longer necessarily unique.
By setting $c(x,y) \assign d(x,y)^p$, this problem induces the celebrated Wasserstein distance on $\prob(\SpX)$.
\begin{proposition}[$p$-Wasserstein metric]
  For $p \in [1,\infty)$, the $p$-Wasserstein distance on $\prob(\SpX)$ between $\mu,\mu'\in\prob(\SpX)$ is given by
    \begin{equation} \label{eq:def_pWassMetric}
      W_p(\mu,\mu'):= \left(\inf_{\pi\in\Pi(\mu,\mu')}\int_{\SpX^2}d(x,x')^p \,\diff\pi(x,x')\right)^{\frac{1}{p}}.
    \end{equation}
    $W_p$ metrizes the weak* topology on $\prob(\SpX)$.
\end{proposition}
A proof is given in \cite[Chapter 5]{santambrogio2015optimal} (recall that $\SpX$ is compact).

\subsection{Transfer operators} \label{sec:transfer}
\begin{definition} \label{def:transfer}
  Let $\mu, \nu \in \prob(\SpX)$. We call a linear map $T:\Leb^1(\mu)\to\Leb^1(\nu)$ \emph{transfer operator} if it preserves non-negativity and the mass of non-negative functions, i.e.~if for $u \in \Leb^1(\mu)$ with $u \geq 0$ one has
  \begin{equation} \label{eq:def_transfer}
    \int_{\SpX} Tu\,\diff \nu = \int_{\SpX} u\,\diff \mu \quad \tn{and} \quad Tu \geq 0
  \end{equation}
  and so in particular $T$ maps probability densities in $\Leb^1(\mu)$ to probability densities in $\Leb^1(\nu)$.
\end{definition}

In this article, we are interested in transfer operators induced by (not necessarily optimal) transport plans $\pi\in\Pi(\mu,\nu)$.
For the special case $\mu=\nu$, a plan $\pi$ can be interpreted as encoding the dynamic of a time-homogeneous Markov chain, and in this case, $T$ is sometimes called a \emph{Markov operator} \cite{eisner2015operator}.

\begin{proposition}\label{prop:inducedTO}
  A transport plan $\pi\in\Pi(\mu,\nu)$ induces a linear operator $T : \Leb^1(\mu) \to \Leb^1(\nu) $
  that is characterized by the relation
  \begin{equation} \label{eq:def_induced_op}
    \int_\SpX(T u)(y)v(y)\,\diff\nu(y)=\int_{\SpX^2} u(x)v(y)\,\diff\pi(x,y)
  \end{equation}
  for any $u \in \Leb^1(\mu)$, $v\in\Leb^\infty(\nu)$.
  $T$ is a transfer operator, $T\mathds1_\mu=\mathds1_\nu$ (where $\mathds1_\mu$ and $\mathds1_\nu$ denote the functions that are 1 $\mu$- and $\nu$-a.e. respectively), and in fact $T$ can be restricted to a bounded linear operator $\Leb^p(\mu) \to \Leb^p(\nu)$ for any $p \in [1,\infty]$ with operator norm $1$.
\end{proposition}
\begin{proof}
  Denoting by $(\pi(\cdot|y))_{y \in \SpX}$ the disintegration of $\pi$ with respect to its second marginal $\nu$, one obtains from the definition \eqref{eq:def_induced_op} that for all $u\in\Leb^1(\mu)$,
  \begin{equation} \label{eq:induced_op_disintegration}
    (T u)(y)=\int_\SpX u(x)\,\diff\pi(x|y)\quad\text{for $\nu$-a.e.~$y\in\SpX$}.
  \end{equation}
  This implies \eqref{eq:def_transfer} and $T\mathds1_\mu=\mathds1_\nu$. 
  Combining \eqref{eq:induced_op_disintegration} with Jensen's inequality yields that $\norm{T u}_{\Leb^p(\nu)}\leq \norm{ u}_{\Leb^p(\mu)}$ for all $p \in [1,\infty]$, so the operator norm of $T$ is bounded by $1$, and the case $T\mathds1_\mu=\mathds1_\nu$ shows that the norm is in fact equal to $1$.
\end{proof}
In the following we will merely consider the case $p=2$, since we are primarily interested in spectral analysis on Hilbert spaces.
As discussed in Section \ref{sec:motivation} the disintegration $(\pi(\cdot|x))_{x \in \SpX}$ of $\pi$ with respect to its first marginal $\mu$ can be interpreted as the transition kernel $(\kappa_x)_{x \in \SpX}$. It is well-defined for $\mu$-almost all $x \in \SpX$, which is sufficient if the law of $x$ is assumed to be $\mu$.
In this article we will frequently use transfer operators induced by optimal entropic transport plans.
\begin{definition} \label{def:Gmne}
  For $\mu, \nu \in \prob(\SpX)$, $\veps>0$, let $\pi$ be the unique entropic optimal transport plan in \eqref{eq:def_entrOT_primal}. Then we denote the operator induced by $\pi$ according to Proposition \ref{prop:inducedTO} as $G_{\mu\nu}^\veps$.
  By \eqref{eq:EntropicPD} and \eqref{eq:def_etk}, the disintegration of $\pi$ with respect to the $\nu$-marginal at $y$ is given by $\pi(\cdot|y)=\kernel{\mu\nu}{\veps}(\cdot,y) \cdot \mu$ and therefore by \eqref{eq:induced_op_disintegration} one has
  \begin{equation} \label{eq:def_Geps}
    G_{\mu\nu}^\veps:
    \Leb^2(\mu) \rightarrow \Leb^2(\nu),
    \quad
    u \mapsto\int_{\SpX}u(x)\kernel{\mu\nu}{\veps}(x,\cdot)\,\diff\mu(x).
  \end{equation}
\end{definition}

An important class of operators is that of Hilbert--Schmidt operators, which are characterized by the following proposition.
\begin{proposition}\protect{\cite[Proposition 9.6 and 9.7] {taylor2006measure}} \label{prop:induced_op_HS}
  An operator $T : \Leb^2(\mu)\to\Leb^2(\nu)$ is Hilbert--Schmidt if and only if there exists an integration kernel $t\in\Leb^2(\mu\otimes\nu)$ such that
  \begin{equation}
    \label{eq:TByKernel}
    \int_\SpX(T u)(y)v(y)\,\diff\nu(y)=\int_{\SpX^2} u(x)v(y)\,t(x,y)\,\diff\mu(x)\,\diff \nu(y)
  \end{equation}
  for any $u \in \Leb^2(\mu)$, $v \in \Leb^2(\nu)$.
  In that case, $\norm{T}_{\HS}=\norm{t}_{\Leb^2(\mu\otimes\nu)}$ and in particular $T$ is a compact operator.
\end{proposition}
We observe that an operator $T$ induced by a transport plan $\pi$ is Hilbert--Schmidt if and only if $\pi=t \cdot \mu\otimes\nu$ for some $t\in\Leb^2(\mu\otimes\nu)$. In this case $t$ is the integration kernel of $T$.

\section{Entropic regularization of transfer operators} \label{sec:Quant_trans_op}
\subsection{Problem statement and definitions} \label{subsec:main_defs}
\begin{figure}[]
    \centering
    \def\arraystretch{1.3}
    \begin{tabular}{c|c|c|c|c}
        operator
         & spaces 
         & kernel
         & measure
         & references
         \\
         \hline
         $T$ 
         & $\Leb^2(\mu) \rightarrow \Leb^2(\nu)$ 
         &
         & $\pi$ 
         &
         \\
         $T^{\veps}$ 
         & $\Leb^2(\mu) \rightarrow \Leb^2(\nu)$ 
         & $t^{\veps} = k_{\mu\mu}^{\veps} : \pi : k_{\nu\nu}^{\veps}$  
         & $t^{\veps} \mu \otimes \nu$
         & \eqref{eq:TEps}, \eqref{eq:TEpsKernel}
         \\
         $T^{C,\veps}_N$
         & $\Leb^2(\mu) \rightarrow \Leb^2(\nu)$ 
         & $t^{C,\veps}_N = k_{\mu\mu}^{\veps} : \pi^N : k_{\nu\nu}^{\veps}$
         & $t^{C,\veps}_N \mu \otimes \nu$
         & \eqref{eq:def_kcontpidisc_op}
         \\
         $T^{B,\veps}_N$
         & $\Leb^2(\mu) \rightarrow \Leb^2(\nu)$ 
         & $t^{\veps}_N = k_{\mu_N\mu_N}^{\veps} : \pi^N : k_{\nu_N\nu_N}^{\veps}$
         & $t^{\veps}_N \mu \otimes \nu$
         & \eqref{eq:def_mixed_op}
         \\
         $T^{A,\veps}_N$
         & $\Leb^2(\mu) \rightarrow \Leb^2(\nu)$ 
         & $t^{A,\veps}_N$
         & $t^{A,\veps}_N \mu \otimes \nu$
         & \eqref{eq:def_T_extension}, \eqref{eq:tNAeps}
         \\
         $T^{\veps}_N$
         & $\Leb^2(\mu_N) \rightarrow \Leb^2(\nu_N)$ 
         & $t^{\veps}_N = k_{\mu_N\mu_N}^{\veps} : \pi^N : k_{\nu_N\nu_N}^{\veps}$
         & $t^{\veps}_N \mu_N \otimes \nu_N$
         & \eqref{eq:TEpsN}, \eqref{eq:TEpsKernelN}
         \\
         $T_N$
         & $\Leb^2(\mu_N) \rightarrow \Leb^2(\nu_N)$ 
         &
         & $\pi^N = \frac{1}{N} \sum_{i=1}^N \delta_{(x_i, y_i)}$
         & \eqref{eq:piN}
    \end{tabular}

    \vspace{4mm}

    \begin{tikzpicture}[scale=2.7]
		\node (a) at (0,0) {$T$};
		\node (b) at (1,0) {$T^{\veps}$};
		\node (c) at (2,0) {$T^{C,\veps}_{N}$};
		\node (d) at (3,0) {$T^{B,\veps}_{N}$};
		\node (e) at (4,0) {$T^{A,\veps}_{N}$};
		\node (f) at (5,0) {$T^{\veps}_{N}$};
		
		\draw[] (a) edge[bend left] node[midway, above] {\Cref{prop:Teps2T}} (b);
        \draw[] (b) edge[bend left] node[midway, above] {\Cref{prop:prob_var_bound}} (c);
        \draw[] (c) edge[bend left] node[midway, above] {\Cref{prop:prob_bias_bound}} (d);
        \draw[] (d) edge[bend left] node[midway, above] {\Cref{prop:cross_operators}} (e);
        \draw[] (e) edge[bend left] node[midway, above, yshift=-2pt] {Prop. \ref{prop:extension_eigenpairs}, \ref{prop:extension_svd}} (f);

        \draw[decorate, decoration={brace, mirror, amplitude=5pt, raise=3pt}] 
			(.05,-0.1) -- (1-.05,-0.1) 
			node[midway, below, yshift=-6pt]{\Cref{subsec:cv_Teps2T}};
        \draw[decorate, decoration={brace, mirror, amplitude=5pt, raise=3pt}] 
			(1+.05,-0.1) -- (4-.05,-0.1) 
			node[midway, below, yshift=-6pt]{\Cref{subsec:quant_cvg}};
        \draw[decorate, decoration={brace, mirror, amplitude=5pt, raise=3pt}] 
			(4+.05,-0.1) -- (5-.05,-0.1) 
			node[midway, below, yshift=-6pt]{\Cref{subsec:SpectralConvergence}};
        
	\end{tikzpicture}

    \caption{Overview of operators defined in this paper. The table summarizes the different operators defined in this paper and the graph below lists the results on their relations.}
    \label{fig:operator_overview}
\end{figure}

Let $T : \Leb^2(\mu)\to\Leb^2(\nu)$ be a transfer operator induced by some plan $\pi \in \Pi(\mu,\nu)$ (see Proposition \ref{prop:inducedTO}).
In this section we will introduce a compact approximation $T^\veps$ of $T$, a discrete approximation $T_N^\veps$ of $T^{\veps}$ based on empirical data, and some auxiliary definitions to examine the convergence of (an extension of) $T_N^\veps$ towards $T^\veps$.

\begin{definition}[Regularized transfer operator] \label{def:entr_trans_basic}
  For some regularization parameter $\veps>0$, we introduce the entropic regularization $T^\veps : \Leb^2(\mu)\to \Leb^2(\nu)$ of the operator $T$ as
  \begin{align} \label{eq:TEps}
    T^\veps \assign  G_{\nu\nu}^\veps \circ T \circ G_{\mu\mu}^\veps.
  \end{align}
\end{definition}
As a composition of three transfer operators, $T^\veps$ is itself a transfer operator.
$T^\veps$ is compact since $G_{\nu\nu}^\veps$ is compact. Borrowing intuition from \cite{EntropicTransfer22}, the operators $G_{\mu\mu}^\veps$ and $G_{\nu\nu}^\veps$ introduce blur at a length scale $\sqrt{\veps}$ and thus $T^\veps$ can be thought of as a compact approximation of $T$ that preserves dynamic features on a length scale above $\sqrt{\veps}$. In the following we will analyse how $T^\veps$ can be approximated from discrete data.
\begin{remark} \label{rem:single_smoothing}
  In \cite{EntropicTransfer22} operators of the form $G_{\mu\nu}^\veps \circ T$, i.e.~with a single blurring step, were considered for deterministic $T$ induced by a continuous map $F : \SpX \to \SpX$ (in this case one has $\pi=(\id,F)_{\#}\mu$).
  The most important difference in definition \eqref{eq:TEps} is a second blurring operator that acts before $T$. We will show in Section \ref{subsec:exampleNonConv} that this is necessary to approximate $T^\veps$ by discrete data in the case where $T$ is stochastic, i.e.~not induced by a deterministic map $F$.
\end{remark}

Functions of the following form will appear repeatedly in this article as integration kernels for various operators.
\begin{proposition} \label{prop:entr_trans_general}
  For some $\mu, \mu', \nu, \nu' \in \prob(\SpX)$, $\pi \in \Pi(\mu,\nu)$, introduce the function
  \begin{equation} \label{eq:def_entr_trans}
    (\kernel{\mu'\mu}{\veps}:\pi:\kernel{\nu\nu'}{\veps}):
    (x',y')
    \mapsto
    \int_{\SpX^2}
    \kernel{\mu'\mu}{\veps}(x',x)\kernel{\nu\nu'}{\veps}(y,y')
    \,\diff\pi(x,y).
  \end{equation}
  It is a non-negative, continuous function on $\SpX \times \SpX$ and defines a transport plan $(\kernel{\mu'\mu}{\veps}:\pi:\kernel{\nu\nu'}{\veps}) \cdot \mu' \otimes \nu' \in \Pi(\mu',\nu')$. 
\end{proposition}
\begin{proof}
  Non-negativity and continuity are inherited from non-negativity of functions and measures in the integral and continuity of the entropic transport kernels. 
  By \eqref{eq:Sinkhorn} we have $\int_{\SpX} \kernel{\mu^{\prime}\mu}{\veps}(x^{\prime},x) \,\diff \mu^{\prime}(x^{\prime}) = 1$ for any $x \in \SpX$ (correspondingly for other measures) and therefore
  \begin{align} \label{eq:def_entr_trans_marginals}
    \int_{\SpX} (\kernel{\mu'\mu}{\veps}:\pi:\kernel{\nu\nu'}{\veps})(x',y)\,\diff \mu'(x') & = 1, &
    \int_{\SpX} (\kernel{\mu'\mu}{\veps}:\pi:\kernel{\nu\nu'}{\veps})(x,y')\,\diff \nu'(y') & = 1
  \end{align}
  for all $x,y \in \SpX$, which implies $(\kernel{\mu'\mu}{\veps}:\pi:\kernel{\nu\nu'}{\veps}) \cdot \mu' \otimes \nu' \in \Pi(\mu',\nu')$. 
\end{proof}
\begin{proposition} \label{prop:entr_trans}
  Let $t^\veps \assign (\kernel{\mu\mu}{\veps}:\pi:\kernel{\nu\nu}{\veps})$. Then $T^\veps$ is induced by the transport plan $t^\veps \cdot (\mu \otimes \nu)$ (in the sense of \Cref{prop:inducedTO}) and the integration kernel of $T^\veps$ is given by $t^\veps$, i.e.
  \begin{equation} \label{eq:TEpsKernel}
    T^\veps : \Leb^2(\mu) \rightarrow \Leb^2(\nu),
    \quad
    u \mapsto \int_\SpX u(x)t^\veps(x,\cdot) \,\diff\mu(x).
  \end{equation}
  In particular, $T^\veps$ is Hilbert--Schmidt.
\end{proposition}
\begin{proof}
  The form of $t^\veps$ follows directly by applying \eqref{eq:def_induced_op} and \eqref{eq:def_Geps} to definition \eqref{eq:TEps}.
  Since $t^\veps$ is bounded (it is continuous on the compact set $\SpX^2$), $T^{\veps}$ is well defined and $t^{\veps} \in \Leb^2(\mu\otimes\nu)$, so $T^{\veps}$ is indeed a Hilbert--Schmidt operator.
\end{proof}

Out next goal is to analyze the operator $T$ and the system it represents based on finite observed data, consisting of $N$ point pairs $(x_i,y_i)_{i=1}^N$, that are generated by identical and independently distributed (i.i.d.) sampling from $\pi$.
\begin{definition}[Empirical data] \label{def:piN_TN}
  For each $N \in \N$, we denote by
  \begin{align} \label{eq:piN}
    \pi_N \assign \frac{1}{N} \sum_{i=1}^N \delta_{(x_i,y_i)}
  \end{align}
  the random empirical measure supported on the $N$ i.i.d.~pairs of random variables $(x_i,y_i)_{i=1}^N$ with common law $\pi$. We also denote by $\mu_N$ and $\nu_N$ respectively the first and second marginals of $\pi_N$. Finally, we denote by $T_N$ the (random) transfer operator induced by $\pi_N$, via \Cref{prop:inducedTO}.
\end{definition}
By the law of large numbers we have weak convergence $\pi_N \rightweaks \pi$, $\mu_N \rightweaks \mu$ and $\nu_N \rightweaks \nu$ almost surely as $N \to \infty$. 
We can interpret $x_i$ as a possible state of our dynamical system at some point in time, then $y_i$ will be the state of the system after one discrete time step. 
The law of $y_i$ conditioned on $x_i=x$ is then given by $\pi(\cdot|x)$ which denotes the disintegration of $\pi$ with respect to its first marginal at $x \in \SpX$.
As mentioned below \Cref{prop:inducedTO}, we can interpret $(\pi(\cdot|x))_{x}$ as a transition kernel associated with the dynamical system.

\begin{remark}[Long trajectories in ergodic systems] \label{rem:long_trajectories_in_ergodic_systems}
  Consider the case where $\mu=\nu$ is an invariant measure of a time-homogeneous Markov chain with transition probabilities encoded by $\pi \in \Pi(\mu,\mu)$. In practice, data is often obtained by sampling one large trajectory $(z_0,z_1,\ldots,z_N)$ from this chain where the law of $z_0$ is $\mu$ and the law of $z_{t+1}$ conditioned on $z_t=z$ is given by $\pi(\cdot|z)$.
  In this case we set $(x_i,y_i)=(z_{i-1},z_i)$ for $i=1,\ldots,N$ and therefore, the different $(x_i,y_i)$ are in general not independent.

  However, for Markov chains with a unique invariant measure $\mu$ where the time-scale of relaxation of the initial distribution to the invariant measure is much smaller than $N$ (in discrete time step units), i.e.~if $T^M u \approx \ones_\mu$ for all probability densities $u \in \Leb^2(\mu)$ and some $M \ll N$ (here $T^M$ denotes the $M$-th power of $T$, see \Cref{prop:inducedTO} on how $T$ is induced by $\pi$), then $z_t$ and $z_{t+M}$ are approximately independent. So if we use pairs $(x_i,y_i)=(z_{M \cdot i-1},z_{M \cdot i})$ with a skip $M$ then the above i.i.d.-assumption is a good approximation. In fact, in this case even using all pairs will work well, since approximate independence still holds for most pairs.
\end{remark}

\begin{definition}[Empirical transfer operator and regularization] \label{def:empirical_operator}
  In analogy to \eqref{eq:TEps} we define the entropic regularization of $T_N$ as
  \begin{align} \label{eq:TEpsN}
    T_N^\veps \assign  G_{\nu_N\nu_N}^\veps \circ T_N \circ G_{\mu_N\mu_N}^\veps.
  \end{align}
\end{definition}
Similar to \eqref{eq:TEpsKernel} one finds that
\begin{equation} \label{eq:TEpsKernelN}
  T_N^\veps : 
  \Leb^2(\mu_N) \rightarrow \Leb^2(\nu_N),
  \quad
  u \mapsto \int_\SpX u(x)t_N^\veps(x,\cdot) \,\diff\mu_N(x)
  \quad \tn{for} \quad
  t_N^\veps \assign (\kernel{\mu_N\mu_N}{\veps}:\pi_N:\kernel{\nu_N\nu_N}{\veps}).
\end{equation}
As with $T^\veps$, $T_N^\veps$ is a transfer operator, induced by the plan $t_N^\veps \cdot (\mu_N \otimes \nu_N) \in \Pi(\mu_N,\nu_N)$, and Hilbert--Schmidt.
$T_N^\veps$ is an operator between two finite-dimensional spaces. 
The operators $G_{\mu_N\mu_N}^\veps$ and $G_{\nu_N\nu_N}^\veps$ can each be obtained by solving a finite-dimensional entropic optimal transport problems from the discrete measure $\mu_N$ or $\nu_N$ onto itself. Hence, $T_N^\veps$ can be studied numerically as long as $N$ is not too large.
Similar to \cite{EntropicTransfer22}, we will show that a suitable extension of $T_N^\veps$ converges to $T^\veps$ almost surely as $N\to \infty$ in Hilbert--Schmidt norm.
This implies that the numerical study of $T_N^\veps$ provides insights on the operator $T^\veps$.
Unlike \cite{EntropicTransfer22}, we will give a quantitative bound on the rate of convergence, which allows us to also study the influence of the ambient and intrinsic dimension of the data and the joint limit $N\to\infty,~\veps\to 0$.

To be able to relate $T_N^\veps$ to $T^\veps$, we extend the former from $\Leb^2(\mu_N) \to \Leb^2(\nu_N)$ to the spaces $\Leb^2(\mu) \to \Leb^2(\nu)$. The following is an adaptation of the construction in \cite[Section 4.7]{EntropicTransfer22}.
\begin{definition}[Empirical transfer operator extension]
 \label{def:T_extension}
  Let $\gamma^\mu_N\in\Pi(\mu,\mu_N)$ and $\gamma^\nu_N\in\Pi(\nu,\nu_N)$ be unregularized optimal transport plans for the quadratic cost $c=d^2$ from $\mu$ to $\mu_N$ and from $\nu$ to $\nu_N$ respectively.
  Let $T^\mu_N : \Leb^2(\mu)\to\Leb^2(\mu_N)$ and $T^\nu_N: \Leb^2(\nu)\to\Leb^2(\nu_N)$ be the corresponding induced operators (see Proposition \ref{prop:inducedTO}). We define
  $T^{A,\veps}_N: \Leb^2(\mu)\to\Leb^2(\nu)$ via
  \begin{equation} \label{eq:def_T_extension}
    T^{A,\veps}_N=(T^\nu_N)^*\circ T^\veps_N\circ T^\mu_N.
  \end{equation}
\end{definition}
The adjoint operator $(T^\nu_N)^* : \Leb^2(\nu_N) \to \Leb^2(\nu)$ coincides with the transfer operator induced by the transpose of the plan $\gamma^\nu_N$. Hence, $T^{A,\veps}_N$ is again a transfer operator.
For $u \in \Leb^2(\mu)$, $v \in \Leb^2(\nu)$ one finds that
\begin{align*}
  \int_{\SpX} v(y)\,(T^{A,\veps}_N u)(y)\,\diff \nu(y) & =
  \int_{\SpX^4} v(y)\,u(x)\,t_N^\veps(x',y')\,\diff \gamma^\mu_N(x,x')\,\diff \gamma^\nu_N(y,y') \\
  & =
  \int_{\SpX^2} v(y)\,u(x)\,t_N^{A,\veps}(x,y)\,\diff \mu(x)\,\diff \nu(y)
\end{align*}
with the integration kernel
\begin{equation} \label{eq:tNAeps}
  t_N^{A,\veps}(x,y) \assign \int_{\SpX^2} t_N^\veps(x',y')\,\diff \gamma^\mu_N(x'|x)\,\diff \gamma^\nu_N(y'|y)
\end{equation}
where $(\gamma^\mu_N(\cdot|x))_x$ denotes the disintegration of $\gamma^\mu_N$ with respect to its $\mu$ marginal and $(\gamma^\nu_N(\cdot|y))_y$ that of $\gamma^\nu_N$ correspondingly.
Under suitable conditions (see Section \ref{subsec:SpectralConvergence}), $T_N^\veps$ and its extension $T_N^{A,\veps}$ have the same non-zero singular values with a one-to-one correspondence for the related singular functions.

To control the discrepancy between $T^\veps$ and $T_N^{A,\veps}$ in Hilbert--Schmidt norm, we introduce below two additional intermediate auxiliary operators $T^{B,\veps}_N$ and $T^{C,\veps}_N$. In Section \ref{subsec:quant_cvg} we then control $\|T^{A,\veps}_N - T^{B,\veps}_N\|_{\HS}$, $\|T^{B,\veps}_N - T^{C,\veps}_N\|_{\HS}$, and $\|T^{C,\veps}_N - T^{\veps}\|_{\HS}$ in terms of $N$ and $\veps$, which yields the desired convergence $T^{A,\veps}_N \to T^\veps$. An overview on all introduced operators and their relations is given in Figure \ref{fig:operator_overview}.

\begin{definition}
  Using that $t^\veps_N = (\kernel{\mu_N\mu_N}{\veps}:\pi_N:\kernel{\nu_N\nu_N}{\veps})$ from \eqref{eq:TEpsKernelN} is a continuous function on the whole space $\SpX \times \SpX$
  (see \Cref{prop:entr_trans_general}), define the linear operator $T^{B,\veps}_N$ as
  \begin{equation} \label{eq:def_mixed_op}
    T^{B,\veps}_N:
    \Leb^2(\mu) \rightarrow \Leb^2(\nu),
    \quad
    u\mapsto \int_{\SpX^2}u(x)t^\veps_N(x,.)\,\diff\mu(x).
  \end{equation}
\end{definition}
$T^{B,\veps}_N$ can be interpreted as the operator induced by the measure $\pi^{B,\veps}_N\assign t^\veps_N \cdot (\mu\otimes\nu)$ but in general $\pi^{B,\veps}_N$ is not a probability measure and in particular not an element of $\Pi(\mu,\nu)$. Therefore $T^{B,\veps}_N$ is in general not a transfer operator in the sense of Definition \ref{def:transfer}.
The fact that $t^\veps_N$ is continuous on $\SpX \times \SpX$ has an additional interesting application for out-of-sample embedding, defined in Section \ref{subsec:out_of_sampling}.

\begin{definition}
  Define the linear operator $T^{C,\veps}_N$ as
  \begin{equation} \label{eq:def_kcontpidisc_op}
    T^{C,\veps}_N:\Leb^2(\mu)\to \Leb^2(\nu)
    ,\quad 
    u\mapsto \int_{\SpX^2}u(x)t^{C,\veps}_N(x,.)\,\diff\mu(x)
    \quad \tn{for} \quad
    t^{C,\veps}_N\assign (\kernel{\mu\mu}{\veps}:\pi_N:\kernel{\nu\nu}{\veps}).
  \end{equation}
\end{definition}

\subsection{Preliminary results} \label{subsec:preliminary}
Before we can prove the main result we collect some preliminary results.
\begin{proposition}[Bound on entropic kernel and its Lipschitz constant] \label{prop:bounds_kernel}
  Let \Cref{ass:lip_c} hold.
  Let $\mu,\nu \in \prob(\SpX)$ and $\veps > 0$. Then for any $x \in \SpX$ and $y \in \spt(\nu)$ one has
  \begin{equation} \label{eq:oneside_infnorm_OTker}
    \kernel{\mu\nu}{\veps}(x,y)\leq\frac{\exp(2\Lip(c))}{\nu(B(y,\veps))}.
  \end{equation}
  For $y \in \spt(\nu)$, $\kernel{\mu\nu}{\veps}(\cdot,y)$ is Lipschitz continuous with
  \begin{equation} \label{eq:oneside_LipCons_OTker}
    \Lip (\kernel{\mu\nu}{\veps}(\cdot,y))\leq \frac{2\Lip(c) \exp(2\Lip(c))}{\veps\nu(B(y,\veps))}.
  \end{equation}
  Analogous statements for the bound if $x \in \spt(\mu)$ and Lipschitz continuity of $\kernel{\mu\nu}{\veps}(x,\cdot)$ hold.
\end{proposition}
\begin{proof}
  Take $x$ in $(\SpX,d)$, $y\in\spt(\nu)$ and $(\alpha,\beta)$ a pair of optimal entropic transport potentials for \eqref{eq:def_entrOT_dual}. 
  As noted in Remark \ref{rem:extended_duals} the Lipschitz constant for $c$ is also valid for $\beta$, since
  \begin{align*}
    \beta(y) - \beta(y^{\prime})
    =&
    -\veps 
    \log\left(
    \frac{\int_{\SpX} 
      \exp\left(\frac{\alpha(x) - c(x, y)}{\veps}\right)
      \diff x}
         {\int_{\SpX} 
           \exp\left(\frac{\alpha(x) - c(x, y)}{\veps}\right)
           \exp\left(\frac{c(x, y) - c(x, y^{\prime})}{\veps}\right)
           \diff x}
         \right)
         \\\leq&
         -\veps 
         \log\left(
         \frac{\int_{\SpX} 
           \exp\left(\frac{\alpha(x) - c(x, y)}{\veps}\right)
           \diff x}
              {\int_{\SpX} 
                \exp\left(\frac{\alpha(x) - c(x, y)}{\veps}\right)
                \exp\left(\frac{\Lip(c) d(y, y^{\prime})}{\veps}\right)
                \diff x}
              \right)
              = 
              \Lip(c) d(y, y^{\prime}).
  \end{align*}
  Then, 
  \begin{align*}
    \exp\left(-\frac{\alpha(x)}{\veps}\right)
    &=
    \int_\SpX \exp\left(\frac{-c(x,y')+\beta(y')}{\veps}\right)\diff\nu(y')
    \\&\geq
    \int_{B(y,\veps)}
    \exp\left(\frac{-c(x,y)+\beta(y)}{\veps}\right)
    \exp\left(\frac{c(x,y)-c(x,y')+\beta(y')-\beta(y)}{\veps}\right)
    \,\diff\nu(y')
    \\&\geq
    \exp\left(\frac{-c(x,y)+\beta(y)}{\veps}\right)\int_{B(y,\veps)}\exp\left(-\frac{2 \Lip(c)\,d(y,y')}{\veps}\right)\diff\nu(y')
    \\&\geq
    \exp\left(\frac{-c(x,y)+\beta(y)}{\veps}\right) \exp(-2 \Lip(c))\,\nu(B(y,\veps)).
  \end{align*}
  Therefore, 
  \begin{align*}
    \kernel{\mu\nu}{\veps}(x,y)
    =
    \exp\left(\frac{\alpha(x) + \beta(y) - c(x, y)}{\veps}\right)
    \leq 
    \frac{\exp(2\Lip(c))}{\nu(B(y,\veps))}.
  \end{align*}
  Note that $\exp$ restricted to $(-\infty, a)$ is Lipschitz with constant $\exp(a)$.
  Therefore we get
  \begin{align*}
    \abs{\kernel{\mu\nu}{\veps}(x,y)-\kernel{\mu\nu}{\veps}(x',y)}
    \leq&
    \sup\left\{\kernel{\mu\nu}{\veps}(\cdot,y)\right\}\frac{\abs{-c(x,y)+\alpha(x)+c(x',y)-\alpha(x')}}{\veps}
    \\\leq& 
    \frac{\exp(2\Lip(c))}{\nu(B(y,\veps))}\frac{2\Lip(c)}{\veps}\,d(x,x'). \qedhere
  \end{align*}
\end{proof}
\begin{remark} \label{rem:better_bound_selfOT}
  The following special case will be relevant later. If $\mu=\nu$, taking $x=y$ in \eqref{eq:oneside_infnorm_OTker} yields for any $x$ in $\spt(\mu)$,
  \begin{align*}
    \exp\left(\frac{\Bar{\alpha}(x)}{\veps}\right)
    \leq
    \frac{\exp(\Lip(c))}{\sqrt{\mu(B(x,\veps))}}
    \lesssim
    \frac{1}{\sqrt{\mu(B(x,\veps))}}
  \end{align*}
  where $\Bar{\alpha}$ is the optimal self-transport potential in \eqref{eq:self_Sinkhorn}. For $(x,y)\in\spt(\mu)\times\spt(\mu)$ this yields the better bound
  \begin{align*}
    \kernel{\mu\mu}{\veps}(x,y)
    \leq 
    \frac{\exp\left(2 \Lip(c) - \frac{c(x, y)}{\veps}\right)}{\mu(B(x,\veps))}
    \lesssim
    \frac{\exp\left(-\tfrac{c(x,y)}{\veps}\right)}{\mu(B(x,\veps))}.
  \end{align*}
\end{remark}

These controls on entropic transport kernels $\kernel{\mu\nu}{\veps}$ imply the following control on the convolution construction \eqref{eq:def_entr_trans}.
\begin{corollary} \label{cor:lip_transfer_kernel}
  Let $\mu, \mu', \nu, \nu' \in \prob(\SpX)$, $\pi \in \Pi(\mu,\nu)$. The function $(\kernel{\mu'\mu}{\veps}:\pi:\kernel{\nu\nu'}{\veps})$ is bounded and Lipschitz-continuous with
  \begin{align} \label{eq:infnorm_transfer_kernel}
    \norm{(\kernel{\mu'\mu}{\veps}:\pi:\kernel{\nu\nu'}{\veps})}_\infty 
    &\lesssim
    \frac{1}{\max\{ \inf_{x\in\spt(\mu)}\mu(B(x,\veps)), \inf_{y\in\spt(\nu)} \nu(B(y,\veps))\}}, 
    \\
    \label{eq:lip_transfer_kernel}
    \Lip((\kernel{\mu'\mu}{\veps}:\pi:\kernel{\nu\nu'}{\veps})(\cdot, y))
    &\lesssim 
    \frac{1}{\veps \cdot \inf_{x\in\spt(\mu)}\mu(B(x,\veps))}
  \end{align}
  and a corresponding statement holds when exchanging the marginals.
  The multiplicative constants in both inequalities only depend on $\Lip(c)$.
\end{corollary}
\begin{proof}
  For simplicity write $t^\veps \assign (\kernel{\mu'\mu}{\veps}:\pi:\kernel{\nu\nu'}{\veps})$. Then, using \Cref{prop:bounds_kernel}, for $x',x'',y'\in\SpX$,
  \begin{equation*}
    \begin{split}
      t^\veps(x',y')
      =&
      \int_{\SpX^2}
      \kernel{\mu'\mu}{\veps}(x',x)\kernel{\nu\nu'}{\veps}(y,y')
      \,\diff\pi(x,y)
      \\\leq&
      \sup_{x\in\spt(\mu)} \kernel{\mu'\mu}{\veps}(x',x)
      \underbrace{
        \int_{\SpX^2}
        \kernel{\nu\nu'}{\veps}(y,y') 
        \,\diff\pi(x,y)  
      }_{=1}
      \stackrel{\eqref{eq:oneside_infnorm_OTker}}{\lesssim}
      \frac{1}{\inf_{x\in\spt(\mu)}\mu(B(x,\veps))}
    \end{split}
  \end{equation*}
  where we used that $\int_{\SpX} \kernel{\nu\nu'}{\veps}(y,y')\,\diff\nu(y)=1$ (by \eqref{eq:Sinkhorn}).
  Combining this with the same calculation with the roles of $\mu$ and $\nu$ swapped gives \eqref{eq:infnorm_transfer_kernel}.
  Similarly, we obtain \eqref{eq:lip_transfer_kernel}:
  \begin{align*}
      \abs{t^\veps(x',y')-t^\veps(x'',y')}
      \leq&
      \int_{\SpX^2}
      \abs{\kernel{\mu'\mu}{\veps}(x',x)-\kernel{\mu'\mu}{\veps}(x'',x)}\kernel{\nu\nu'}{\veps}(y,y')
      \,\diff\pi(x,y)
      \\\leq& 
      \sup_{x\in\spt(\mu)}
      \abs{\kernel{\mu'\mu}{\veps}(x',x)-\kernel{\mu'\mu}{\veps}(x'',x)} 
      \int_{\SpX^2}
      \kernel{\nu\nu'}{\veps}(y,y')
      \,\diff\pi(x,y)
      \\\lesssim& 
      \frac{d(x',x'')}{\veps\inf_{x\in\spt(\mu)}\mu(B(x,\veps))}.
  \end{align*}
\end{proof}
To control the terms depending on the mass of small balls in \eqref{eq:infnorm_transfer_kernel}, \eqref{eq:lip_transfer_kernel}, we make the following assumption on the marginal measures $\mu$ and $\nu$.
\begin{assumption} \label{ass:min_ball_mass}
  There exist constants $C_\mu$, $\mc{D}_\mu, \delta_\mu > 0$ such that for any $\delta \in (0,\delta_\mu]$ and $x\in\spt(\mu)$
  \begin{align*}
    \mu(B(x,\delta))\geq C_\mu\delta^{\mathcal{D}_\mu}.
  \end{align*}
  For simplicity we assume $C_{\mu} \leq 1$ and $\mathcal{D}_{\mu} \leq d$.
  Equivalent constants $C_\nu$, $\mc{D}_\nu, \delta_\nu$ exist for $\nu$.
\end{assumption}
The values $\mc{D}_\mu$ and $\mc{D}_\nu$ can be interpreted as effective dimensions of $\mu$ and $\nu$ which may be smaller than the dimension of $\SpX$.
\Cref{ass:min_ball_mass} transfers to the empirical approximations $\mu_N$, $\nu_N$ with high probability, as follows.
\begin{lemma} \label{lem:empirical_lip_bounds}
  Let $\mu_N \assign \frac{1}{N} \sum_{i=1}^N \delta_{x_i}$ be a random empirical approximation of $\mu$, generated from $N$ independent random variables $(x_i)_{i=1}^N$ with common law $\mu$. 
  Then for any $x\in\SpX$,
  \begin{equation*}
    \mathbb{P}\left(\mu_N(B(x,\veps))\leq \inf_{x'\in\spt(\mu)} \frac{\mu(B(x',\veps))}{2}\right) 
    \ \leq\ 
    \mathbb{P}\left(\mu_N(B(x,\veps))\leq \frac{\mu(B(x,\veps))}{2}\right) 
    \ \leq\ 
    \exp\left(-\frac{N}{2}\mu(B(x,\veps))^2\right).
  \end{equation*}
\end{lemma}
\begin{proof}
  The first inequality is immediate.
  Denoting $X_i = -\mathds{1}_{x_i\in B(x,\veps)}$, we have $\mu_N(B(x,\veps)) = -\frac{1}{N}\sum_{i=1}^N X_i$ and 
  \begin{align*}
    \mathbb{E}(\mu_N(B(x,\veps)))
    =
    -\mathbb{E}(X_1) 
    =
    \mu(B(x,\veps)).
  \end{align*}
  The result follows immediately by applying Hoeffding's inequality (see \Cref{thm:Hoeffding} right below) with $s = \frac{1}{2}\mu(B(x,\veps))$.
\end{proof}

\begin{theorem}[Hoeffding \protect{\cite[Theorem 2]{hoeffding1963probability}}] \label{thm:Hoeffding}
  If $X_1, X_2, ..., X_N$ are independent random variables with $a_i \leq X_i \leq b_i$ for $i = 1, ..., N$ and $\Bar{X}\assign\frac{1}{N}\sum_{i=1}^N X_i$, then for $s>0$
  \begin{align*}
    \mathbb{P}(\Bar{X}-\mathbb{E}(\Bar{X})\geq s) 
    \ \leq\ 
    \exp\left(-\frac{2N^2 s^2}{\sum_{i=1}^N(b_i-a_i)^2}\right).
  \end{align*}
\end{theorem}

This allows then to give adjusted versions of \Cref{prop:bounds_kernel} and \Cref{cor:lip_transfer_kernel} for their empirical approximations.
\begin{corollary} \label{cor:empirical_lip_bounds}
  Let Assumptions \ref{ass:lip_c} and \ref{ass:min_ball_mass} hold. 
  Let $N \in \N, \veps \in (0, \min\{\delta_{\mu}, \delta_{\nu}\}]$ and
    \begin{align*}
      \tau_{\mu} 
      := 
      N \exp\left(-\frac{N}{2} \left(C_{\mu} \veps^{\mathcal{D}_{\mu}}\right)^2\right),
      \qquad
      \tau_{\nu} 
      := 
      N \exp\left(-\frac{N}{2} \left(C_{\nu} \veps^{\mathcal{D}_{\nu}}\right)^2\right)
    \end{align*}
    The following statements hold.
    \begin{enumerate}
    \item With probability at least $1- \tau_{\nu}$, for $y\in\spt(\nu_N)$
      \begin{align*}
        \norm{\kernel{\mu_N\nu_N}{\veps}(\cdot,y)}_\infty
        \lesssim
        \frac{1}{\veps^{\mathcal{D}_\nu}}
        ,\qquad 
        \Lip (\kernel{\mu_N\nu_N}{\veps}(\cdot,y))
        \lesssim
        \frac{1}{\veps^{1+\mathcal{D}_\nu}}
      \end{align*}
      where the constants depend only on $C_{\nu}$ and $\Lip(c)$.
    \item For some $\mu', \nu' \in \prob(\SpX)$, denoting $t^\veps_N\assign(\kernel{\mu'\mu_N}{\veps}:\pi_N:\kernel{\nu_N\nu'}{\veps})$, one has with probability at least $1-\tau_{\mu}$
      \begin{align*}
        \norm{t^\veps_N}_\infty
        \lesssim
        \frac{1}{\veps^{\mathcal{D}_\mu}}
        ,\qquad
        \Lip(t^\veps_N(\cdot,y))
        \lesssim
        \frac{1}{\veps^{1+\mathcal{D}_\mu}}
      \end{align*}
      where the constants depend only on $C_{\mu}$ and $\Lip(c)$.
    \end{enumerate}
    Analogous statements with swapped marginals hold and with probability at least $1 - (\tau_{\mu} + \tau_{\nu})$ all inequalities above hold simultaneously.
\end{corollary}

\begin{proof}
  Given the assumptions it follows from \Cref{lem:empirical_lip_bounds} that for any $x \in \SpX$
  \begin{align*}
    \mathbb{P}\left(\mu_N(B(x,\veps)) > \inf_{x'\in\spt(\mu)} \frac{\mu(B(x',\veps))}{2}\right)
    \geq 
    1-\exp\left(-\frac{N}{2} \mu(B(x,\veps))^2\right)
    \geq 
    1 - \frac{\tau_{\mu}}{N}
  \end{align*}
  and therefore (since $\abs{\spt(\mu_N)} = N$)
  \begin{equation*}
    \mathbb{P}\left(\min_{x'\in\spt(\mu_N)}\mu_N(B(x',\veps))> \inf_{x'\in\spt(\mu)}\frac{\mu(B(x',\veps))}{2}\right)\geq 1-N\frac{\tau_{\mu}}{N}=1-\tau_{\mu}.
  \end{equation*}
  The corresponding statement for $\nu, \nu_N$ follows in the same way. 
  Consequently, with probability $\geq 1 - \tau_{\mu}$ (resp.~$1 - \tau_{\nu}$), we can replace in \Cref{prop:bounds_kernel} and \Cref{cor:lip_transfer_kernel} for $\mu_N$, $\nu_N$ and $\pi_N\in\Pi(\mu_N,\nu_N)$, the masses $\mu_N(B(x,\veps))$ and $\nu_N(B(x,\veps))$ by the corresponding ones for $\mu/2$ and $\nu/2$ and then apply \Cref{ass:min_ball_mass}. For example, for $y\in\spt(\nu_N)$, one finds with probability at least $1 - \tau_{\nu}$
  \begin{equation*}
    \Lip (\kernel{\mu_N\nu_N}{\veps}(\cdot,y))
    \stackrel{\eqref{eq:oneside_LipCons_OTker}}{\leq}
    \frac{2\Lip(c) \exp(2\Lip(c))}
         {\min_{y\in \spt(\nu_N)}\veps\nu_N(B(y,\veps))}
         \lesssim
         \frac{1}
              {\inf_{y\in \spt(\nu)} \veps\nu(B(y,\veps))}
              \stackrel{\ref{ass:min_ball_mass}}{\lesssim}
              \frac{1}
                   {\veps^{1+\mathcal{D}_\nu}}. \qedhere
  \end{equation*}
\end{proof}

\subsection[Quantitative convergence analysis T\textasciicircum{}eps\_N to T\textasciicircum{}eps]{Quantitative convergence analysis $T^\veps_N \to T^\veps$} \label{subsec:quant_cvg}
We begin this section by bounding the discrepancy in Hilbert--Schmidt norm between the extension $T^{A,\veps}_N$ and the intermediate auxiliary operator $T^{B,\veps}_N$.
\begin{theorem} \label{prop:cross_operators}
  Let Assumptions \ref{ass:lip_c} and \ref{ass:min_ball_mass} hold. Let $N \in \N$, $\veps \in (0,\min\{\delta_\mu,\delta_\nu,1\}]$, and $\tau<1$ such that
    \begin{align*}
      \tau 
      \geq 
      2 N \exp\left(-\frac{N}{2} \min\left\{\left(C_\mu \veps^{\mathcal{D}_\mu}\right)^2,\left(C_\nu \veps^{\mathcal{D}_\nu}\right)^2\right\}\right).
    \end{align*}
    Then with a probability of at least $1-\tau$ we have
    \begin{equation} \label{eq:piecewise_constant_cv}
      \norm{T^{A,\veps}_N - T^{B,\veps}_N}_{\HS}
      =
      \left\| t_N^{A,\veps}-t_N^\veps \right\|_{\Leb^2(\mu\otimes\nu)}
      \lesssim 
      \veps^{-1-\max\{\mathcal{D}_\mu,\mathcal{D}_\nu\}}\sqrt{W_2^2(\mu_N,\mu)+W_2^2(\nu_N,\nu)}.
    \end{equation}
\end{theorem}
\begin{proof}
  Both $T^{A,\veps}_N$ and $T^{B,\veps}_N$ can be expressed by integral kernels, see \eqref{eq:tNAeps} and \eqref{eq:def_mixed_op}. Therefore the Hilbert--Schmidt norm of their difference is given by the $\Leb^2(\mu \otimes \nu)$-norm of the difference between their kernels.
  \begin{align*}
    \norm{T^{A,\veps}_N -T^{B,\veps}_N}^2_{\HS}
    &=
    \left\| t_N^{A,\veps}-t_N^\veps \right\|^2_{\Leb^2(\mu\otimes\nu)} 
    \\&=
    \int_{\SpX^2} \left( \int_{\SpX^2} t_N^\veps(x',y')\,\diff \gamma^\mu_N(x'|x)\,\diff \gamma^\nu_N(y'|y)-t^\veps_N(x,y)\right)^2 \diff \mu(x)\,\diff \nu(y) 
    \\&\leq 
    \int_{\SpX^4} \left(t^\veps_N(x',y') - t_N^\veps(x,y)\right)^2
    \diff \gamma^\mu_N(x'|x)\,\diff \gamma^\nu_N(y'|y)\, \diff \mu(x)\, \diff \nu(y)
    \\&\leq 
    \int_{\SpX^4} (\Lip(t^{\veps}_N))^2 \left(d^2(x, x') + d^2(y, y')\right)
    \diff \gamma^\mu_N(x'|x)\,\diff \gamma^\nu_N(y'|y)\, \diff \mu(x)\, \diff \nu(y)
    \\&=
    (\Lip(t_N^\veps))^2 \left(W_2^2(\mu_N,\mu)+W_2^2(\nu_N,\nu)\right)
  \end{align*}
  where $\Lip(t_N^{\veps})$ is a Lipschitz constant with respect to the $2$-product metric. 
  The last equality follows from $\gamma^{\mu}_N, \gamma^{\nu}_N$ being optimal transport plans. 
  We can use the marginal Lipschitz constants from \Cref{cor:empirical_lip_bounds} to construct this. Since by assumption $\tau \geq \tau_{\mu} + \tau_{\nu}$, we get with a probability of at least $1 - \tau$
  \begin{align*}
    \abs{t^{\veps}_N(x', y') - t^{\veps}_N(x, y)}
    &\leq
    \abs{t^{\veps}_N(x', y') - t^{\veps}_N(x, y')}
    + \abs{t^{\veps}_N(x, y') - t^{\veps}_N(x, y)}
    \\&\leq
    \Lip(t^{\veps}_N(\cdot,y'))d(x, x')
    + \Lip(t^{\veps}_N(x,\cdot))d(y, y')
    \\&\leq
    \underbrace{\max\{\Lip(t^{\veps}_N(\cdot,y')), \Lip(t^{\veps}_N(x,\cdot))\}}_{=: \Lip(t^{\veps}_N)}
    \sqrt{d^2(x,x') + d^2(y,y')}
    .\qedhere
  \end{align*}
\end{proof}

The Hilbert--Schmidt distance between the two auxiliary operators $T^{B,\veps}_N$ and $T^{C,\veps}_N$ can be bounded using sample complexity estimates for entropic dual potentials. These estimates require higher differentiability for the cost used in transport, which we now introduce.

\begin{assumption} \label{ass:Cs+1_c}
  $\SpX$ is a compact subset of $\R^d$ with Lipschitz boundary and the cost $c$ is in $\mc{C}^{s+1}(\SpX\times\SpX)$ for some $s>d/2$.
\end{assumption}

\begin{theorem} \label{prop:prob_bias_bound}
  Let Assumptions \ref{ass:lip_c}, \ref{ass:min_ball_mass} and \ref{ass:Cs+1_c} hold and assume $\veps \leq 1$.
  Then there exist positive constants $K$, $C$, $L$ 
  (only depending on $c$ and $\SpX$) 
  such that for any $\tau \in (0,1)$, $\veps>0$, 
  and $N$ sufficiently large to satisfy
  \begin{align*}
    \frac{\log(3/\tau)}{\veps^{ d/2}\sqrt{N}}\left(4L\exp(C/\veps) + 2\sqrt{K}\right) \leq 1
  \end{align*}
  we have with a probability of at least $1-4\tau$ 
  \begin{align*} \label{eq:bias_bound}
    \norm{T^{B,\veps}_N-T^{C,\veps}_N}_{\HS}
    \leq
    \norm{t^{\veps}_N-t^{C,\veps}_N}_\infty
    \lesssim
    \frac{\veps^{-(d/2 + \mathcal{D}_{\mu} + \mathcal{D}_{\nu})}}{\sqrt{N}} 
    \exp\left(\frac{C}{\veps}\right)
    \log\left(\frac{3}{\tau}\right).
  \end{align*}
\end{theorem}

\begin{proof}
  Similar to before, both $T^{B,\veps}_N$ and $T^{C,\veps}_N$ are kernel operators. Therefore
  \begin{equation*}
    \norm{T^{B,\veps}_N-T^{C,\veps}_N}^2_{\HS}
    =
    \int_{\SpX^2}
    \abs{t^\veps_N(x,y)-t^{C,\veps}_N(x,y)}^2
    \diff\mu(x)\diff\nu(y)
    \leq
    \norm{t^{\veps}_N - t^{C,\veps}_N}_{\infty}^2,
  \end{equation*}
  showing the first inequality.
  Recall the kernel definitions \eqref{eq:def_mixed_op} and \eqref{eq:def_kcontpidisc_op}. For $(x,y)\in\SpX^2$,
  \begin{align*}
    \abs{t^\veps_N(x,y)-t^{C,\veps}_N(x,y)}
    \leq&
    \int_{\SpX} 
    \abs{
      \kernel{\mu_N\mu_N}{\veps}(x,x') \kernel{\nu_N\nu_N}{\veps}(y,y') 
      - \kernel{\mu\mu}{\veps}(x,x') \kernel{\nu\nu}{\veps}(y,y')
    } \diff \pi_N(x', y')
    \\\leq &
    \int_{\SpX} \kernel{\mu_N\mu_N}{\veps}(x,x')\abs{\kernel{\nu_N\nu_N}{\veps}(y,y')-\kernel{\nu\nu}{\veps}(y,y')}\diff\pi_N(x',y')
    \\&+\int_{\SpX} \kernel{\nu\nu}{\veps}(y,y')\abs{\kernel{\mu_N\mu_N}{\veps}(x,x')-\kernel{\mu\mu}{\veps}(x,x')}\diff\pi_N(x',y')
    \\\leq&
    \sup_{y'\in\spt(\nu_N)}\norm{\kernel{\nu_N\nu_N}{\veps}(\cdot,y')-\kernel{\nu\nu}{\veps}(\cdot,y')}_\infty
    \\&+\sup_{\substack{x'\in\spt(\mu_N)\\y'\in\spt(\nu_N)}}\norm{\kernel{\nu\nu}{\veps}(\cdot,y')}_\infty\norm{\kernel{\mu_N\mu_N}{\veps}(\cdot,x')-\kernel{\mu\mu}{\veps}(\cdot,x')}_\infty
    \stepcounter{equation}\tag{\theequation}
    \label{eq:bound_t_tC_by_kernel_differences}
  \end{align*}
  where in the second inequality, we used $\int\kernel{\mu_N\mu_N}{\veps}(x,x')\,\diff\mu_N(x')=1$.
  Let $K\assign\frac{\exp(2\Lip(c))}{2 \min\{C_{\mu}, C_{\nu}\}}$.
  Due to \Cref{prop:bounds_kernel} with \Cref{ass:min_ball_mass} we have for $x' \in \spt(\mu)$
  \begin{align} \label{eq:kernel_bound_with_min_ball_mass}
    \kernel{\mu\mu}{\veps}(x,x') \leq \frac{2 K}{\veps^{\mathcal{D}_{\mu}}}.
  \end{align}
  We now derive a bound for $\norm{\kernel{\mu_N\mu_N}{\veps}(\cdot,x')-\kernel{\mu\mu}{\veps}(\cdot,x')}_\infty$ for points $x'\in\spt(\mu_N)\subset\spt(\mu)$, the other difference can be controlled in the same way.
  Let $\bar{\alpha}^{\veps}$ and $\bar{\alpha}^{\veps}_N$ be the optimal symmetric duals for entropic self-transport of $\mu$ and $\mu_N$ respectively.
  Let $\sigma_N\in \R$ such that $\bar{\alpha}^\veps_N(x_0)=\bar{\alpha}^\veps(x_0) - \sigma_N$ for some fixed $x_0 \in \SpX$.
  According to \cite[Lemma E.4, Proposition E.5]{luise2019sinkhorn} (and using \cite[Proposition 1]{Genevay19a} to bound the $r$ appearing in $\bar{r}$ defined in \cite[Equation (E.3)]{luise2019sinkhorn}), there exist constants $L$, $C$ depending only on $c$ and $\SpX$ such that we have with a probability of at least $1-\tau$ 
  \begin{align} \label{eq:dual_potential_convergence}
    \norm{\bar{\alpha}^\veps_N + \sigma_N -\bar{\alpha}^\veps}_\infty
    \leq 
    L \veps^{1 - \lfloor d/2 \rfloor}
    \frac{\exp(C/\veps)}{\sqrt{N}}\log\left(\frac 3\tau\right).
  \end{align}
  For any fixed $x'\in \SpX$, we have
  \begin{align*}
    \frac{2\sigma_N}{\veps} 
    &=
    \log\left(
    \exp\left(\frac{2 \sigma_N}{\veps}\right)
    \int_{\SpX} \kernel{\mu_N\mu_N}{\veps}(x,x') \,\diff\mu_N(x)
    \right)
    \\&=
    \log\left( 
    \int_{\SpX}
    \exp\left(
    \frac{
      \bar{\alpha}_N^\veps(x) + \bar{\alpha}_N^\veps(x') 
      - \bar{\alpha}^\veps(x) - \bar{\alpha}^\veps(x') + 2\sigma_N
    }{\veps}
    \right) 
    \kernel{\mu\mu}{\veps}(x,x') 
    \,\diff\mu_N(x)
    \right)
    \\
    \Rightarrow \quad
    \abs{\frac{2\sigma_N}{\veps}}
    &\leq
    \frac{2}{\veps}
    \norm{\bar{\alpha}^\veps_N + \sigma_N -\bar{\alpha}^\veps}_{\infty}
    + \abs{\log\left(\int_{\SpX} \kernel{\mu\mu}{\veps}(x,x') \,\diff\mu_N(x)\right)}.
    \stepcounter{equation}\tag{\theequation}
    \label{eq:bound_sigma_N}
  \end{align*}
  Using \eqref{eq:kernel_bound_with_min_ball_mass}, by \Cref{thm:Hoeffding} we have with a probability of at least $1-\tau$
  \begin{align*}
    \abs{\int\kernel{\mu\mu}{\veps}(x,x')\diff\mu_N(x) - 1} 
    \leq 
    \sqrt{\frac{K \log(2/\tau)}{N\veps^{\mathcal{D}_\mu}}}
    \leq
    \frac{1}{2}
  \end{align*}
  where the last inequality follows from the assumption.
  Note that for $\abs{a - 1} \leq \frac{1}{2}$ we have $\abs{\log(a)}\leq 2\abs{a - 1}$, hence by \eqref{eq:bound_sigma_N} with a probability of at least $1-\tau$ 
  \begin{align*}
    \abs{\frac{2\sigma_N}{\veps}}
    &\leq
    \frac{2}{\veps}
    \norm{\bar{\alpha}^\veps_N + \sigma_N -\bar{\alpha}^\veps}_\infty 
    + 2\sqrt{\frac{K \log(2/\tau)}{N\veps^{\mathcal{D}_\mu}}}.
    \stepcounter{equation}\tag{\theequation}
    \label{eq:bound_sigma_N_2}
  \end{align*}
  Using \eqref{eq:dual_potential_convergence} and \eqref{eq:bound_sigma_N_2}, by the union bound with a probability of at least $1-2\tau$ we have
  \begin{align*}
    \abs{\frac{\bar{\alpha}_{N}^\veps(x)-\bar{\alpha}^{\veps}(x)+\bar{\alpha}_N^\veps(x')-\bar{\alpha}^\veps(x')}{\veps}}
    &\leq
    \frac{2}{\veps} 
    \norm{\bar{\alpha}^\veps_N + \sigma_N -\bar{\alpha}^\veps}_\infty
    + \abs{\frac{2\sigma_N}{\veps}}
    \\
    &\leq
    \frac{4}{\veps}
    \norm{\bar{\alpha}^\veps_N + \sigma_N -\bar{\alpha}^\veps}_\infty 
    + 2\sqrt{\frac{K \log(2/\tau)}{N\veps^{\mathcal{D}_\mu}}} 
    \\&\leq
    \frac{\log(3/\tau)}{\veps^{d/2} \sqrt{N}}\left(4L\exp(C/\veps) 
    + 2\sqrt{K}\right) 
    \leq 
    1
    \stepcounter{equation}\tag{\theequation}
    \label{eq:joint_dual_difference}
  \end{align*}
  where the last inequality corresponds to the assumption. 
  Note that for any $\abs{a}\leq 1$ we have $\abs{\exp(a) - 1} \leq 2\abs{a}$.
  Using this bound with \eqref{eq:joint_dual_difference} and \eqref{eq:kernel_bound_with_min_ball_mass} yields with a probability of at least $1 - 2 \tau$
  \begin{align*}
    \norm{
      \kernel{\mu_N\mu_N}{\veps}(\cdot,x')
      - \kernel{\mu\mu}{\veps}(\cdot,x')
    }_\infty
    &\leq 
    \frac{2K}{\veps^{\mathcal{D}_\mu}}
    \norm{\exp\left(\frac{\bar{\alpha}_{N}^\veps(\cdot)-\bar{\alpha}^{\veps}(\cdot)+\bar{\alpha}_N^\veps(x')-\bar{\alpha}^\veps(x')}{\veps}\right)-1}_\infty
    \\&\leq
    \frac{4K\log(3/\tau)}{\veps^{\mathcal{D}_{\mu}+d/2}\sqrt{N}}\left(4L\exp(C/\veps) + 2\sqrt{K}\right).
    \stepcounter{equation}\tag{\theequation}
    \label{eq:kernel_difference}
  \end{align*}    
  Combining 
  \eqref{eq:bound_t_tC_by_kernel_differences} with 
  \eqref{eq:kernel_bound_with_min_ball_mass} and 
  \eqref{eq:kernel_difference} we get with a probability of at least $1 - 4 \tau$ (since we need the shown bounds for both $\mu$ and $\nu$)
  \begin{align*}
    \norm{t^{\veps}_N - t^{C,\veps}_N}_{\infty}
    &\leq
    \frac{4K\log(3/\tau)}{\veps^{\mathcal{D}_{\nu}+d/2}\sqrt{N}}
    \left(4L\exp(C/\veps) + 2\sqrt{K}\right)
    +
    \frac{2 K}{\veps^{\mathcal{D}_{\nu}}}
    \frac{4K\log(3/\tau)}{\veps^{\mathcal{D}_\mu+d/2}\sqrt{N}}\left(4L\exp(C/\veps) + 2\sqrt{K}\right).
    \qedhere
  \end{align*}
\end{proof}

\begin{remark}
  For fixed $\veps>0$, qualitative convergence $\|t^{\veps}_N-t^{C,\veps}_N\|_\infty \to 0$ and $\|T^{B,\veps}_N-T^{C,\veps}_N\|_{\HS} \to 0$ as $N \to \infty$ can be established for compact metric spaces $\SpX$ that are not subsets of $\R^d$ with continuous cost functions $c \in \Cont(\SpX \times \SpX)$ by compactness arguments where one uses that the entropic transport kernels $\kernel{\mu_N\nu_N}{\veps}$ are equicontinuous with their modulus of continuity only depending on $c$ and $\veps$, but not on $\mu_N$ or $\nu_N$.
  This proof strategy was used in \cite{EntropicTransfer22} for the single-smoothed transfer operators (see Remark \ref{rem:single_smoothing}) and it can be adapted to the setting of this article.
\end{remark}

Finally, we bound the distance between $T^{C,\veps}_N$ and $T^{\veps}_N$.
\begin{theorem} \label{prop:prob_var_bound}
  Let Assumptions \ref{ass:lip_c} and \ref{ass:min_ball_mass} hold and let $\SpX$ be a compact subset of $\R^d$ with $\diam(\SpX) \geq 1$ for some $d \in \N$. 
  Let $\tau \in (0,1)$ and assume $\veps \leq 1$ and $N \geq 3$. 
  Then with a probability of at least $1-\tau$ we have
  \begin{align*}
    \norm{T^{C,\veps}_N - T^{\veps}}_{\HS}
    \leq
    \norm{t_N^{C,\veps} - t^\veps}_\infty 
    \lesssim
    \frac{\sqrt{2d}}{\veps^{1+\mathcal{D}_\mu + \mathcal{D}_{\nu}}} \sqrt{\frac{\log N}{N}} \sqrt{\log\left(\frac{4 \sqrt{2} \diam(\SpX)}{\tau}\right)}
  \end{align*}
  with a constant depending only on $C_{\mu}, C_{\nu}$ and $\Lip(c)$.
\end{theorem}

The proof of \Cref{prop:prob_var_bound} is based on the following Lemma, which is a variant of a standard result on the concentration of empirical processes in the spirit of \cite{cucker2002mathematical}, adapted to our setting.

\begin{lemma} \label{lem:covering_number}
  Let $\mathcal{T}$ be a compact subset of a finite-dimensional vector space with norm $\norm{\cdot}$ with $\diam(\mathcal{T}) \geq 1$.
  Let $\mathcal{Y}$ be a compact space and $\rho \in \mathcal{P}(\mathcal{Y})$.
  Suppose we have a parametrized function class $\mathcal{F} \assign \{f_t \in \Leb^{\infty}(\rho) \mid t \in \mathcal{T}\}$ such that there exist constants $C, L > 0$ for which
  \begin{align*}
    \forall t, t' \in \mathcal{T}:
    \ 
    \norm{f_t - f_{t^\prime}}_{\Leb^{\infty}(\rho)} \leq L \norm{t-t^\prime}
    \ \wedge\ 
    f_t \geq 0
    \ \wedge\ 
    \norm{f_t}_{\Leb^{\infty}(\rho)} \leq C.
  \end{align*}
  Then, for i.i.d.~random variables $Y, (Y_i)_{i = 1}^N \sim \rho$, $N \geq 3$ and $\eta \in (0,1)$ we have
  \begin{align*}
    \mathbb{P}\left(
    \sup_{t\in \mathcal{T}}\abs{\frac{1}{N}\sum_{i = 1}^N f_t(Y_i) - \mathbb{E}(f_t(Y))} 
    > 
    \sqrt{\frac{\log N}{N}} 
    \left(C + 2L\right)
    \sqrt{\dim(\mathcal{T})\log\left(\frac{4\diam(\mathcal{T})}{\eta}\right)}
    \right)
    < 
    \eta.
  \end{align*}
\end{lemma}
\begin{proof}
  For $r \in (0, \diam(\mathcal{T})]$, let $\mathcal{S}_r$ be the center points a minimal $r$-cover of $\mathcal{T}$, i.e.~$\mathcal{S}_r \subset \mathcal{T}$ is a set with minimum cardinality such that for any $t\in \mathcal{T}$, there exists some $t_j \in \mathcal{S}_r$ with $\norm{t_j - t} \leq r$. By \cite[Prop.~5]{cucker2002mathematical} we have
    \begin{align*}
      \abs{\mathcal{S}_r} \leq \left(\frac{2\,\tn{diam}(\mathcal{T})}{r}\right)^{\dim(\mathcal{T})}.
    \end{align*}
    Note that $0 \leq f_t(Y) \leq C$ almost surely. 
    By \Cref{thm:Hoeffding} and the union bound we have for any $s > 0$
    \begin{align*}
      \mathbb{P}\left(
      \sup_{t_j\in \mathcal{S}_r}\abs{\frac{1}{N}\sum_{i = 1}^N f_{t_j}(Y_i) - \mathbb{E}(f_{t_j}(Y))} 
      > 
      s
      \right)  
      &\leq 
      2\abs{\mathcal{S}_r}\exp\left(-\frac{2Ns^2}{C^2}\right).
    \end{align*}
    Now for an arbitrary $t\in\mathcal{T}$, let $t_j\in\mathcal{S}_r$ such that $\norm{t_j - t}\leq r$. Then
    \begin{align*}
      &\abs{\frac{1}{N}\sum_{i = 1}^N f_t(Y_i) - \mathbb{E}(f_t(Y))} 
      \\&\leq
      \abs{\frac{1}{N}\sum_{i = 1}^N f_t(Y_i) - \frac{1}{N}\sum_{i = 1}^N f_{t_j}(Y_i)} 
      + \abs{\frac{1}{N}\sum_{i = 1}^N f_{t_j}(Y_i) - \mathbb{E}(f_{t_j}(Y))} 
      + \abs{\mathbb{E}(f_{t_j}(Y)) - \mathbb{E}(f_t(Y))}
      \\&\leq
      2Lr + \abs{\frac{1}{N}\sum_{i = 1}^N f_{t_j}(Y_i) - \mathbb{E}(f_{t_j}(Y))}.
    \end{align*}
    Therefore
    \begin{align*}
      \mathbb{P}\left(\sup_{t\in \mathcal{T}}\abs{\frac{1}{N}\sum_{i = 1}^N f_t(Y_i) - \mathbb{E}(f_t(Y))} > s + 2Lr\right)  
      &\leq 
      2
      \left(\frac{2\tn{diam}(\mathcal{T})}{r}\right)^{\dim(\mathcal{T})}
      \exp\left(-\frac{2Ns^2}{C^2}\right).
    \end{align*}
    Now set $r \assign \frac{1}{\sqrt{N}}$ and $s$ such that the right hand side is equal to $\eta$, that is
    \begin{align*}
      s 
      \assign 
      \frac{C}{\sqrt{2N}}\sqrt{\dim(\mathcal{T}) \log\left(\frac{2\diam(\mathcal{T})}{r}\right) + \log\left(\frac{2}{\eta}\right)}.
    \end{align*}
    Using basic bounds, rearrangement and $a + b \leq 2 a b$ for $a, b \geq 1$, we obtain 
    \begin{align*}
      s + 2 L r
      &=
      \frac{C}{\sqrt{2N}}\sqrt{\dim(\mathcal{T}) \log\left(2\diam(\mathcal{T})\sqrt{N}\right) + \log\left(\frac{2}{\eta}\right)} + \frac{2L}{\sqrt{N}}
      \\&\leq
      \sqrt{\frac{\log N}{N}}
      \left(C + 2L\right)
      \sqrt{\dim(\mathcal{T})\log\left(\frac{4\diam(\mathcal{T})}{\eta}\right)}.
    \end{align*}
    Putting everything together we arrive at the result.
\end{proof}

\begin{proof}[Proof of \Cref{prop:prob_var_bound}]
  Recall the kernel definitions \eqref{eq:def_kcontpidisc_op} and \eqref{eq:TEpsKernel}.
  The first inequality follows the same way as in the proof of \Cref{prop:prob_bias_bound}. 
  Define the function class
  \begin{align*}
    \mathcal{F} 
    = 
    \Big\{f_{(t,t')}(x,y) = \kernel{\mu\mu}{\veps}(t,x) \kernel{\nu\nu}{\veps}(t',y)
    \ \Big|\ 
    (t,t')\in \SpX^2
    \Big\}.
  \end{align*}
  Due to \eqref{prop:bounds_kernel} with \Cref{ass:min_ball_mass} we have that any $f_{(t,t')} \in \mathcal{F}$ is bounded on $\spt(\pi)$, specifically
  \begin{align*}
    \sup_{\substack{x \in \spt(\mu),\\y \in \spt(\nu)}}
    \kernel{\mu\mu}{\veps}(t,x) \kernel{\nu\nu}{\veps}(t',y) 
    \lesssim 
    \veps^{-\mathcal{D}_\mu - \mathcal{D}_\nu}.
  \end{align*}
  Furthermore, for any $(s,s'), (t,t') \in \SpX^2 \subset \R^{2d}$ and $(x,y) \in \spt(\pi)$ we have
  \begin{align*}
    \Big|f_{(t,t')}(x,y) - f_{(s,s')}(x,y)\Big| 
    &= 
    \Big|
    \kernel{\mu\mu}{\veps}(t,x)\kernel{\nu\nu}{\veps}(t',y) 
    - \kernel{\mu\mu}{\veps}(s,x) \kernel{\nu\nu}{\veps}(s',y)
    \Big| 
    \\&\leq 
    \kernel{\nu\nu}{\veps}(t', y)
    \Big|\kernel{\mu\mu}{\veps}(t, x) - \kernel{\mu\mu}{\veps}(s, x)\Big|
    + \kernel{\mu\mu}{\veps}(s, x)
    \Big|\kernel{\nu\nu}{\veps}(t', y) - \kernel{\nu\nu}{\veps}(s', y)\Big|
    \\&\leq 
    \norm{\kernel{\nu\nu}{\veps}(\cdot,y)}_\infty 
    \Lip(\kernel{\mu\mu}{\veps}(\cdot,x))
    \norm{s - t} 
    + \norm{\kernel{\mu\mu}{\veps}(\cdot,x)}_\infty
    \Lip(\kernel{\nu\nu}{\veps}(\cdot,y))
    \norm{s'-t'}
    \\&\lesssim
    \veps^{-\mathcal{D}_\mu - \mathcal{D}_\nu - 1}
    \norm{(t,t') - (s,s')}.
  \end{align*}
  This allows us to apply \Cref{lem:covering_number} (with $\mathcal{T} \assign \SpX^2$, $\rho \assign \pi$), finishing the proof.
\end{proof}

\subsection{Convergence to the unregularized transfer operator} \label{subsec:cv_Teps2T}
So far we have discussed the discrepancy between $T_N^\veps$ (or its extensions $T_N^{A,\veps}$ and $T_N^{B,\veps}$) and $T^\veps$. In this section we briefly address the relation between $T$ and $T^\veps$ as $\veps \to 0$.
When $T$ is not compact, then we expect $T^\veps$ to diverge in Hilbert--Schmidt norm, as $\veps \to 0$. The following assumption and proposition give an exemplary setting where we find $T^\veps \to T$ as $\veps \to 0$ in Hilbert--Schmidt norm.

\begin{assumption} \label{ass:t_dens_Pi}
  The plan $\pi$ is absolutely continuous with respect to the product of its marginals $\mu$ and $\nu$, i.e. $\pi \ll \mu \otimes \nu$, and the density $t \assign \RadNik{\pi}{\mu \otimes \nu}$ satisfies the Hölder-type continuity property
  \begin{equation} \label{eq:t_li_c}
    \abs{t(x',y')-t(x,y)}\leq L(c(x,x')+c(y,y'))^{l}
    \quad \tn{for all} \quad x,x',y,y' \in \SpX,
  \end{equation}
  for suitable constants $L,l>0$. 
\end{assumption}
Of course, for $c(x,x')=d(x,x')^2$ this reduces to standard Hölder continuity on $\SpX \times \SpX$.
Under this assumption, the discrepancy between $T^\veps$ and $T$ vanishes with an explicit rate in $\veps$.
\begin{theorem} \label{prop:Teps2T}
  Given Assumptions \ref{ass:min_ball_mass} and \ref{ass:t_dens_Pi} for sufficiently small $\veps>0$ and any $x\in\spt(\mu)$, $y\in\spt(\nu)$ we have
  \begin{align*}
    \abs{t^\veps(x,y)-t(x,y)}\lesssim(\veps\log(1/\veps))^{l}.
  \end{align*}
  The same upper bound holds for $\norm{T^\veps-T}_{\HS}$. The multiplicative constant depends only on $t$, the cost $c$, and the measures $\mu$ and $\nu$.
\end{theorem}

\begin{proof}
  Like in the previous proofs, since $T^\veps$ and $T$ are kernel operators, we only need to provide the upper bound on $\abs{t^\veps(x,y)-t(x,y)}$ for $x,y$ on the right supports. For any $(x,y)\in\spt(\mu)\times\spt(\nu)$ and any radius $\eta>0$ we have
  \begin{align*}
    \abs{t^\veps(x,y)-t(x,y)} 
    &\leq 
    \int_{\SpX^2}
    \abs{t(x',y')-t(x,y)}
    \kernel{\mu\mu}{\veps}(x,x')
    \kernel{\nu\nu}{\veps}(y,y')
    \,\diff\mu(x')\diff\nu(y')
    \\&\leq 
    L\int_{\substack{c(x,x')\leq\eta\\c(y,y')\leq\eta}}
    (c(x,x')+c(y,y'))^{l}
    \kernel{\mu\mu}{\veps}(x,x')
    \kernel{\nu\nu}{\veps}(y,y')
    \,\diff\mu(x')\diff\nu(y')
    \\&\quad 
    +2\norm{t}_\infty
    \int_{c(x,x')>\eta}
    \kernel{\mu\mu}{\veps}(x,x')
    \,\diff\mu(x')+2\norm{t}_\infty
    \int_{c(y,y')>\eta}
    \kernel{\nu\nu}{\veps}(y,y')
    \,\diff\nu(y')
    \\&\lesssim 
    \eta^{l}
    +\norm{t}_\infty \frac{e^{-\frac{\eta}{\veps}}}{\veps^{\mathcal{D}_\mu}}+\norm{t}_\infty\frac{e^{-\frac{\eta}{\veps}}}{\veps^{\mathcal{D}_\nu}}.
  \end{align*}
  Here, to obtain the second inequality, we split the integral on $\SpX^2$ according to $\eta$-level sets of $c$ and then used \eqref{eq:t_li_c} on the first term. Then for the last inequality we argue as follows.
  In the first term, replace $c$ by its upper bound $\eta$ and use that integral over the kernels is bounded by $1$, e.g.~due to $\int_{\SpX} \kernel{\mu\mu}{\veps}(x,x')\,\diff \mu(x)=1$.
  In the second and third term, use the bound of \Cref{rem:better_bound_selfOT} to control the remaining kernels, assuming that $\veps$ is sufficiently small for the bounds of \Cref{ass:min_ball_mass} to apply, and then use that $\mu$ and $\nu$ have mass $1$.
  Assume then that $\veps \leq \exp(-1)$ and set $\eta \assign ((\max\{\mathcal{D}_\mu,\mathcal{D}_\nu\}+l)\veps\log(1/\veps))$ each of the three terms in the final expression are bounded from above (up to a multiplicative constant) by $(\veps\log(1/\veps))^{l}$.
\end{proof}

\subsection[The stationary case mu=nu]{The stationary case $\mu=\nu$} \label{subsec:stationary}
In the case $\mu=\nu$ the operators $T$ and $T^\veps$ can be interpreted as transition operators for a time-homogeneous Markov chain and they become endomorphisms that can be analyzed by eigendecomposition as an alternative to singular value decomposition. 
However, even in this setting the two empirical marginal measures are usually different, i.e.~$\mu_N \neq \nu_N$, and thus neither $T_N$ nor $T_N^\veps$, as introduced in \eqref{eq:TEpsN}, will be endomorphisms.
In this case, we can adjust the definition of $T_N^\veps$ to turn it into an endomorphism on $\Leb^2(\mu_N)$, that can be analyzed by finite-dimensional eigendecomposition.
We achieve this by adjusting the second blur operator to transfer from $\Leb^2(\nu_N)$ back to $\Leb^2(\mu_N)$ (as originally proposed in \cite{EntropicTransfer22}). The following definition collects all adaptations.

\begin{definition}[Operator variants for stationary case] \label{def:op_stationary}
  The regularized empirical transfer operator is defined as
  \begin{align*}
    T_N^\veps & \assign G_{\nu_N\mu_N}^\veps \circ T_N \circ G_{\mu_N\mu_N}^\veps,
  \end{align*}
  which is associated with the integral kernel $t_N^\veps \assign (\kernel{\mu_N\mu_N}{\veps}:\pi_N:\kernel{\nu_N\mu_N}{\veps})$. The extension is set to be
  \begin{align*}
    T_N^{A,\veps} & \assign (T^{\mu}_N)^* \circ T_N^\veps \circ T^{\mu}_N
  \end{align*}
  where $T^\mu_N$ is the operator induced by the optimal unregularized plan $\gamma^\mu_N$ between $\mu$ and $\mu_N$. $T_N^{A,\veps}$ has the integral kernel
  $t_N^{A,\veps}(x,y) \assign \int_{\SpX^2} t_N^\veps(x',y')\,\diff \gamma^\mu_N(x'|x)\,\diff \gamma^\mu_N(y'|y)$
  where $(\gamma^\mu_N(\cdot|x))_x$ denotes the disintegration of $\gamma^\mu_N$ with respect to its $\mu$ marginal. The auxillary operator $T^{B,\veps}_N$ is defined as
  \begin{equation*}
    T^{B,\veps}_N: \Leb^2(\mu) \to \Leb^2(\mu),
    \quad
    u\mapsto \int_{\SpX^2}u(x)t^\veps_N(x,\cdot) \,\diff\mu(x)
  \end{equation*}
  where the difference to \eqref{eq:def_mixed_op} is the definition of $t^\veps_N$.
  And the auxillary operator $T^{C,\veps}_N$ is defined as
  \begin{equation*}
    T^{C,\veps}_N: \Leb^2(\mu) \to \Leb^2(\mu),
    \quad
    u\mapsto \int_{\SpX^2}u(x)t^{C,\veps}_N(x,\cdot)\,\diff\mu(x)
    \quad \tn{where} \quad
    t^{C,\veps}_N \assign (\kernel{\mu\mu}{\veps}:\pi_N:\kernel{\mu\mu}{\veps}).
  \end{equation*}
\end{definition}

In the case where $\mu=\nu$, replacing the definitions of Section \ref{subsec:main_defs} by those of Definition \ref{def:op_stationary} one finds that the convergence results of Section \ref{subsec:quant_cvg} still hold, if one replaces references to $\nu$ and $\nu_N$ by $\mu$ and $\mu_N$. The corresponding adaptation of the proofs is straight-forward.

\subsection{Spectral convergence} \label{subsec:SpectralConvergence}

Section \ref{subsec:quant_cvg} establishes convergence in Hilbert--Schmidt norm of $T_N^{A,\veps}$ and $T_N^{B,\veps}$ to $T^\veps$ as $N \to \infty$. For sufficiently regular $T$, by combining Sections \ref{subsec:quant_cvg} and \ref{subsec:cv_Teps2T} we find $T_N^{A,\veps},T_N^{B,\veps} \to T$ for suitable joint limits $N \to \infty$, $\veps \to 0$.
This convergence implies a notion of spectral convergence for eigen- and singular values and functions.
In addition, if the extension operators $T_N^{\mu}$ and $T_N^{\nu}$ (see Definition \ref{def:T_extension}) are chosen suitably, then the non-trivial parts of the spectra of $T_N^{A,\veps}$ and $T_N^{\veps}$ are identical and the related eigen- or singular functions are in one-to-one correspondence (\Cref{prop:extension_eigenpairs,prop:extension_svd}).
In this section we briefly recall some results for the stationary setting (Section \ref{subsec:stationary}) as discussed in \cite[Sections 4.8 and 4.9]{EntropicTransfer22} and discuss corresponding results for the non-stationary setting and singular value decomposition.

\begin{lemma}[\protect{\cite[Lemma 2.2]{bramble1973rate} as stated in \cite[Lemma 1]{EntropicTransfer22}}] \label{prop:cv_eigenvalues}
  In the stationary setting, assume that $T_N^{A,\veps} \to T^\veps$ as $N \to \infty$ in Hilbert--Schmidt norm. Let $\lambda_\veps$ be a nonzero eigenvalue of $T^\veps$ with algebraic multiplicity $m$ and $\Gamma$ be a disk centered at $\lambda_\veps$ containing no other point of the spectrum of $T^\veps$. Then for $N$ large enough, there are exactly $m$ eigenvalues $(\lambda_{N,j}^{A,\veps})_{j=1\dots m}$ (counted with multiplicity) for $T_N^{A,\veps}$ lying inside $\Gamma$.
\end{lemma}

\begin{theorem}[\protect{\cite[Theorem 5]{osborn1975spectral} as stated in \cite[Theorem 2]{EntropicTransfer22}}] \label{prop:spectral_vec}
  Let $(\lambda_{N}^{A,\veps})_N$ be a sequence of eigenvalues of $T^{A,\veps}_N$ that converges to an eigenvalue $\lambda^{\veps}$ of $T^\veps$ as $N \to \infty$.
  For each $N$, let $u_N^{A,\veps}$ be a corresponding unit eigenvector of $T^{A,\veps}_N$ at $\lambda_{N}^{A,\veps}$. Then there is a sequence of generalized eigenvectors $u_{N}^{\veps}$ of $T^\veps$ at $\lambda^\veps$ such that
  \begin{equation*}
    \norm{u_N^{A,\veps}-u_N^{\veps}}_{\Leb^2(\mu)}\lesssim   \opNorm{T_N^{A,\veps}-T^{\veps}}.
  \end{equation*}
  The multiplicative constant may depend on $\veps$ but does not depend on $N$.
\end{theorem}
Both results also hold when replacing $T_N^{A,\veps}$ by $T_N^{B,\veps}$ as defined in the stationary setting.
In the setting of Section \ref{subsec:cv_Teps2T} one may also replace $T^\veps$ by $T$ (with multiplicative constants then being independent of $\veps$).
Note that in \Cref{prop:spectral_vec} the `limiting generalized eigenvector' $u_N^\veps$ also depends on $N$. This accounts for the case when the eigenvalue $\lambda^\veps$ has geometric multiplicity greater one, and therefore the approximating sequence $u_N^{A,\veps}$ may be oscillating and non-convergent. Alternatively, it would be possible to formulate the above convergence result in terms of convergence of the orthogonal projections on each generalized eigenspace, using \cite[Theorem 1]{osborn1975spectral} and the relation between the gap between finite-dimensional subspaces and the orthogonal projectors to them \cite{kato1958perturbation}.

Next, we recall a corresponding convergence result for the singular value decomposition.
\begin{theorem}[\protect{\cite[Theorem 4.6]{crane2020singular}}]
  In the non-stationary setting, assume that $T_N^{A,\veps} \to T^\veps$ as $N \to \infty$ in Hilbert--Schmidt norm. Let $(\phi_k)_k \subset \Leb^2(\mu)$, $(\psi_k)_k \subset \Leb^2(\nu)$ be orthonormal sequences and $(\sigma_k)_k \subset \R$ a positive decreasing sequence s.t.~the singular value expansion for $T^\veps$ is $
  \sum_{k=1}^{\infty} \sigma_k\,\psi_k \otimes \phi_k^{\ast}$. Similarly, let $(\phi_{k,N})_k \subset \Leb^2(\mu)$, $(\psi_{k,N})_k \subset \Leb^2(\nu)$ be orthonormal sequences, and $(\sigma_{k,N})_k \subset \R$ a positive decreasing sequence s.t.~$T_N^{A,\veps} = \sum_{k=1}^{\infty} \sigma_{k,N}\,\psi_{k,N} \otimes \phi_{k,N}^{\ast}$. Then
  \begin{align*}
    \abs{\sigma_{k,N}-\sigma_k}
    \leq
    \opNorm{T_N^{A,\veps}-T^{\veps}} 
    \quad 
    \tn{for any $N$, $k$.}
  \end{align*}
\end{theorem}
There are also corresponding results for the convergence of the singular functions. Consider the span of singular functions associated with a given singular value $\sigma_k$, i.e.~let $E_k\assign\{\psi\in \Leb^2(\mu) \,:\,(T^\veps)^*T^\veps\psi=\sigma_k^2\psi\}$ and $F_k\assign\{\phi\in \Leb^2(\nu) \,:\,T^\veps(T^\veps)^*\phi=\sigma_k^2\phi\}$. Similarly, let $E_{k,N}$, $F_{k,N}$ be the respective subspaces of all the corresponding singular values $\sigma_{l,N}$ of $T_N^{A,\veps}$ that converge to $\sigma_k$ (i.e.~$l$ may be non-unique and different from $k$ when singular values repeat, but $\sigma_l=\sigma_k$ in the limit).
\begin{proposition}
  For any $k$, let $(\phi_{k,N}^{A,\veps})_N$ be a sequence of unit vectors in $(E_{k,N})_N$. Then there exists a sequence of vectors $(\phi_{k,N}^\veps)_N$ in $E_k$ such that for any $N$,
  \begin{equation} \label{eq:cv_sing_vecs}
    \norm{\phi_{k,N}^{A,\veps}-\phi_{k,N}^\veps}_{\Leb^2(\mu)}\lesssim\opNorm{T_N^{A,\veps}-T^{\veps}}
  \end{equation}
  with a multiplicative constant that does not depend on $N$, but could depend on $\veps$ or $k$.
  The symmetrical result between vectors of $F_{k,N}$ and $F_{k}$ also holds.
\end{proposition}

\begin{proof}
  This is a direct consequence of \cite[Corollary 4.9]{crane2020singular}, using the definition of the gap between these singular spaces that is used in the Corollary.
\end{proof}

Finally, we discuss the relation for eigenpairs and singular value decomposition between $T_N^{\veps}$ and its extension $T_N^{A,\veps}$. The following proposition and \Cref{rem:extension_choose_map} are closely related to \cite[Section 4.7]{EntropicTransfer22}.

\begin{proposition}[Correspondence of eigenpairs of $T_N^\veps$ and $T_N^{A,\veps}$ in the stationary setting] \label{prop:extension_eigenpairs}
  Consider the stationary setting, Definition \ref{def:op_stationary}, and assume that the optimal transport plan $\gamma_N^\mu \in \Pi(\mu,\mu_N)$ is induced by some map $\phi_N^\mu : \SpX \to \SpX$. Then $\lambda  \neq 0$ is a non-zero eigenvalue of $T_N^{\veps}$ if and only if it is an eigenvalue of $T_N^{A,\veps}$, and $u \in \Leb^2(\mu_N)$ is a corresponding eigenfunction of $T_N^{\veps}$ if and only if $(T_N^\mu)^* u$ is a corresponding eigenfunction of $T_N^{A,\veps}$. All eigenfunctions of $T_N^{A,\veps}$ for non-zero eigenvalues are of the form $(T_N^\mu)^* u$ for some $u \in \Leb^2(\mu^N)$.
\end{proposition}
\begin{proof}
  In this setting we have
  \begin{align*}
    \langle T_N^\mu u, v \rangle_{\Leb^2(\mu_N)} = \int_{\SpX^2} u(x)\,v(y)\,\diff \gamma_N^\mu(x,y)
    =
    \int_{\SpX} u(x)\,v(\phi_N^\mu(x))\,\diff \mu(x)
    \quad 
    \tn{for} \quad u \in \Leb^2(\mu),\, v \in \Leb^2(\mu_N),
  \end{align*}
  and therefore $(T_N^\mu)^* v=v \circ \phi_N^\mu$. Therefore
  \begin{align*}
    \langle (T_N^\mu)^* u, (T_N^\mu)^* v \rangle_{\Leb^2(\mu)}
    =
    \langle u, v \rangle_{\Leb^2(\mu_N)}
    \quad 
    \tn{for} \quad u,v \in \Leb^2(\mu_N),
  \end{align*}
  that is, $(T_N^\mu)^*$ is an isometric embedding of $\Leb^2(\mu_N)$ into $\Leb^2(\mu)$.
  Therefore, the restriction of $T_N^{A,\veps}$ to the image of $(T_N^\mu)^*$ can be identified with $T_N^\veps$, and $T_N^{A,\veps}$ is zero on the orthogonal complement. This implies the claim.
\end{proof}
In full analogy one obtains the following result for the non-stationary case.
\begin{proposition}[Correspondence of singular value decompositions of $T_N^\veps$ and $T_N^{A,\veps}$ in the non-stationary setting] \label{prop:extension_svd}
  Consider the non-stationary setting, Definitions \ref{def:entr_trans_basic} and \ref{def:T_extension}, and assume that the optimal transport plans $\gamma_N^\mu$ and $\gamma_N^\nu$ are induced by maps $\phi_N^\mu$ and $\phi_N^\nu$. Then $\lambda >0$ is a singular value of $T_N^{\veps}$ if and only if it is a singular value of $T_N^{A,\veps}$. Furthermore $u \in \Leb^2(\mu_N)$ and $v \in \Leb^2(\nu_N)$ are corresponding left- and right-singular functions of $T_N^{\veps}$ if and only if their piecewise constant extenstions $(T_N^\mu)^* u$ and $(T_N^\nu)^* v$ are corresponding left- and right-singular functions of $T_N^{A,\veps}$.
  All singular functions of $T_N^{A,\veps}$ for non-zero singular values are of the form $(T_N^\mu)^* u$ or $(T_N^\nu)^* v$ for some $u \in \Leb^2(\mu^N)$, $v \in \Leb^2(\nu^N)$.
\end{proposition}

\begin{remark}[Approximation of plans $\gamma_N^\mu$ and $\gamma_N^\nu$ by transport maps] \label{rem:extension_choose_map}
  Of course, the optimal plans $\gamma_N^\mu$ and $\gamma_N^\nu$ will in general not necessarily be induced by maps. A sufficient condition for this is, for instance, if $\SpX \subset \R^d$ and $\mu, \nu \ll \Lebesgue$, by virtue of Brenier's theorem \cite{brenier1991polar}. However, for the convergence analysis in Section \ref{subsec:quant_cvg} it is not necessary that the plans are actually optimal, as long as their induced transport costs tend to zero sufficiently fast as $N \to \infty$.
  For instance, by \cite[Theorem 1.32]{santambrogio2015optimal}, when $\SpX$ is a compact subset of $\R^d$, plans induced by maps are dense in the set of all plans $\Pi(\mu,\mu_N)$ as long as $\mu$ has no atoms. For any factor $q>1$ it is therefore always possible to find a plan of the form $\gamma_N^\mu = (\id,\phi_N^\mu)_{\#}\mu$ such that
  \begin{align*}
    \int_{\SpX^2} d^2\,\diff \gamma_N^\mu 
    = 
    \int_{\SpX} d(x,\phi_N^\mu(x))^2\,\diff \mu(x) \leq q \cdot W_2^2(\mu,\mu_N).
  \end{align*}
  Therefore, Proposition \ref{prop:cross_operators} (and its stationary variant, see Section \ref{subsec:stationary}) also hold when $T^{A,\veps}_N$ is constructed with these approximate plans that are induced by maps, if the multiplicative constant is increased slightly.
\end{remark}

\subsection{Out-of-sample embedding} \label{subsec:out_of_sampling}
Eigen- and singular functions for large eigen- and singular values of $T^\veps$ and its discrete approximation $T_N^\veps$ are important tools for analyzing the prominent features of the system dynamics. They can be used in methods such as spectral embedding and spectral clustering and give a coarse-grained description of the system.

For simplicity, in the following discussion we consider the stationary setting (Section \ref{subsec:stationary}), but analogous results can be obtained for the non-stationary setting.
Assume that based on some observed samples $(x_i,y_i)_{i=1}^N$ we have computed $T_N^\veps$, extracted some relevant eigenfunctions numerically and generated a spectral embedding of the samples.
Now, additional samples $(x_i,y_i)_{i=N+1}^{N+M}$ become available and we would like to insert them into the spectral embedding for some subsequent analysis of the system at hand.
It may be impractical to recompute $T_{N+M}^\veps$ and its eigenvectors each time some samples are added, or it may even be intractable for very large $M$. For subsequent analysis tasks it could instead be sufficient if an approximate interpolation of eigenfunctions $u$ of $T_N^\veps$ to the new samples $(x_i)_{i=N+1}^{M+N}$ was available.

Eigenfunctions $u$ of $T_N^\veps$ live in $\Leb^2(\mu_N)$. By \Cref{prop:extension_eigenpairs} (see also Remark \ref{rem:extension_choose_map}) any $u$ can be extended to an eigenfunction $(T_N^\mu)^*u \in \Leb^2(\mu)$ of $T_N^{A,\veps}$ that could be evaluated almost surely at the new positions $(x_i)_{i=N+1}^{M+N}$.
However, the transport plan $\gamma_\mu^N$ or map $\phi^\mu_N$ underlying $T_N^\mu$ is unknown in practice. In addition, the extended function $(T_N^\mu)^*u=u \circ \phi_N^\mu$ is piecewise constant and may therefore be undesirable as an interpolation for spectral embedding.

In this section we propose an alternative interpolation scheme, exploiting the regularity of entropic transport kernels (\Cref{def:entOT_kernel}) and the induced regularized operator kernel $t_N^\veps$ (\Cref{def:empirical_operator}), which is consistent in the limit $N \to \infty$ (and potentially in a suitable joint limit with $\veps \to 0$, see Section \ref{subsec:cv_Teps2T}).
By boundedness and equicontinuity of the family of functions $(t_N^\veps(x,\cdot))_{x \in \SpX}$ (see Section \ref{subsec:preliminary}), 
equation \eqref{eq:TEpsKernelN} defining $T^{\veps}_N$ maps $u \in \Leb^2(\mu_N)$ to a continuous function.
Indeed, $T^{\veps}_N$ can be interpreted as a compact operator from $\Leb^2(\mu_N)$ to $\Cont(\SpX) \hookrightarrow \Leb^2(\mu_N)$.
Let us denote this operator by
\begin{equation*}
  \widetilde{T}_N^\veps : \Leb^2(\mu_N) \to \Cont(\SpX),
  \quad
  u \mapsto 
  \int_{\SpX} u(x) t_N^\veps(x,\cdot)\,\diff \mu_N(x).
\end{equation*}
By definition one has $\widetilde{T}_N^\veps u = T_N^\veps u$ $\mu_N$-almost everywhere, and therefore $\widetilde{T}_N^\veps$ can indeed be interpreted as interpolation of $T_N^\veps$ from $\spt(\mu_N)$ to all of $\SpX$.
Let now $u \in \Leb^2(\mu_N)$ be an eigenfunction of $T_N^\veps$ for some eigenvalue $\lambda \neq 0$. Then by definition $u=\tfrac1\lambda T_N^\veps u$.
We therefore introduce the extension of $u$ to $\SpX$ as
\begin{align} \label{eq:out_of_sample}
  \widetilde{u} \assign \frac1\lambda \widetilde{T}_N^\veps u.
\end{align}
This extension satisfies $\widetilde{u}=u$ $\mu_N$-almost everywhere.
In addition (see \Cref{prop:extension_eigenpairs}), $(T_N^\mu)^*u \in \Leb^2(\mu)$ is an eigenfunction of $T_N^{A,\veps}$ for the same eigenvalue $\lambda$ and $(T_N^\mu)^*u=\tfrac1\lambda (T_N^\mu)^* T_N^\veps u$. The following estimate in the spirit of \Cref{prop:cross_operators} can then be used to control the discrepancy between $\widetilde{u}$ and $(T_N^\mu)^*u$. Combined with results from Section \ref{subsec:SpectralConvergence} this implies asymptotic consistency of the interpolation in the limit $N \to \infty$ (and $\veps \to 0$, when appropriate).

\begin{proposition} \label{prop:interpolation}
  Consider the stationary setting, and let Assumptions  \ref{ass:lip_c} and \ref{ass:min_ball_mass} hold. Let $N\in \N$, $\veps>0$ sufficiently small and $\tau < 1$ such that $\tau \geq N \exp\left(-\tfrac{N}{2}(C_\nu \veps^{\mathcal{D}_\nu})^2\right)$.
  For $u \in \Leb^2(\mu_N)$ set $\varphi \assign (T_N^\mu)^* T_N^\veps u$ and $\widetilde{\varphi} \assign \widetilde{T}_N^\veps u$. Then with probability at least $1-\tau$
  \begin{align*}
    \norm{\widetilde{\varphi}-\varphi}_{\Leb^2(\mu)}
    \lesssim
    \norm{u}_{\Leb^2(\mu_N)} \frac{W_2(\mu_N,\mu)}{\veps^{1+\mathcal{D}_\nu}}.
  \end{align*}
  where the constant depends only on $C_{\nu}$ and $\Lip(c)$.
\end{proposition}
\begin{proof}
  Using Jensen's inequality,
  \begin{align*}
    \norm{\widetilde{\varphi}-\varphi}_{\Leb^2(\mu)}^2
    &=
    \int_\SpX
    \abs{
      \int_\SpX u(x)\,t_N^\veps(x,y)\,\diff\mu_N(x)
      - \int_\SpX\int_\SpX u(x)\,t^\veps_N(x,y')\,\diff\mu_N(x)\,\diff\gamma_{N}^\mu(y'|y)
    }^2
    \diff\mu(y)
    \\&\leq 
    \int_{\SpX^3}
    \abs{u(x)}^2\abs{t^\veps_N(x,y)-t^\veps_N(x,y')}^2\diff\mu_N(x)
    \,\diff\gamma_{N}^\mu(y,y')
    \\&\leq 
    \int_{\SpX^3}
    \abs{u(x)}^2\Lip(t^\veps_N(x,\cdot))^2d(y,y')^2\diff\mu_N(x)\,\diff\gamma_{N}^\mu(y,y')
    \\&\leq
    \norm{u}_{\Leb^2(\mu_N)}^2 \cdot \sup_{x\in \SpX}\Lip(t^\veps_N(x,\cdot))^2 \cdot W^2_2(\mu,\mu_N).
  \end{align*}
  The statement follows since by the assumptions on $N$, $\veps$ and $\tau$, with probability at least $1-\tau$, by \Cref{cor:empirical_lip_bounds} one has
  $\Lip(t^\veps_N(x,\cdot)) \lesssim \veps^{-1-\mathcal{D}_\nu}$.
\end{proof}

\begin{corollary}
  Consider the setting of \Cref{prop:interpolation} and assume that $u \in \Leb^2(\mu_N)$ is an eigenfunction of $T_N^\veps$ for eigenvalue $\lambda \neq 0$. Set $\widetilde{u} \assign \tfrac1\lambda \widetilde{T}_N^\veps u$. Then with probability at least $1-\tau$,
  \begin{align*}
    \norm{\widetilde{u}-(T_N^\mu)^* u}_{\Leb^2(\mu)}
    \lesssim
    \frac{1}{\abs{\lambda}} \norm{u}_{\Leb^2(\mu_N)} \frac{W_2(\mu_N,\mu)}{\veps^{1+\mathcal{D}_\nu}}.
  \end{align*}
\end{corollary}

Finally, we briefly discuss the analogue concept for singular value decomposition in the non-stationary setting.
In complete analogy to the above results one obtains the following interpolation and error estimate.
\begin{corollary}
  Consider the setting of \Cref{prop:interpolation}, but now for the non-stationary setting.
  Let $(v,u) \in \Leb^2(\nu_N) \otimes \Leb^2(\mu_N)$ be a pair of left- and right-singular functions of $T_N^\veps$ for singular value $\lambda > 0$, i.e.
  \begin{align*}
    v 
    &= 
    \frac1\lambda T_N^\veps u, 
    &
    u 
    &= 
    \frac1\lambda (T_N^\veps)^* v.
  \end{align*}
  Introduce the interpolations $\widetilde{v},\widetilde{u} \in \Cont(\SpX)$ of $v,u$ as
  \begin{align*}
    \widetilde{v} 
    &: 
    y \mapsto \frac{1}{\lambda}\langle t_N^\veps(\cdot,y),u \rangle_{\Leb^2(\mu_N)}, 
    &
    \widetilde{u} 
    &: 
    x \mapsto \frac{1}{\lambda}\langle t_N^\veps(x,\cdot),v \rangle_{\Leb^2(\nu_N)}.
  \end{align*}
  Then, with probability at least $1-\tau_{\nu}$ (resp.~$1-\tau_{\mu}$, see \Cref{cor:empirical_lip_bounds}), one has
  \begin{align*}
    \norm{\widetilde{v}-(T_N^\nu)^* v}_{\Leb^2(\nu)} 
    &\lesssim
    \frac1\lambda \norm{u}_{\Leb^2(\mu_N)} \frac{W_2(\nu_N,\nu)}{\veps^{1+\mathcal{D}_\nu}}, 
    &
    \norm{\widetilde{u}-(T_N^\mu)^* u}_{\Leb^2(\mu)} 
    &\lesssim
    \frac1\lambda \norm{v}_{\Leb^2(\nu_N)} \frac{W_2(\mu_N,\mu)}{\veps^{1+\mathcal{D}_\mu}}.
  \end{align*}
\end{corollary}

\begin{proof}
  Similar calculations as in the proof of \Cref{prop:interpolation} give
  \begin{align*}
    \norm{\widetilde{v} - (T^{\nu}_N)^{\ast} v}_{\Leb^2(\nu)}^2
    \leq
    \frac{1}{\lambda^2} 
    \norm{u}^2_{\Leb^2(\mu_N)}
    \sup_{x\in \SpX} 
    \Lip(t^{\veps}_N(x,\cdot))^2 
    W_2^2(\nu_N, \nu).
  \end{align*}
  Taking the square root and using \Cref{cor:empirical_lip_bounds}, with probability at least $1 - \tau_{\nu}$ this gives the first result. The second one follows symmetrically.
\end{proof}

\section{Examples and numerical experiments} \label{sec:numerics}
Code for the numerical examples is available at \url{https://github.com/OTGroupGoe/StochasticETO}.
\subsection{Non-convergence for single blur} \label{subsec:exampleNonConv}
In this article we constructed the entropic transfer operator by applying two blurring steps $T^{\veps} = G^{\veps}_{\nu\nu} \circ T \circ G^{\veps}_{\mu\mu}$ (see \eqref{eq:TEps}) as opposed to a single one $\widetilde{T}^{\veps} \assign G^{\veps}_{\nu\nu} \circ T$ as in \cite{EntropicTransfer22}.
We will now give an example for a non-deterministic $T$ (i.e. $T$ is not induced by a time evolution map $F$) where double blurring is required for convergence in Hilbert--Schmidt norm.

Let $\SpX \assign [0,1]$, $\mu \assign\mathcal{U}(\SpX)$, $\nu \assign \frac{1}{2}(\delta_0 + \delta_1)$, $\pi \assign \mu \otimes \nu$ and use squared Euclidean distance cost.
By \eqref{eq:induced_op_disintegration} we get for $u \in \Leb^2(\mu)$ and $y \in \{0,1\}$ that $(T u)(y) = \int_{\SpX} u(x) \,\diff\pi(x|y) = \int_{\SpX} u(x) \,\diff\mu(x)$.
Since $(Tu)(y)$ does not depend on $y$, we have $\widetilde{T}^{\veps} = G^{\veps}_{\nu\nu} \circ T = T$ and $T$ has the constant integration kernel $\tilde{t}^{\veps}(x,y) = 1$ for $(x,y) \in \spt(\pi)$.

\begin{figure}[bt]
  \centering
  		\definecolor{C0}{HTML}{1f77b4}
		\definecolor{C1}{HTML}{ff7f0e}
		\begin{tikzpicture}[x=1cm, y=1cm, scale=3.2]
			\tikzstyle{thin}=[draw, color=C0, line width=0.5];
			\tikzstyle{mid}=[draw, color=C0, line width=2];
			\tikzstyle{thick}=[draw, color=C0, line width=3.5];
			\tikzstyle{annotation}=[draw, line width=1];
			
			\pgfmathsetmacro{\h}{0.8}
			\pgfmathsetmacro{\os}{0.15}
			\pgfmathsetmacro{\sp}{0.2}
			\pgfmathsetmacro{\rb}{0.03}
			\pgfmathsetmacro{\rm}{0.02}
			\pgfmathsetmacro{\rs}{0.005}
			\pgfmathsetmacro{\heady}{0.92}

			\fill[color=black]  (-\sp,0) circle[radius=\rb];
			\fill[color=black] (-\sp,\h) circle[radius=\rb];
			\node[] at (-\sp-0.1,0) {$0$};
			\node[] at (-\sp-0.1,\h) {$1$};
			\node[] at (-\sp-0.1,0.5*\h) {$\nu$};

			\begin{scope}[xshift=0, yshift=0]
				\node[] at (0.5,\heady) {$\widetilde{T}^{\varepsilon}$};
				
				\draw[mid, color=black] (0,-\sp) -- (1,-\sp);
				\foreach \x in {0, 1}
					\draw[style=annotation] (\x, -\sp+0.05) -- (\x, -\sp-0.05) node[below] {$\x$};
				\node[] at (0.5, -\sp-0.15) {$\mu$};
				
				\draw[style=mid] (0,0) -- (1,0) (0,\h) -- (1,\h);
			\end{scope}
			
			\begin{scope}[xshift=1.2cm, yshift=0]
				\node[] at (0.5,\heady) {$T^{N}$};
				
				\draw[fill]
				(0.029684677097052282, -\sp) circle[radius=\rm]
				(0.06794445565580626, -\sp) circle[radius=\rm]
				(0.12741319647284527, -\sp) circle[radius=\rm]
				(0.1589986749301918, -\sp) circle[radius=\rm]
				(0.2390952971190599, -\sp) circle[radius=\rm]
				(0.2715377339757441, -\sp) circle[radius=\rm]
				(0.2999697881019155, -\sp) circle[radius=\rm]
				(0.36962233906779784, -\sp) circle[radius=\rm]
				(0.43584001290170404, -\sp) circle[radius=\rm]
				(0.5569385485983827, -\sp) circle[radius=\rm]
				(0.623221577495637, -\sp) circle[radius=\rm]
				(0.681888149168618, -\sp) circle[radius=\rm]
				(0.7479988073961981, -\sp) circle[radius=\rm]
				(0.8145579830400664, -\sp) circle[radius=\rm]
				(0.8732712436640183, -\sp) circle[radius=\rm]
				(0.8972790486294402, -\sp) circle[radius=\rm];
				
				\foreach \x in {0, 1}
					\draw[style=annotation] (\x, -\sp+0.05) -- (\x, -\sp-0.05) node[below] {$\x$};
				\node[] at (0.5, -\sp-0.15) {$\mu^N$};
				
				\draw[color=C1, fill]
				(0.029684677097052282, \h*0) circle[radius=\rm]
				(0.06794445565580626, \h*1) circle[radius=\rm]
				(0.12741319647284527, \h*0) circle[radius=\rm]
				(0.1589986749301918, \h*1) circle[radius=\rm]
				(0.2390952971190599, \h*0) circle[radius=\rm]
				(0.2715377339757441, \h*0) circle[radius=\rm]
				(0.2999697881019155, \h*0) circle[radius=\rm]
				(0.36962233906779784, \h*1) circle[radius=\rm]
				(0.43584001290170404, \h*0) circle[radius=\rm]
				(0.5569385485983827, \h*1) circle[radius=\rm]
				(0.623221577495637, \h*1) circle[radius=\rm]
				(0.681888149168618, \h*0) circle[radius=\rm]
				(0.7479988073961981, \h*1) circle[radius=\rm]
				(0.8145579830400664, \h*1) circle[radius=\rm]
				(0.8732712436640183, \h*1) circle[radius=\rm]
				(0.8972790486294402, \h*0) circle[radius=\rm]
				
				;
			\end{scope}
			
			\begin{scope}[xshift=2.4cm, yshift=0]
				\node[] at (0.5,\heady) {$G^{\varepsilon}_{\nu^N\nu^N} \circ T^{N}$};
				
				\draw[fill]
				(0.029684677097052282, -\sp) circle[radius=\rm]
				(0.06794445565580626, -\sp) circle[radius=\rm]
				(0.12741319647284527, -\sp) circle[radius=\rm]
				(0.1589986749301918, -\sp) circle[radius=\rm]
				(0.2390952971190599, -\sp) circle[radius=\rm]
				(0.2715377339757441, -\sp) circle[radius=\rm]
				(0.2999697881019155, -\sp) circle[radius=\rm]
				(0.36962233906779784, -\sp) circle[radius=\rm]
				(0.43584001290170404, -\sp) circle[radius=\rm]
				(0.5569385485983827, -\sp) circle[radius=\rm]
				(0.623221577495637, -\sp) circle[radius=\rm]
				(0.681888149168618, -\sp) circle[radius=\rm]
				(0.7479988073961981, -\sp) circle[radius=\rm]
				(0.8145579830400664, -\sp) circle[radius=\rm]
				(0.8732712436640183, -\sp) circle[radius=\rm]
				(0.8972790486294402, -\sp) circle[radius=\rm];
				
				\foreach \x in {0, 1}
				\draw[style=annotation] (\x, -\sp+0.05) -- (\x, -\sp-0.05) node[below] {$\x$};
				\node[] at (0.5, -\sp-0.15) {$\mu^N$};
				
				\draw[color=C1, fill]
				(0.029684677097052282, 0) circle[radius=\rm]
				(0.029684677097052282, \h) circle[radius=\rs]
				(0.06794445565580626, 0) circle[radius=\rs]
				(0.06794445565580626, \h) circle[radius=\rm]
				(0.12741319647284527, 0) circle[radius=\rm]
				(0.12741319647284527, \h) circle[radius=\rs]
				(0.1589986749301918, 0) circle[radius=\rs]
				(0.1589986749301918, \h) circle[radius=\rm]
				(0.2390952971190599, 0) circle[radius=\rm]
				(0.2390952971190599, \h) circle[radius=\rs]
				(0.2715377339757441, 0) circle[radius=\rm]
				(0.2715377339757441, \h) circle[radius=\rs]
				(0.2999697881019155, 0) circle[radius=\rm]
				(0.2999697881019155, \h) circle[radius=\rs]
				(0.36962233906779784, 0) circle[radius=\rs]
				(0.36962233906779784, \h) circle[radius=\rm]
				(0.43584001290170404, 0) circle[radius=\rm]
				(0.43584001290170404, \h) circle[radius=\rs]
				(0.5569385485983827, 0) circle[radius=\rs]
				(0.5569385485983827, \h) circle[radius=\rm]
				(0.623221577495637, 0) circle[radius=\rs]
				(0.623221577495637, \h) circle[radius=\rm]
				(0.681888149168618, 0) circle[radius=\rm]
				(0.681888149168618, \h) circle[radius=\rs]
				(0.7479988073961981, 0) circle[radius=\rs]
				(0.7479988073961981, \h) circle[radius=\rm]
				(0.8145579830400664, 0) circle[radius=\rs]
				(0.8145579830400664, \h) circle[radius=\rm]
				(0.8732712436640183, 0) circle[radius=\rs]
				(0.8732712436640183, \h) circle[radius=\rm]
				(0.8972790486294402, 0) circle[radius=\rm]
				(0.8972790486294402, \h) circle[radius=\rs];
			\end{scope}
			
			\begin{scope}[xshift=3.6cm, yshift=0]	
				\node[] at (0.5,\heady) {$G^{\varepsilon}_{\nu^N\nu^N} \circ T^{N} \circ T^{\mu}_N$};
				
				\draw[mid, color=black] (0,-\sp) -- (1,-\sp);
				\foreach \x in {0, 1}
					\draw[style=annotation] (\x, -\sp+0.05) -- (\x, -\sp-0.05) node[below] {$\x$};
				\node[] at (0.5, -\sp-0.15) {$\mu$};
				
				\draw[color=C1, fill, opacity=0.5]
				(0.0, 0) -- (0.029684677097052282, 0+\os) circle[radius=\rm] -- (0.0625, 0)
				(0.0, \h) -- (0.029684677097052282, \h-\os) circle[radius=\rs] -- (0.0625, \h)
				(0.0625, 0) -- (0.06794445565580626, 0+\os) circle[radius=\rs] -- (0.125, 0)
				(0.0625, \h) -- (0.06794445565580626, \h-\os) circle[radius=\rm] -- (0.125, \h)
				(0.125, 0) -- (0.12741319647284527, 0+\os) circle[radius=\rm] -- (0.1875, 0)
				(0.125, \h) -- (0.12741319647284527, \h-\os) circle[radius=\rs] -- (0.1875, \h)
				(0.1875, 0) -- (0.1589986749301918, 0+\os) circle[radius=\rs] -- (0.25, 0)
				(0.1875, \h) -- (0.1589986749301918, \h-\os) circle[radius=\rm] -- (0.25, \h)
				(0.25, 0) -- (0.2390952971190599, 0+\os) circle[radius=\rm] -- (0.3125, 0)
				(0.25, \h) -- (0.2390952971190599, \h-\os) circle[radius=\rs] -- (0.3125, \h)
				(0.3125, 0) -- (0.2715377339757441, 0+\os) circle[radius=\rm] -- (0.375, 0)
				(0.3125, \h) -- (0.2715377339757441, \h-\os) circle[radius=\rs] -- (0.375, \h)
				(0.375, 0) -- (0.2999697881019155, 0+\os) circle[radius=\rm] -- (0.4375, 0)
				(0.375, \h) -- (0.2999697881019155, \h-\os) circle[radius=\rs] -- (0.4375, \h)
				(0.4375, 0) -- (0.36962233906779784, 0+\os) circle[radius=\rs] -- (0.5, 0)
				(0.4375, \h) -- (0.36962233906779784, \h-\os) circle[radius=\rm] -- (0.5, \h)
				(0.5, 0) -- (0.43584001290170404, 0+\os) circle[radius=\rm] -- (0.5625, 0)
				(0.5, \h) -- (0.43584001290170404, \h-\os) circle[radius=\rs] -- (0.5625, \h)
				(0.5625, 0) -- (0.5569385485983827, 0+\os) circle[radius=\rs] -- (0.625, 0)
				(0.5625, \h) -- (0.5569385485983827, \h-\os) circle[radius=\rm] -- (0.625, \h)
				(0.625, 0) -- (0.623221577495637, 0+\os) circle[radius=\rs] -- (0.6875, 0)
				(0.625, \h) -- (0.623221577495637, \h-\os) circle[radius=\rm] -- (0.6875, \h)
				(0.6875, 0) -- (0.681888149168618, 0+\os) circle[radius=\rm] -- (0.75, 0)
				(0.6875, \h) -- (0.681888149168618, \h-\os) circle[radius=\rs] -- (0.75, \h)
				(0.75, 0) -- (0.7479988073961981, 0+\os) circle[radius=\rs] -- (0.8125, 0)
				(0.75, \h) -- (0.7479988073961981, \h-\os) circle[radius=\rm] -- (0.8125, \h)
				(0.8125, 0) -- (0.8145579830400664, 0+\os) circle[radius=\rs] -- (0.875, 0)
				(0.8125, \h) -- (0.8145579830400664, \h-\os) circle[radius=\rm] -- (0.875, \h)
				(0.875, 0) -- (0.8732712436640183, 0+\os) circle[radius=\rs] -- (0.9375, 0)
				(0.875, \h) -- (0.8732712436640183, \h-\os) circle[radius=\rm] -- (0.9375, \h)
				(0.9375, 0) -- (0.8972790486294402, 0+\os) circle[radius=\rm] -- (1.0, 0)
				(0.9375, \h) -- (0.8972790486294402, \h-\os) circle[radius=\rs] -- (1.0, \h);
				
				\draw[thin]
				(0.0,\h*1) -- (0.0625,\h*1)
				(0.0625,\h*0) -- (0.125,\h*0)
				(0.125,\h*1) -- (0.1875,\h*1)
				(0.1875,\h*0) -- (0.25,\h*0)
				(0.25,\h*1) -- (0.3125,\h*1)
				(0.3125,\h*1) -- (0.375,\h*1)
				(0.375,\h*1) -- (0.4375,\h*1)
				(0.4375,\h*0) -- (0.5,\h*0)
				(0.5,\h*1) -- (0.5625,\h*1)
				(0.5625,\h*0) -- (0.625,\h*0)
				(0.625,\h*0) -- (0.6875,\h*0)
				(0.6875,\h*1) -- (0.75,\h*1)
				(0.75,\h*0) -- (0.8125,\h*0)
				(0.8125,\h*0) -- (0.875,\h*0)
				(0.875,\h*0) -- (0.9375,\h*0)
				(0.9375,\h*1) -- (1.0,\h*1);
				
				\draw[thick]
				(0.0,\h*0) -- (0.0625,\h*0)
				(0.0625,\h*1) -- (0.125,\h*1)
				(0.125,\h*0) -- (0.1875,\h*0)
				(0.1875,\h*1) -- (0.25,\h*1)
				(0.25,\h*0) -- (0.3125,\h*0)
				(0.3125,\h*0) -- (0.375,\h*0)
				(0.375,\h*0) -- (0.4375,\h*0)
				(0.4375,\h*1) -- (0.5,\h*1)
				(0.5,\h*0) -- (0.5625,\h*0)
				(0.5625,\h*1) -- (0.625,\h*1)
				(0.625,\h*1) -- (0.6875,\h*1)
				(0.6875,\h*0) -- (0.75,\h*0)
				(0.75,\h*1) -- (0.8125,\h*1)
				(0.8125,\h*1) -- (0.875,\h*1)
				(0.875,\h*1) -- (0.9375,\h*1)
				(0.9375,\h*0) -- (1.0,\h*0);
			\end{scope}
			
		\end{tikzpicture}
  \caption{Counterexample for convergence with single blurring. 
    The leftmost panel shows the kernel of the true transfer operator $\widetilde{T}^{\veps}$ w.r.t.~$\mu\otimes\nu$ (it is uniform).
    The second panel shows the kernel of the discrete observed operator $T^N$ w.r.t.~$\mu^N \otimes \nu^N$ (here $\nu^N = \nu$).
    The third panel shows how the kernel changes when the single blur operator is applied. A small amount of the mass that was previously mapped to the top row, is now mapped to the bottom row and vice versa.
    The rightmost panel shows the kernel of the fully assembled operator estimate w.r.t.~$\mu\otimes\nu$.
    The kernel oscillates between $2(1-\varsigma)$ and $2\varsigma$ and therefore does not converge towards the kernel of $\widetilde{T}^{\veps}$ in the $\Leb^2$-norm as $N \to \infty$.
    }
  \label{fig:blurring_counterexample}
\end{figure}

We will now construct the empirical operator.
Let $((x_i, y_i))_{i=1}^N \subset \spt(\pi)$ be i.i.d.~samples from the distribution $\pi$.
Similar to \Cref{def:T_extension}, we extend the single-blurred operator to $\Leb(\mu) \to \Leb(\nu)$ by
\begin{align*}
  \widetilde{T}^{A,\veps}_N 
  &= 
  (T^{\nu}_N)^{\ast} \circ G^{\veps}_{\nu_N\nu_N} \circ T_N \circ T^{\mu}_N.
\end{align*}
Similar to $T^{A,\veps}_N$, this operator is not meant to be constructed numerically, but merely serves as an object for theoretical analysis.
For simplicity assume $\nu_N = \nu$, the example also works in the general case $\nu_N \neq \nu$ but is more tedious.
With this assumption, $T^{\nu}_N$ is the identity and $G^{\veps}_{\nu_N\nu_N} = G^{\veps}_{\nu\nu}$, which we compute next.
By the symmetry of the $\nu$ self-transport problem, the corresponding dual $\bar{\alpha}$ (see \Cref{prop:self_transport}) is constant on $\spt(\nu)$. 
Using this together with the property $\int_{\SpX} \kernel{\nu\nu}{\veps}(y,y') \,\diff \nu(y) = 1$, allows us to compute $\kernel{\nu\nu}{\veps}$ straight from its definition \eqref{eq:def_etk}.
For $y, y' \in \spt(\nu) = \{0,1\}$ we get
\begin{align*}
  \kernel{\nu\nu}{\veps}(y,y')
  &= 
  \begin{cases} 
    2(1-\varsigma) & \tn{if } y=y' \\ 
    2\varsigma & \tn{if } y\neq y'
  \end{cases}
  \quad \tn{for}\quad
  \varsigma=\frac{1}{1+\exp(1/\veps)}.
\end{align*}
Finally, we need to determine $T^{\mu}_N$, which is by definition induced by the optimal unregularized transport plan of $\mu$ to $\mu_N$. Since $\mu$ has a Lebesgue density (it is the Lebesgue measure on $[0,1]$), the transport plan has a density $\kernel{\mu\mu_N}{}$ w.r.t.~$\mu\otimes\mu_N$.
Since we are in one dimension, the unregularized transport problem amounts to sorting the input, i.e.~the $q$-th quantile of one measure is assigned to the $q$-th quantile of the other for all $q \in [0,1]$, see \cite[Chapter 2]{santambrogio2015optimal} for more details.
W.l.o.g.~assume that $x_i$ are sorted in strictly increasing order (in particular there are no duplicates, which holds almost surely).
Then the point $x_i$ is transported to the interval~$\big(\tfrac{i-1}{N}, \tfrac{i}{N}\big)$, i.e.~for $x \in \SpX$ and $x_i \in \spt(\mu_N)$ we have
\begin{align*}
  \kernel{\mu\mu_N}{}(x, x_i)
  &=
  \begin{cases}
    N & \tn{if } x \in \big(\tfrac{i-1}{N}, \tfrac{i}{N}\big), \\
    0 & \tn{otherwise}.
  \end{cases}
\end{align*}
The normalization factor $N$ stems from the fact that integrating $\kernel{\mu\mu_N}{}(\cdot, x_i)$ over the interval $\big(\tfrac{i-1}{N}, \tfrac{i}{N}\big)$ with respect to the restricted Lebesgue measure $\mu$ must yield the density 1, since $\kernel{\mu\mu_N}{} \cdot \mu\otimes\mu_N$ is a transport plan.
Putting everything together, we get that $\widetilde{T}^{A,\veps}_N$ has the integration kernel
\begin{align*}
  \tilde{t}^{A,\veps}_N(x,y)
  &=
  \int_{\SpX^2} \kernel{\nu\nu}{\veps}(y,y') \kernel{\mu\mu_N}{}(x,x') \,\diff\pi_N(x',y')
  \\&=
  \begin{cases}
    2(1 - \varsigma) 
    & \tn{if } y = y_i \tn{ where } i \tn{ is uniquely defined by } x \in \big(\tfrac{i-1}{N}, \tfrac{i}{N}\big)
    \\
    2 \varsigma 
    & \tn{otherwise}.
  \end{cases}
\end{align*}
A visualization of $\tilde{t}^{\veps}$ and $\tilde{t}^{A,\veps}_N$ is depicted in \Cref{fig:blurring_counterexample}.
From here it is easy to see that $\big\|{\tilde{t}^{\veps} - \tilde{t}^{A,\veps}_N}\big\|_{\Leb^2(\mu\otimes\nu)}$ does not converge to $0$ as $N \to \infty$, indeed the norm does not even depend on $N$.

The interpretation as to why convergence fails in this case is that with single blurring, in order to estimate each $\nu$-slice $\tilde{t}^{A,\veps}_N(x,\cdot)$, we only use a single sample. 
This is sufficient if the transfer operator is deterministic (as shown in \cite{EntropicTransfer22}), since there is only a single value to approximate.
For probabilistic transfer operators however, this example shows that single blurring does not suffice.
With double blurring all samples in the proximity of $x$ contribute to the approximation $t^{A,\veps}_N(x,\cdot)$ of $t^{\veps}(x,\cdot)$, which allows for convergence in a much more general setting, as shown in \Cref{subsec:quant_cvg}.

\subsection{Numerical workflow and algorithms} \label{subsec:workflow}
For numerical data analysis on dynamical systems the objects of interest are the finite-dimensional operator $T^{\veps}_N$ on the discrete data, and its kernel $t^{\veps}_N$, which can be evaluated on the whole domain.
In this section we outline the corresponding steps and a typical workflow. Some tutorial code and code to reproduce the figures in this article can be found online.\footnote{\url{https://github.com/OTGroupGoe/StochasticETO}}
For simplicity, we consider the stationary setting of Section \ref{subsec:stationary}. Adaptations to the non-staionary setting are straight-forward.

\paragraph{Matrix representations of operators and transport plans.}
Assume that we are given samples $(x_i,y_i)_{i=1}^N$ from $\pi$ on $\SpX \times \SpX$. For now, assume that all $(x_i)_i$ and $(y_i)_i$ are distinct. This holds almost surely if $\pi$ has no atoms and \Cref{rem:duplicates} explains why we may ignore duplicate points even when they occur.

Our goal is to study $T^{\veps}_N$ on $L^2(\mu_N) \to L^2(\mu_N)$ numerically.
For this we equip $L^2(\mu_N)$ with the canonical orthonormal basis given by functions $(\ones_{x_i})_i$ where
$$\ones_{x_i}(x_j) \assign \begin{cases} \sqrt{N} & \tn{if } x_j=x_i, \\ 0 & \tn{otherwise,} \end{cases}$$
and analogously we equip $L^2(\nu_N)$ with the basis $(\ones_{y_i})_i$. For these bases one then has
$$T_N \ones_{x_i} = \ones_{y_i}$$
and therefore the matrix representation $\bm{T}_N$ of $T_N$ in these bases is simply the $N \times N$ identity matrix.

Transport plans between $\mu_N$ and itself (and analogously for $\nu_N$ and itself, or $\mu_N$ and $\nu_N$) can be represented by non-negative matrices $\bm{\pi} \in \R^{N \times N}$ where each row and column of $\bm{\pi}$ sums to $1/N$. The matrix $\bm{\pi}$ corresponding to the optimal entropic plan in \eqref{eq:def_entrOT_primal} has the form given by \eqref{eq:EntropicPD},
$$\bm{\pi}_{i,j}=\kernel{\mu_N\mu_N}{\veps}(x_j,x_i)/N^2 \qquad \tn{where} \qquad \kernel{\mu_N\mu_N}{\veps}(x_j,x_i)=\exp([\alpha(x_j)+\beta(x_i)-c(x_j,x_i)]/\veps)$$
and the factor $1/N^2$ accounts for the masses that $\mu_N$ assigns to the points $x_i$ and $x_j$.
Note that we choose here the convention that the first (row) index of $\bm{\pi}$ corresponds to the output, the second (column) index to the input space, as is standard for matrices, whereas for integration kernels throughout the paper we have adopted the convention that the first argument corresponds to the input and the second argument to the output space.
$\bm{\pi}$ can be obtained efficiently with the Sinkhorn algorithm (see \cite{peyre2019computational} and references therein). Given this matrix $\bm{\pi}$, the matrix representation $\bm{G}_{\mu_N\mu_N}^\veps$ of the operator $G_{\mu_N\mu_N}^\veps$ in the basis $(\ones_{x_i})_i$ is then given by $\bm{G}_{\mu_N\mu_N}^\veps=N \cdot \bm{\pi}$, such that each row and column sums to 1 (and analogously for $G_{\nu_N\nu_N}^\veps$ and $G_{\nu_N\mu_N}^\veps$).

We can therefore obtain a matrix representation of $T_N^\veps=G_{\nu_N\mu_N}^\veps T_N G_{\mu_N\mu_N}^\veps$ by solving two entropic optimal transport problems and then multiplying the two matrices $\bm{T}^\veps_N=\bm{G}_{\nu_N\mu_N}^\veps\,\bm{G}_{\mu_N\mu_N}^\veps$ (of course we may skip the identity matrix $\bm{T}_N$). 
For very large $N$ it would be computationally costly to calculate $\bm{T}^{\veps}_N$ explicitly as a dense matrix, instead one can define it as an abstract linear operator that multiplies by the two blur operators in succession (e.g.~using \texttt{scipy.}\allowbreak{}\texttt{sparse.}\allowbreak{}\texttt{linalg.}\allowbreak{}\texttt{LinearOperator} as in interface).
To save on memory, one may additionally use an abstract representation of $\bm{G}^{\veps}_{\cdot\cdot}$, see also the paragraph on large-scale computations in Section \ref{sec:convection}.

In this fashion, we are now able to construct a numerical representation of $T^\veps_N$ and subsequently extract its dominant eigenpairs or singular values and vectors.

\begin{remark}
\label{rem:duplicates}
If two points $x_i$ and $x_j$, $i\neq j$ are identical, one can merge them into a single point with increased weight in the vector representation of $\mu_N$ and adopt the corresponding basis vector $\ones_{x_i}(x_k)=\sqrt{N/2}$ for $x_k=x_i=x_j$ and zero otherwise. Alternatively, it is possible to ignore this collision and to simply keep both copies $x_i$ and $x_j$: Since the transport cost function $c(x_i,y)=c(x_j,y)$ will be equal for all $y$, by virtue of the entropic regularization, the rows (or columns, depending on convention) in the optimal entropic transport matrix $\pi$ corresponding to $x_i$ and $x_j$ will also be equal. Consequently, all eigen or singular vectors of the matrix representation of $T^{N,\veps}$ will be equal in the rows corresponding to $x_i$ and $x_j$. The increased weight of this point is accounted for by the fact that this row appears twice in the matrix representation. In the same way additional duplicates or duplicates in $(y_i)_i$ may be ignored.
\end{remark}

\paragraph{Sweeping analysis of spectrum.}
When studying a new dynamical system, as a starting point we recommend to compute and visualize the dominant part of the spectrum of $T^\veps_N$ over a range of different $\veps$, such as in Figure \ref{fig:embedding_nonrot} and in \cite{EntropicTransfer22}. It is advisable to start with large $\veps$ and then to decrease $\veps$ gradually to speed up calculation by virtue of $\veps$-scaling techniques, as described in \cite{SchmitzerScaling2019}.
For `very large' $\veps$, $T_N^\veps$ is typically oversmoothed and all eigenvalues except for the one corresponding to the stationary density are close to $0$. For `very small' $\veps$, one typically finds many eigenvalues with absolute value close to $1$, indicative of discretization artefacts. The range where $\veps$ is `very large' depends on the length scales of the system $T$, the regime of `very small' $\veps$ additionally depends on the number $N$ of available data points, and the complexity and (intrinsic) dimensionality of $T$, $\mu$, and $\nu$.
The results of Section \ref{subsec:quant_cvg} provide some guidance for this relation and it is illustrated by the examples below and those in \cite{EntropicTransfer22}. Part of the motivation for this sweeping analysis is to find out where these regimes lie for the given system.

If the original system $T$ exhibits a spectral gap (by which we mean a gap between the absolute values of any two adjacent eigenvalues in the ordered spectrum) due to a time scale separation (see for instance \cite{BittracherTransitionManifolds2018} and references therein) then this spectral gap will also be visible in $T^\veps_N$ for intermediate $\veps$ (if $N$ is sufficiently high).
Such a gap is clearly visible in Figure \ref{fig:embedding_nonrot} (see also \cite[Figures 3, 4, 5, and 7]{EntropicTransfer22}).

\paragraph{Spectral embedding.}
When an intermediate $\veps$ with a spectral gap has been identified, one can use spectral embedding \cite{Coifman2006} to visualize the samples. For instance, sample $x_i$ may be represented by the tuple $(u_k(x_i))_{k \in I}$ where $(u_k)_k$ denotes the eigenfunctions of $T^{\veps}_N$ and $I$ is some index set. We assume that eigenfunctions are enumerated by decreasing absolute value of the eigenvalues. A typical choice is $I=\{2,\ldots,K\}$, i.e.~we skip the trivial constant eigenfunction for eigenvalue $\lambda_1=1$, corresponding to the uniform stationary density, and go up to the $K$-th one. For visualization in 2 or 3 dimensions, one may experiment with different choices of indices (cf.~\cite{Koltai_2020} for some examples). If $\bm{u}_k$ is the $k$-th eigenvector of the matrix representation $\bm{T}^\veps_N$, then $u_k(x_i)=\sqrt{N} \cdot (\bm{u}_k)_i$, where the latter denotes the $i$-th entry of the vector $\bm{u}_k$. The factor $\sqrt{N}$ accounts for the discrepancy of the naive Euclidean inner product on $\R^N$ (or $\C^N$) and the one in $L^2(\mu_N)$ where each point is weighted with a factor $1/N$, since one has $u(x)=\sum_{i=1}^N \ones_{x_i}(x) \cdot \bm{u}_i$ (see above for the definition of the basis functions $\ones_{x_i}$).

\paragraph{Out-of-sample extension.}
As discussed, the kernel $t_N^\veps$ of $T_N^\veps$ can be extended beyond $((x_i,y_j))_{i,j}$ via the regularity properties of entropic dual transport potentials.
For $(x,y) \in \SpX$ we have by \Cref{def:op_stationary}
$$t_N^\veps(x,y) = (\kernel{\mu_N\mu_N}{\veps}:\pi_N:\kernel{\nu_N\mu_N}{\veps})(x,y)
=\sum_{i=1}^N \kernel{\mu_N\mu_N}{\veps}(x,x_i)\, \kernel{\nu_N\mu_N}{\veps}(y_i,y).$$
To evaluate $\kernel{\mu_N\mu_N}{\veps}(x,x_i)=\exp([\alpha(x)+\beta(x_i)-c(x,x_i)]/\veps)$ for $x \notin \{x_j\}_j$ one first computes $\alpha(x)$ via \eqref{eq:Sinkhorn},
$$\alpha(x)=-\veps \log\left( \frac1N \sum_{j=1}^N \exp([\beta(x_j)-c(x,x_j)]/\veps) \right).$$
For $\kernel{\nu_N\mu_N}{\veps}$ one proceeds analogously.
For evaluating $T^\veps_N u$ at some point $x \in \SpX$ one uses
\begin{align*}
    (T^\veps_Nu)(x) & = (G^\veps_{\nu_N\mu_N} T_N G^\veps_{\mu_N \mu_N} u)(x)
    = \int_{\SpX} k_{\nu_N \mu_N}^\veps(z,x) (T_N G^\veps_{\mu_N \mu_N} u)(z) \diff \nu_N(z) \\
    & = \frac1N \sum_{i,j=1}^N k_{\nu_N \mu_N}^\veps(y_i,x) (\bm{G}^\veps_{\mu_N \mu_N})_{ij} \bm{u}_j
\end{align*}
where $\bm{u}$ is again the vector representation of the function $u \in L^2(\mu_N)$ with $u(x_j)=\sqrt{N}\cdot \bm{u}_j$.

\subsection{Stochastic shift on torus} \label{subsec:num_1d}
\begin{figure}[bt]
  \centering
  \includegraphics[width = 16cm]{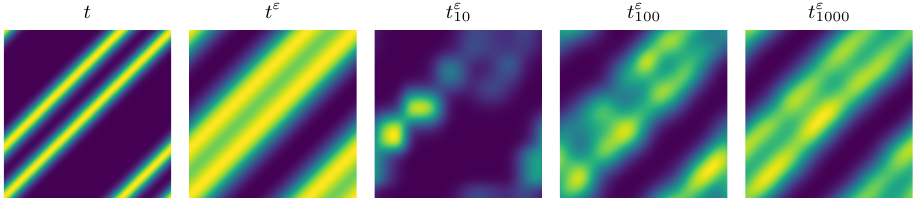}
  \caption{Integral kernels $t$, $t^\veps$ and $t_N^\veps$ for the system \eqref{eq:torus_pi} for $\sigma = 0.05$, $\veps = 0.01$, and various $N$. Yellow indicates high values, dark blue indicates zero; color scales are adjusted to each panel separately for better visibility.}
  \label{fig:general_plot}
\end{figure}

\paragraph{Problem description.}
Similar as in \cite[Section 6.1]{EntropicTransfer22} and \cite[Section 6.1]{beier2024transfer} we use the 1-torus as a transparent toy example to illustrate key properties of stochastic entropic transfer operators.
We focus here on the analysis of the kernel $t^\veps_N$ (see \eqref{eq:TEpsKernelN}) as a function on $\SpX \times \SpX$ and its convergence toward $t^\veps$, as captured by \Cref{prop:prob_bias_bound,prop:prob_var_bound}.
The extension $t^{A,\veps}_N$ (see \eqref{eq:tNAeps}) is obtained from $t^\veps_N$ by an additional piecewise constant approximation step. While this is important to study spectral convergence of $T_N^{A,\veps}$, we ignore this additional step here for simplicity.

Let $\SpX \assign \mathbb{R}/\mathbb{Z}$ be the 1-torus, and $\pi \in \mathcal{P}(\SpX\times \SpX)$ be such that
\begin{align} \label{eq:torus_pi}
  (X,Y) \sim \pi \quad \Leftrightarrow \quad
  \begin{cases}
    X &\sim \mu \assign \mathcal{U}(\SpX)\\
    Y|X=x &\sim \frac{1}{2} \, \widetilde{\mathcal{N}}(x,\sigma^2) + \frac{1}{2} \, \widetilde{\mathcal{N}}(x+0.3,\sigma^2)
  \end{cases}
\end{align}
where $\mathcal{U}(\SpX)$ denotes the uniform distribution and $\widetilde{\mathcal{N}}(m,\sigma^2)$ denotes the wrapped Gaussian distribution with mean $m$ and standard deviation $\sigma$, concretely, for the canonical projection $f : \R \to \R/\Z$ we have $\widetilde{\mathcal{N}}(m,\sigma^2)=f_\# \mathcal{N}(m,\sigma^2)$. By symmetry the marginal distribution $\nu$ of $Y$ is also uniform, and the self-transport potential \eqref{eq:self_Sinkhorn} for $\kernel{\mu\mu}{\veps}=\kernel{\nu\nu}{\veps}$ is constant.
The kernels $t$, $t^\veps$ and $t^\veps_N$ are illustrated in \Cref{fig:general_plot}.

Figure \ref{fig:line_plot} shows $L^2(\mu\otimes\nu)$ distances between $t$, $t^\veps$, and $t^\veps_N$ for various parameters. The distance between $t$ and $t^\veps$ is calculated via discretisation on a regular grid, the others are approximated via Monte-Carlo integration. Plots involving the empirical $t_N^\veps$ show averages over 100 simulations.
We discuss the observations below.

\begin{figure}[bt]
  \centering
  \includegraphics[width = \textwidth]{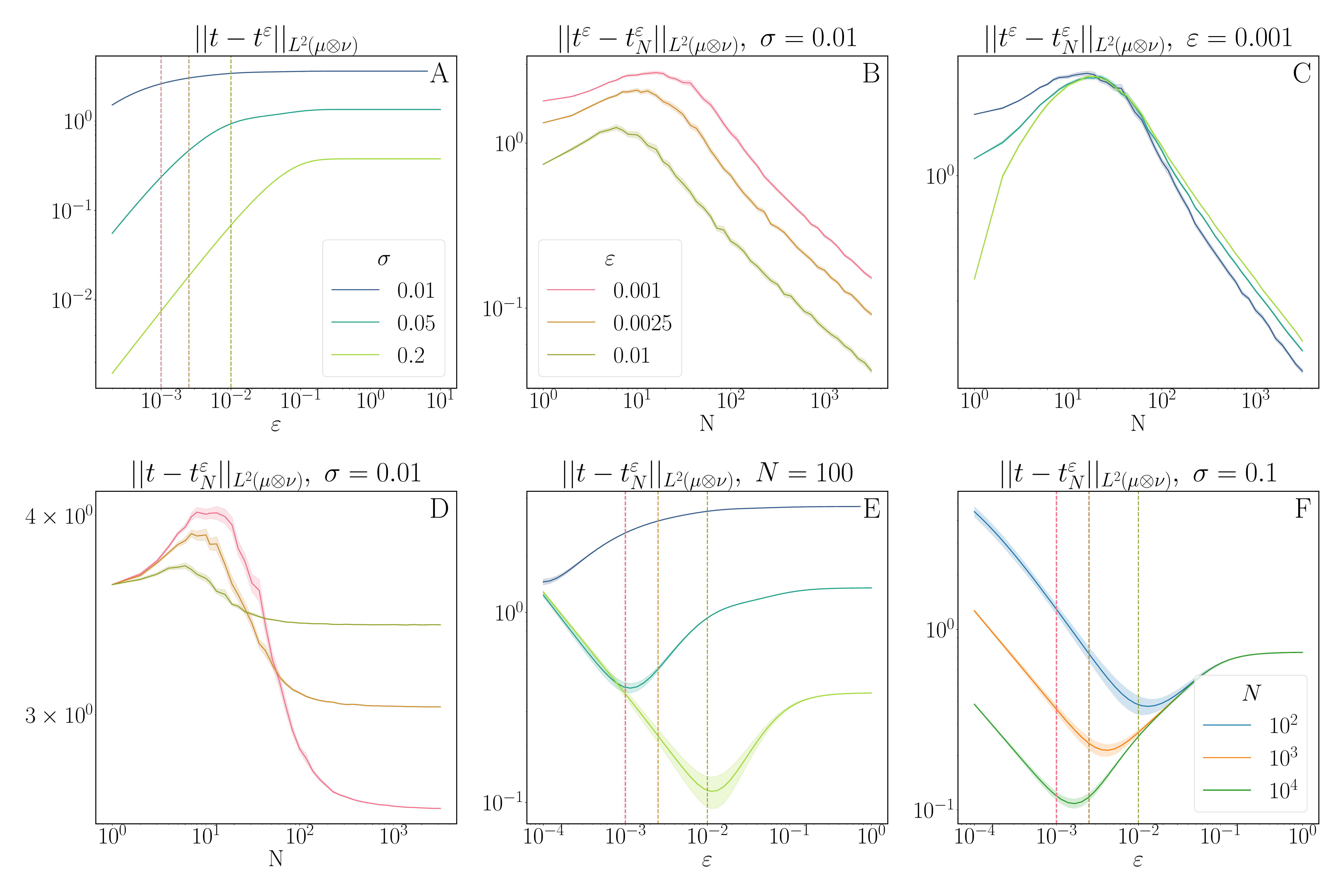}
  \caption{$L^2$ distances between different integral kernels and parameters for the system \eqref{eq:torus_pi}. Colors for encoding $\sigma$ and $\veps$ are consistent in all panels. Vertical dashed lines indicate values of $\veps$ used in other panels. Plots that involve empirical data show the estimated mean with $95\%$ confidence interval (based on $100$ simulations), all y-axes are in log scale.}
  \label{fig:line_plot}
\end{figure}

\paragraph{$\norm{t-t^\veps}_{L^2(\mu\otimes\nu)}$ for varying $\veps$.}
$t^{\veps}_N$ provides an empirical approximation of the regularized $t^\veps$, not of $t$ itself. Hence it is important to understand the difference between $t$ and $t^\veps$, which can be interpreted as the \emph{bias} introduced by convolution with the self-transport kernels. This is shown in Figure \ref{fig:line_plot}A.
Note that the bias is higher when $\pi$ is more concentrated (i.e.~it increases as $\sigma$ decreases).
As predicted by \Cref{prop:Teps2T} the discrepancy vanishes as $\veps \to 0$. Intuitively, this is due to the fact that $G_{\mu\mu}^\veps$ converges to the identity operator as $\veps \to 0$.
On the other hand, as $\veps \to \infty$, the optimal entropic self-transport plans approach the product measures, and consequently $t^\veps$ converges to the constant function $1$.
This is reflected by the plateaus in the plot, which lie at values $\norm{t-1}_{L^2(\mu\otimes\nu)}$.

\paragraph{$\norm{t^\veps - t^\veps_N}_{L^2(\mu\otimes\nu)}$ for varying $N$ and different $\veps$.}
Figure \ref{fig:line_plot}B shows the discrepancy between the empirical $t^{\veps}_N$ and the regularized $t^\veps$, which is related to the \emph{variance} of our estimator. We expect the variance to converge to 0 as $N\to\infty$ approximately with rate $O(1/\sqrt{N})$ by \Cref{prop:prob_bias_bound,prop:prob_var_bound}.
For small $N$, all three lines first increase. For $N=1$ one finds that $t_1^\veps$ is constant and equal to 1, for small $N>1$, $t_N^\veps$ first becomes `spiky' (cf.~Figure \ref{fig:general_plot}) (which has a higher $L^2$ distance to $t^\veps$ than the uniform $t_1^\veps$). Eventually $N$ is sufficiently high to cover the region where $t^\veps$ is substantially non-zero with small blobs on the length scale $\sqrt{\veps}$ and the error starts to decrease.
Therefore, this trend reversal takes longer as $\veps$ decreases (in analogy to a kernel density estimator).
Generally, the variance decreases as the regularization $\veps$ increases.

\paragraph{$\norm{t^\veps - t^\veps_N}_{L^2(\mu\otimes\nu)}$ for varying $N$ and different $\sigma$.}
Figure \ref{fig:line_plot}C is similar to Figure \ref{fig:line_plot}B, but shows different $\sigma$ instead.
The reason for the non-monotonicity is as before. Note that for small $N$ the error is larger for small $\sigma$ (since the true distribution is more concentrated and thus further from the kernel which is constant $1$), whereas for large $N$ this behaviour is reversed (since for small $\sigma$, $t^\veps$ is substantially non-zero only on a smaller region of $\SpX \times \SpX$).

\paragraph{$\norm{t - t^\veps_N}_{L^2(\mu\otimes\nu)}$ for varying $N$, $\veps$, and $\sigma$.}
Figure \ref{fig:line_plot}D-F illustrate the combination of bias and variance in the discrepancy between the empirical $t^\veps_N$ and the true $t$ and the resulting bias-variance trade-off.
In Figure \ref{fig:line_plot}D, with increasing $N$ the error decreases earlier for large $\veps$ (small variance) but then plateaus at a higher value (high bias), whereas for small $\veps$ it takes longer to decrease (high variance) but ultimately reaches a lower level (small bias).
In Figure \ref{fig:line_plot}E, with increasing $\veps$ the error first decreases (decreasing variance), and ultimately increases (increasing bias). For large $\sigma$, the decrease lasts longer, since the bias is lower (see Figure \ref{fig:line_plot}A).
Likewise, in Figure \ref{fig:line_plot}F, the error first decreases and then increases with increasing $\veps$. The decrease reaches the lowest level for large $N$, since the variance is lower.

\paragraph{Out-of-sample extension of eigenfunctions.}
\begin{figure}[bt]
  \centering
  \includegraphics[width = 16cm]{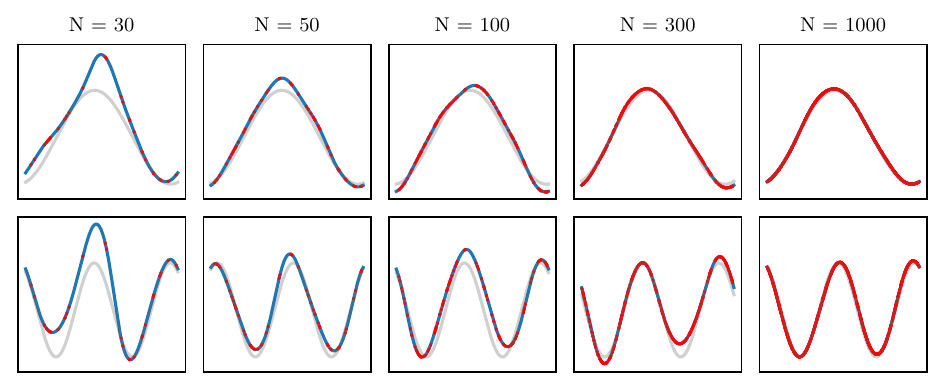}
  \caption{The second and fourth dominant eigenfunctions (real part)for the system \eqref{eq:torus_pi_simple} for $\sigma = 0.01$, $\veps = 0.01$ on $(x_i)_i$ (red points), out-of-sample extension (blue line), and true eigenfunctions (grey line, aligned over the ambiguous phase shift).}
  \label{fig:out_of_sample}
\end{figure}
Figure \ref{fig:out_of_sample} shows the real parts of the dominant non-trivial eigenfunctions of $T^\veps_N$ and their out-of-sample extension via \eqref{eq:out_of_sample}, indicating that the discrete eigenfunctions converge to the limiting eigenfunctions and that the extension yields a meaningful continuous interpolation.
Note that here we use a simplified model instead of \cref{eq:torus_pi} for better understanding:
\begin{align}
\label{eq:torus_pi_simple}
    X \sim \mathcal{U}(\mathcal{X}), \qquad Y|X=x\sim \widetilde{\mathcal{N}}(x,\sigma^2)
\end{align}
By arguments similar to \cite[Proposition 3]{EntropicTransfer22} one can prove that the eigenfunction of $T^\varepsilon$ are Fourier modes, which is consistent with our simulations, as can be seen from Figure \ref{fig:out_of_sample}.
We will also use this method in Section \ref{sec:convection}.

\paragraph{Higher dimensions.}
\begin{figure}[bt]
  \centering
  \includegraphics[width = 16cm]{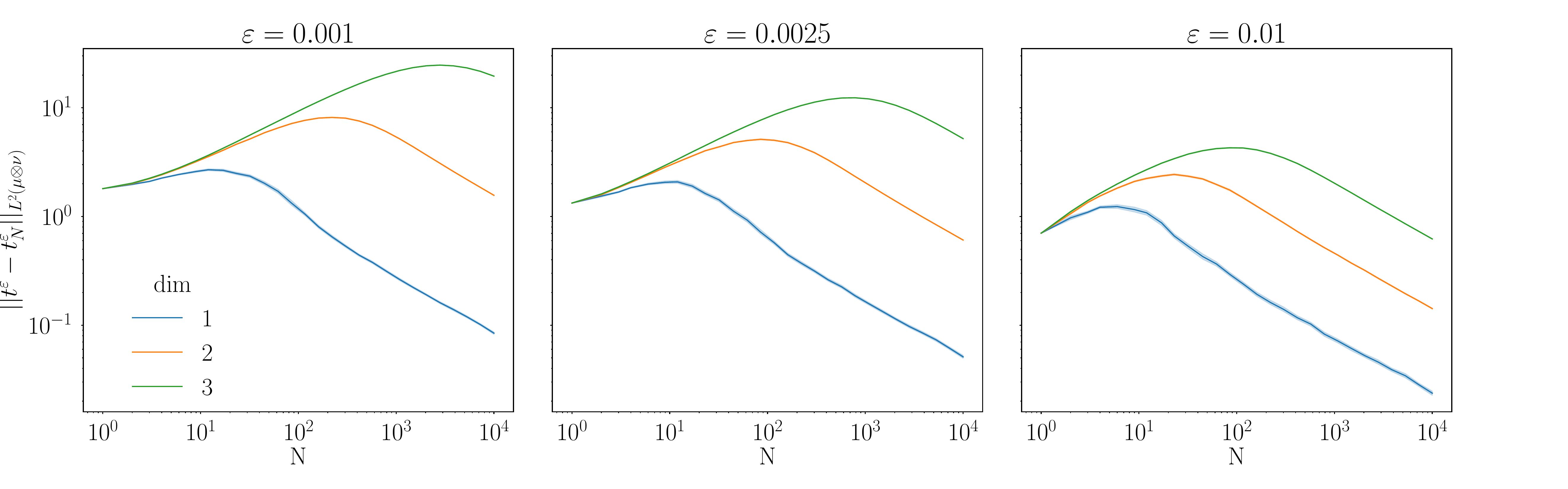}
  \caption{Variance $\norm{t^\veps - t^\veps_N}_{L^2(\mu\otimes\nu)}$ for the system \eqref{eq:torus_pi} for different dimensions $m$ and regularization $\veps$. Plots show estimated mean with with $95\%$ confidence interval. $\sigma = 0.01$.}
  \label{fig:dim_lineplot}
\end{figure}
Now let $\SpX \assign \mathbb{R}^d/\mathbb{Z}^d$ and $\tilde{\pi} \in \prob(\mathbb{R}/\mathbb{Z} \times \mathbb{R}/\mathbb{Z})$ be as in \eqref{eq:torus_pi}. Set $\pi \in \prob(\SpX \times \SpX)$ to
\begin{align*}
  \pi = \tilde{\pi}\otimes\Big(\mathcal{U}(\mathbb{R}/\mathbb{Z} \times \mathbb{R}/\mathbb{Z})\Big)^{\otimes (d-1)},
\end{align*}
i.e.~for a random variable pair $(X,Y)$ with joint law $\pi$, the first components follow the `shift and blur' pattern of \eqref{eq:torus_pi} and the other dimensions are simply uniformly distributed. Figure \ref{fig:dim_lineplot} shows the $L^2$ distance between $t^\veps$ and $t^N_\veps$ for varying $N$ and different dimension $d$. In accordance with \Cref{prop:prob_var_bound,prop:prob_bias_bound} the distance decreases with rate $O(N^{-1/2})$, but the constant increases with $d$.

\subsection{Comparison with Ulam's method} \label{sec:Ulam}
In this subsection we provide a comparison between entropic transfer operators and Ulam's method. For this we consider a shift on the torus, embedded into a higher-dimensional ambient space, with additional noise. While initially a simple system, with increasing dimensionality and noise level it becomes more and more difficult to extract its dynamic structure from data.
\paragraph{System description.}
Let $\mathcal{X} \assign \mathbb{R}/\mathbb{Z}$, and $\pi\in \mathcal{P}(\mathcal{X}\times\mathcal{X})$ such that
\begin{align*}
    (\tilde{X},\tilde{Y}) \sim \pi \quad \Leftrightarrow \quad
    \begin{cases}
        \tilde X &\sim \mathcal{U}(\SpX)\\
        \tilde Y| \tilde X &= \, \tilde X + \frac{1}{5}.
    \end{cases}
\end{align*}
Let $\emb : \mathcal{X} \to \mathbb{R}^2 ,\  \tilde x \mapsto (\cos{(2\pi\tilde x)} , \sin (2\pi\tilde x))$ be the embedding of the one-torus into $\mathbb{R}^2$. For $d \in \N$, $d\geq 2$, let $f_d : \mathbb{R}^2 \to \mathbb{R}^d ,\ (x_1,x_2) \mapsto (x_1,x_2,0,...,0)$ be the canonical embedding of $\mathbb{R}^2$ in $\mathbb{R}^d$. 
Sample $A_{n,k}, B_{n,k} \sim \mathcal{N}(0,1)$ for $k \in \llbracket 1,10\rrbracket$ and $n \in \llbracket 1, d\rrbracket$, denote $\mathcal{F}_{n}(x) := \sum_{k=1}^{10}\frac{A_{n,k}}{k}\cos(2\pi k x) + \frac{B_{n,k}}{k}\sin (2\pi k x)$. I.e.~$\mathcal{F}_n$ are randomly weighted combinations of the first $10$ Fourier modes. Additionally by $R \in \text{SO}(d)$ denote a (uniform) random rotation operator, and let $\tau = 0.2$ be an arbitrarily chosen damping parameter. 
Then
\begin{align*}
    \Emb & : \mathcal{X}\to\mathbb{R}^d, & \tilde{X} & \mapsto R\Big( f_d\circ \emb(\tilde{X}) + (\tau^2 \mathcal{F}_1(\tilde{X}), \tau^2 \mathcal{F}_2(\tilde{X}), \tau \mathcal{F}_3(\tilde{X}),..., \tau \mathcal{F}_d(\tilde{X}))\Big)
\end{align*}
maps the torus to a rotated, distorted circle in $\mathbb{R}^d$. Note that it has less distortion in the two first dimensions. See \Cref{fig:3d_ring} for an example when $d = 3$.
\begin{figure}[bt]
    \centering
	\includegraphics[width = 6cm]{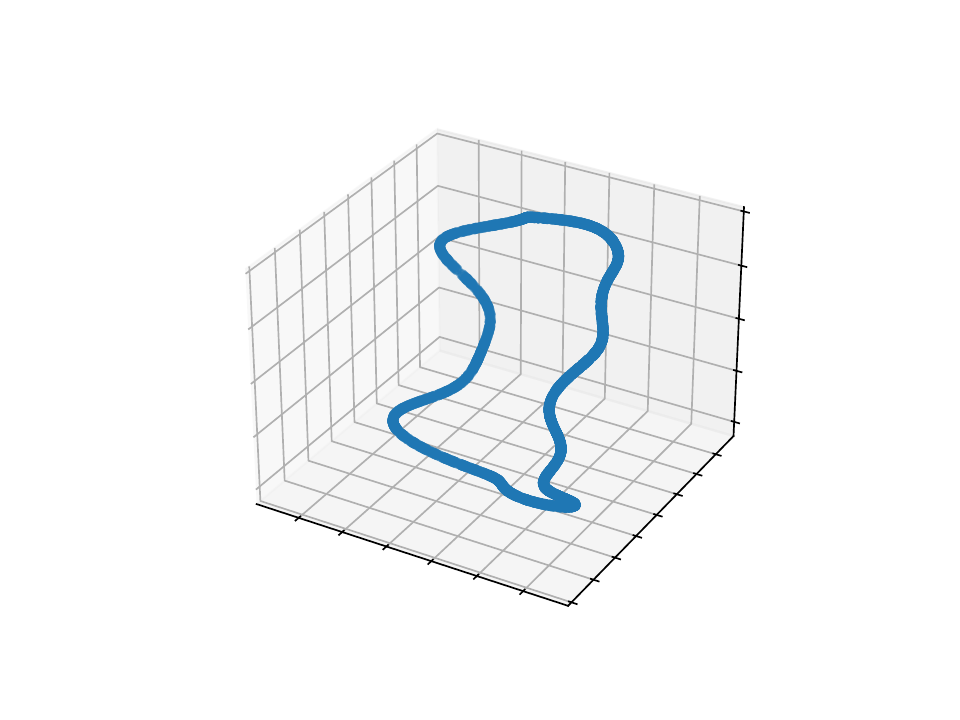}
    \caption{Image of operator Emb when $d = 3$}
    \label{fig:3d_ring}
\end{figure}
Finally we add some random normal noise. For $Z, Z' \sim \mathcal{N}(0, I_d)$ and $\sigma > 0$, define
\begin{align*}
    (X,Y) \assign \left(\Emb(\tilde{X}) + \sigma Z, \Emb(\tilde{Y}) + \sigma Z'\right) \in \mathbb{R}^d.
\end{align*}
We take the joint law of these two random variables as the joint law $\pi$ of our dynamical system.
The blur is implemented differently as in Section \ref{subsec:num_1d} and only applied after the embedding, so that the dimension of the support of the law of $(X,Y)$ is $2d$, making the system more challenging to analyze, especially as $d$ and $\sigma$ are increased.

We expect that the spectrum of the resulting system has eigenvalues approximately at angles $2\pi \cdot k/5$, $k \in \Z$, according to the shift from $\tilde{X}$ to $\tilde{Y}$, with eigenfunctions being approximately given by Fourier modes along the ring. Due to the noise, the eigenvalues for modes with higher frequency will have a damped amplitude.

\paragraph{Results.}
In \Cref{fig:Ulam_ETO} we show simulation results for the estimated eigenvalues by Ulam's method and entropic transfer operators for varying parameters:
\begin{itemize}
    \item For entropic transfer operators we choose $10$ values for the entropic regularization constant $\veps$ equally spaced from $0.01$ to $0.1$. For Ulam's method we discretize $\R^d$ by equal-sized hyper cubes with side length $2\sqrt{\veps}$, such that the spatial resolution of both methods is roughly comparable.
    \item We choose the noise parameter $\sigma \in \{0.05, 0.1, 0.2, 0.4\}$.
    \item We choose the dimension $d \in \{2, 10\}$, note that $d = 2$ means we only have the main dimensions. 
    \item We calculate the first $10$ leading eigenvalues.
\end{itemize}
\begin{figure}[bt]
    \centering
	\includegraphics[width = 10cm]{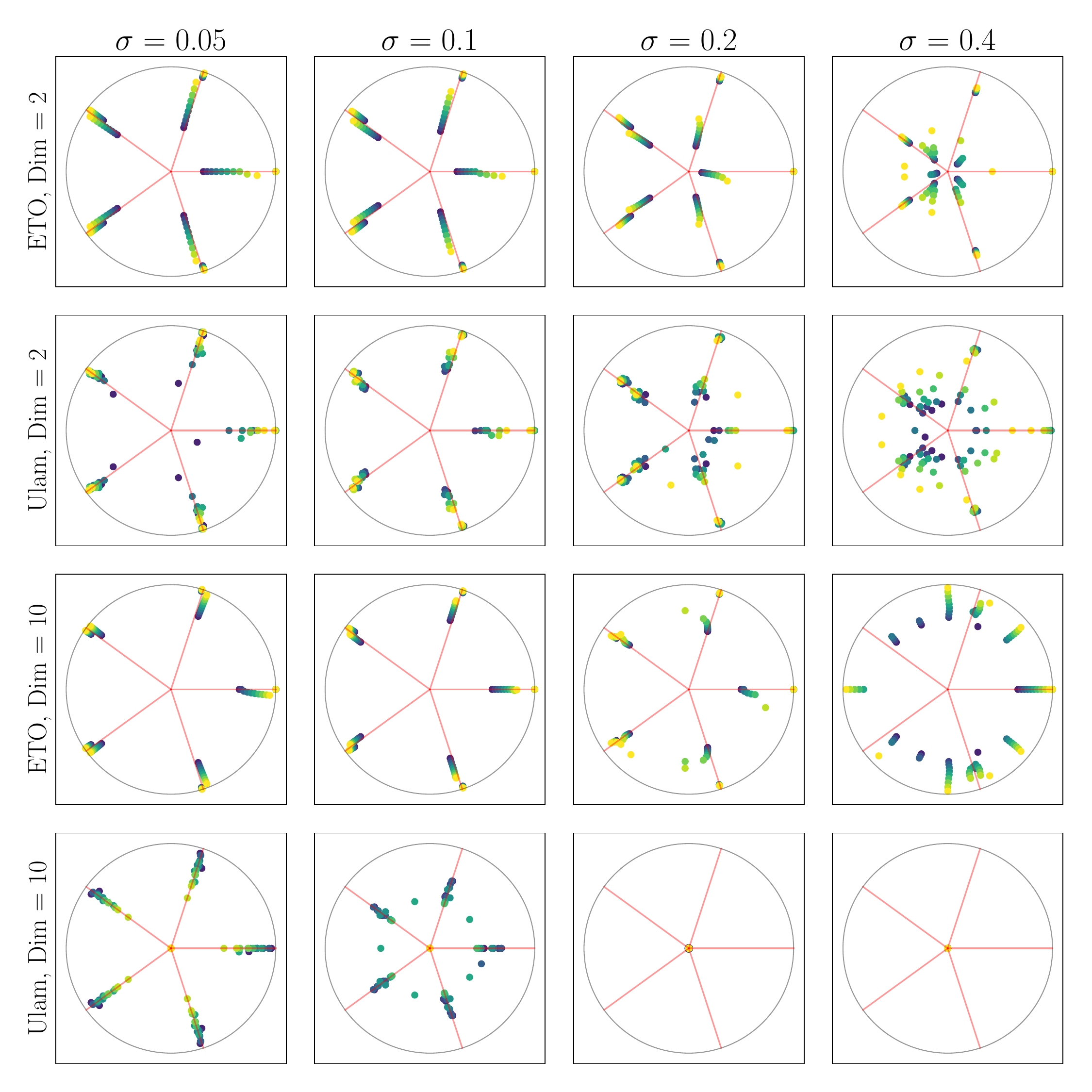}
    \caption{First 10 leading eigenvalues in the example of Section \ref{sec:Ulam}, estimated by Ulam's method and entropic transfer operators, for varying $\veps$ (encoded by color, dark blue is larger), noise level $\sigma$, and ambient dimension $d$. $N = 500$. See text for more details.}
    \label{fig:Ulam_ETO}
\end{figure}
For $d = 2$ and small noise $\sigma$ both approaches work well. For the entropic transfer operator, the spectrum depends more smoothly on $\veps$, the separation between the first 5 and the second 5 eigenvalues is cleaner, and the spectrum is still clean at $\sigma=0.2$ where Ulam's method already exhibits a few artifacts.
For $d = 10$, entropic transfer operators still work well up to including $\sigma=0.2$, whereas Ulam's method already produces a corrupted spectrum at $\sigma=0.05$ (which completely collapses to $0$ for $\sigma=0.2,~0.4$). In conclusion, the mesh-free regularization of entropic transfer operators, compared to the binning in Ulam's method, introduces less artefacts and is more robust in higher dimensions.

\subsection{Rayleigh-B\'enard convection data} \label{sec:convection}

\paragraph{Description of the dataset.}
In this section we consider an example from fluid dynamics and analyze a dataset of a turbulent Rayleigh--B\'enard convection that was previously studied in \cite{Koltai_2020} (see \cite{weiss2011DataStationary,weiss2011DataRotating,weiss2011Aspect} for more details on the experiment, the data, and the physical motivation, concretely, we consider run 1003261, as described in \cite{weiss2011DataStationary}).
The experimental setup consists of a cylinder filled with water that is heated at the bottom and cooled at the top. The aspect ratio between the diameter of the cylinder $D$ and its height $L$ was $D/L=0.5$. The average water temperature is $40$\textdegree{}C, the temperature difference between top and bottom plate is $19.9$\textdegree{}C.
The fluid temperature is measured with $24$ thermistors that are embedded into the cylinder wall in three layers at heights $L/4$, $L/2$, and $3L/4$, $8$ per layer, with evenly distributed angles along the cylinder circumference. These 24 measurements give a rough characterization of the fluid temperature profile.
Measurements were recorded approximately once every $3.4$s for a total duration of approximately $12$ days. In total $300\,362$ sets of $24$ measurements have been recorded.
Data from the first one or two hours is discarded to be sure that only such data points were considered in the analysis where the system was in a statistical equilibrium.

Now $\SpX$ is (a compact subset) of $\R^{24}$, a single state $x \in \SpX$ is given by $x=(t_{l,k})_{l \in \{\tn{b,m,t}\}, k \in \{0,\ldots,7\}}$
where $t_{l,k}$ denotes the temperature measurement (in degrees Celsius) in layer $l$ (the letters stand for bottom, middle, and top) at the azimuthal position $\theta_k \assign 2\pi \cdot k/8$.
In the selected regime of physical parameters large-scale circulation rolls form to transport heat through the cylinder and typical configurations are either a single large roll state (SRS) or a `double roll' state (DRS).
The experimental setup and the single rolls are sketched in \cite[Figure 1]{Koltai_2020}.

A rough physical summary of the system state is given by specifying whether the system is in a SRS or a DRS and by the roll orientation. This can be approximately extracted from a measurement $(t_{l,k})_{l \in \{\tn{b,m,t}\}, k \in \{0,\ldots,7\}}$ as follows (see \cite{weiss2011DataStationary} for more details): first, a cosine curve is fitted to the temperatures in each layer, i.e.~a least squares regression problem is solved to approximate
\begin{equation} \label{eq:convection_fit}
  t_{l,k} \approx \frac{1}{8}\sum_{k'=0}^7 t_{l,k'} + A_l \cos\left(\theta_k-\psi_l \right)
\end{equation}
for each layer $l \in \{\tn{b,m,t}\}$, where $A_l$ and $\psi_l$ are the amplitude and phase of the cosine profile respectively.
Examples of a temperature measurement and corresponding fitted curves are shown in \cite[Figure 1]{Koltai_2020}.
In the SRS, the three phases $\psi_l$ are expected to be similar, in the DRS the phase difference between bottom and top layers should approximately be $\pi$. For simplicity, we will assume that some $x \in \SpX$ is in SRS if the absolute phase difference between top and bottom is less than $\pi/2$ (up to multiples of $2\pi$) and in DRS otherwise.
Of course this will misclassify some states, including such that are neither in SRS nor DRS. More sophisticated classification rules are discussed in \cite{weiss2011DataStationary}. However, for the purpose of demonstrating that the subsequent transfer operator analysis is consistent with the physical interpretation of states, the above simplified rule is sufficient.

Let $(z_t)_{t=1}^T \subset \SpX$ be the sequence of measured states. We extract from this a collection of observed transitions $((x_i,y_i))_{i=1}^N$ by setting
\begin{equation} \label{eq:convection_xy_extraction}
  x_i \assign z_{t_0+s \cdot i} \quad \tn{and} \quad y_i\assign z_{t_0+s \cdot i+l}
\end{equation}
for admissible values of $i$. Here $t_0 \in \N$ can be used to discard initial measurements, before the system has reached statistical equilibrium. The value $s \in \N$ is a sub-sampling parameter (stride) which can be chosen $>1$ to reduce computational load, and to reduce the dependency between the considered samples (however the latter is not necessary as we expect the system to be ergodic). Finally, $l \in \N$ denotes the time lag between the entries of each pair $(x_i,y_i)$ and setting $l>1$ effectively corresponds to studying the $l$-th power of the transfer operator relative to the setting $l=1$.

\paragraph{Overview of \cite{Koltai_2020} and comparison.}
In \cite{Koltai_2020} diffusion maps are used to obtain a two-dimensional embedding of the points $(z_t)_t$.
This embedding is approximately disk-shaped and physically meaningful in the sense that the radius and azimuthal coordinate of an embedded point reflect the amplitude $A_{\tn{m}}$ and orientation $\psi_{\tn{m}}$ of the SRS states and DRS states are clustered near the center of the disk.
The embedding does not yet consider information on the observed transitions $((x_i,y_i))_i$. This information is processed in a subsequent step by applying Ulam's method on the embedding, i.e.~the embedding space is partitioned into boxes and the transition rates between the boxes are estimated.
These rates represent a discretized regularized version of the transfer operator. Due to the high dimension of $\SpX$ it is not possible to apply Ulam's method directly on the original data.

This pipeline requires one to carefully choose several parameters.
The bandwidth for the diffusion maps must be selected (there are established procedures for this choice). Then, if suitable eigenvectors for a meaningful low-dimensional embedding can be identified, a box discretization scale for Ulam's method must be chosen. If the number of boxes is too high, the number of available samples might not suffice to robustly estimate all relevant transition rates. If the number of boxes is too low, many $(x_i,y_i)$ might end up in the same box and thus dynamical features of the system remain invisible. Changes and distortions in the embedding will directly influence the estimation of densities and rates.

As we will demonstrate below, with entropic transfer operators one can perform an analysis similar to the combination of diffusion maps with Ulam's method as in \cite{Koltai_2020} but simpler in the following sense:
the estimation of the transfer operator is performed directly on the original data (observed transitions between the temperature measurements, $((x_i,y_i))_i$), not on an intermediate embedding. No box discretization is necessary.
Only a single parameter (the entropic regularization $\veps$) must be set. Similar to the bandwidth in diffusion maps, this parameter has a transparent interpretation as a spatial blur scale, results are relatively robust with respect to small changes, and reasonable values can be determined via a preliminary exploratory analysis (see below). It is also feasible to perform the analysis at multiple scales.

A common computational bottleneck of the approach in \cite{Koltai_2020} and entropic transfer operators is the handling of large matrices of size $N \times N$, related to running Sinkhorn's algorithm to assemble the discrete entropic transfer operator, or to extract dominating eigenpairs from graph Laplacian or transfer matrices.
As a remedy, in \cite{Koltai_2020} the diffusion map embedding is only computed on a small subset of the samples and then extrapolated to the full dataset via an out-of-sample extension. The complexity of Ulam's method depends primarily on the number of boxes, which is much smaller than the number of available samples.
Below we will show that subsampling and out-of-sample extension can also be applied to entropic transfer operators. This is particularly useful to obtain a first understanding of the dataset and to determine an appropriate (range of) value(s) for $\veps$.
We will also demonstrate that with modern GPU hardware and suitable software it is now also possible to perform the full analysis on the whole dataset in reasonable time and without memory issues.

After obtaining a meaningful embedding and an estimate of the transfer operator, \cite{Koltai_2020} carefully discusses the physical interpretation of these results. Such an analysis is beyond the scope of the present article.

\paragraph{Exploratory analysis.}

\begin{figure}[bt]
  \centering
  \includegraphics[width=\textwidth]{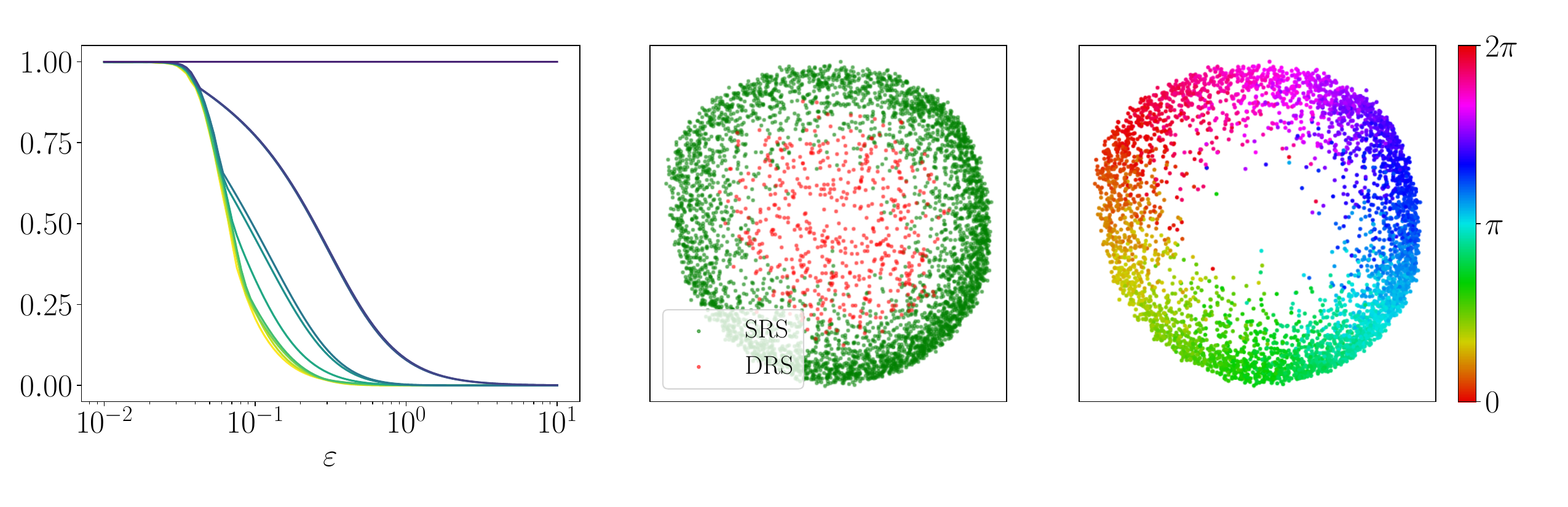}

  \caption{Left: 10 largest (in absolute value) eigenvalues of $T_N^\veps$ for different $\veps$ for $t_0=2000$, $s=60$, $l=1$. Middle and right: spectral embedding of points $(x_i)_i$ based on two sub-dominant eigenvectors $(u_2,u_3)$ of $T_N^\veps$ at $\veps=0.1$. Color represents SRS/DRS classification (middle) and SRS roll orientation $\psi_{\tn{m}}$ (right, only SRS states are shown). }
  \label{fig:embedding_nonrot}
\end{figure}

\begin{figure}[bt]
  \centering
  \includegraphics[width = 0.7\textwidth]{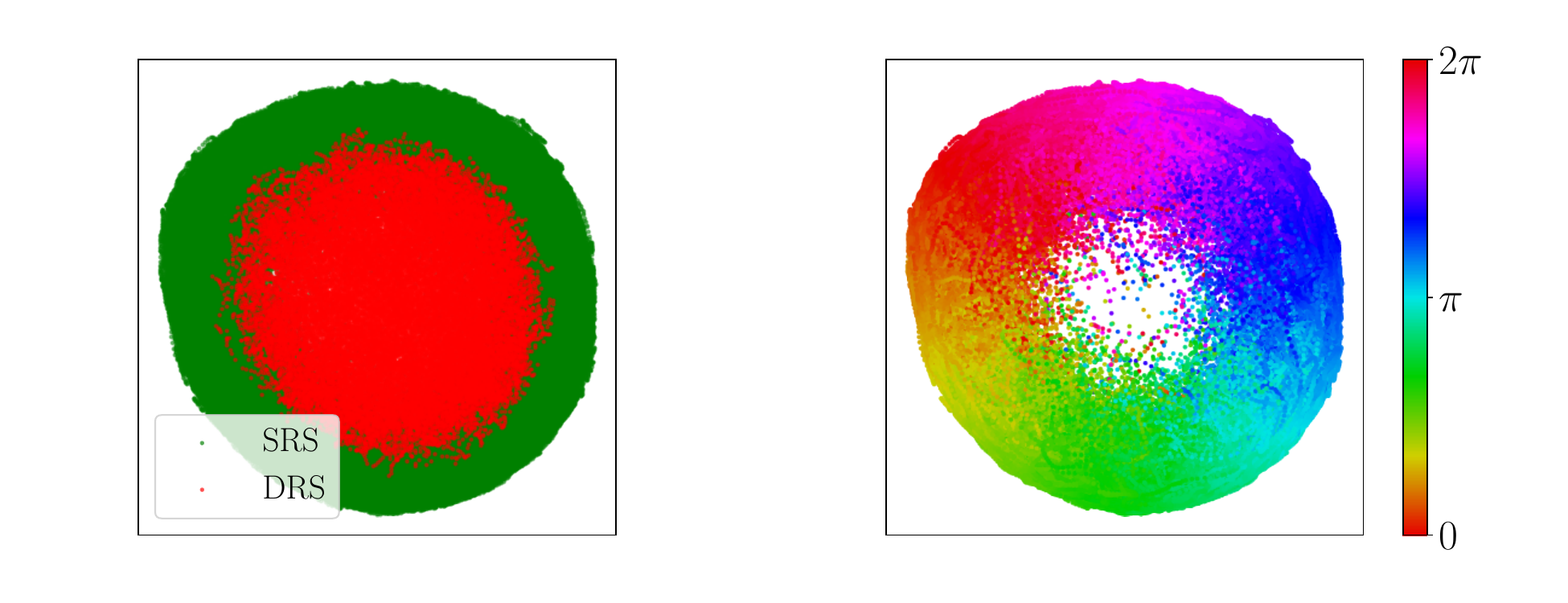}
  \caption{Out-of-sample extension of the embedding of Figure \ref{fig:embedding_nonrot} to the full dataset.} 
  \label{fig:embedding_nonrot_oos}
\end{figure}

We now perform a first exploratory analysis of the data by considering a small subset. We set $t_0=2000$, $s=60$, and $l=1$ in \eqref{eq:convection_xy_extraction}, resulting in $N=5007$ observed transitions.
On this small subset we compute $T^\veps_N$ (in the stationary variant) for various $\veps \in [10^{-2},1]$ and extract the 10 largest (in absolute value) eigenvalues $(\lambda_1,\ldots,\lambda_{10})$. We find that all these eigenvalues are real, consistent with an approximate diffusion with zero drift. The eigenvalues are shown in Figure \ref{fig:embedding_nonrot}, left.
As expected, $\lambda_1$ is always 1, corresponding to the equilibrium distribution. For $\veps=10^{-2}$ all extracted eigenvalues are approximately one, suggesting that at this small blur scale the system has many approximately disconnected subsystems. For $\veps=1$ the blur has reduced all non-dominant eigenvalues to close to zero. (See \cite[Sections 5 and 6.1]{EntropicTransfer22} for a detailed discussion on the effect of $\veps$ on the spectrum.)
The two largest sub-dominant eigenvalues $(\lambda_2,\lambda_3)$ are almost equal and they decay substantially slower than the rest.
We now fix $\veps=0.1$ where $\lambda_2 \approx \lambda_3 \approx 0.77$ are still somewhat close to 1, and well separated from the smaller eigenvalues ($\lambda_4 \approx 0.48$), and use the corresponding eigenfunctions $u_2,u_3$ for a spectral embedding (shown in Figure \ref{fig:embedding_nonrot}, middle and right).
Similar to \cite[Figure 4]{Koltai_2020} the embedding is disk shaped, with DRS states near the center, and SRS states on a ring around with the angle encoding the roll orientation.
Also similar to \cite[Figure 3]{Koltai_2020} the higher order eigenmodes seem to roughly correspond to those of a disk, with higher azimuthal and radial modes (not shown here). The obtained higher order modes are also roughly consistent with the modes for the transfer operator estimated via Ulam's method shown in \cite[Figures 8,10]{Koltai_2020}, but a precise correspondence is probably not to be expected due to the substantially different numerical strategy.
Using the out-of-sample extension described in Section \ref{subsec:out_of_sampling} we can then add the full dataset into the embedding. As shown in Figure \ref{fig:embedding_nonrot_oos} this extension is also consistent with the physical parameters of the states.

Similar to \cite[Figure 2]{Koltai_2020}, for smaller $\veps$ (e.g.~0.06) and small lag and stride in \eqref{eq:convection_xy_extraction}, some higher order modes would occasionally capture transient events, where the trajectory briefly departs quite far from the dominant disc structure.
For some values of $\veps$, stride $s$ and lag $l$ we found a few isolated points in the spectral embeddings, occasionally also captured as spurious eigenmodes. Upon closer inspection we found that in these points some of the temperature values were set to 10, which was caused by a faulty relay. In total, $17$ datapoints seem to be affected by this. Fortunately, such spurious eigenmodes are easy to identify, since the corresponding eigenvector will usually be approximately binary, with most values close to zero, and only a few isolated values being substantially non-zero. A related phenomenon was also reported in \cite[Section 6.2]{EntropicTransfer22} for isolated datapoints in regions where the attractor has a low density.

\paragraph{Large-scale computations.}

\begin{figure}[bt]
  \centering
  \includegraphics[width = \textwidth]{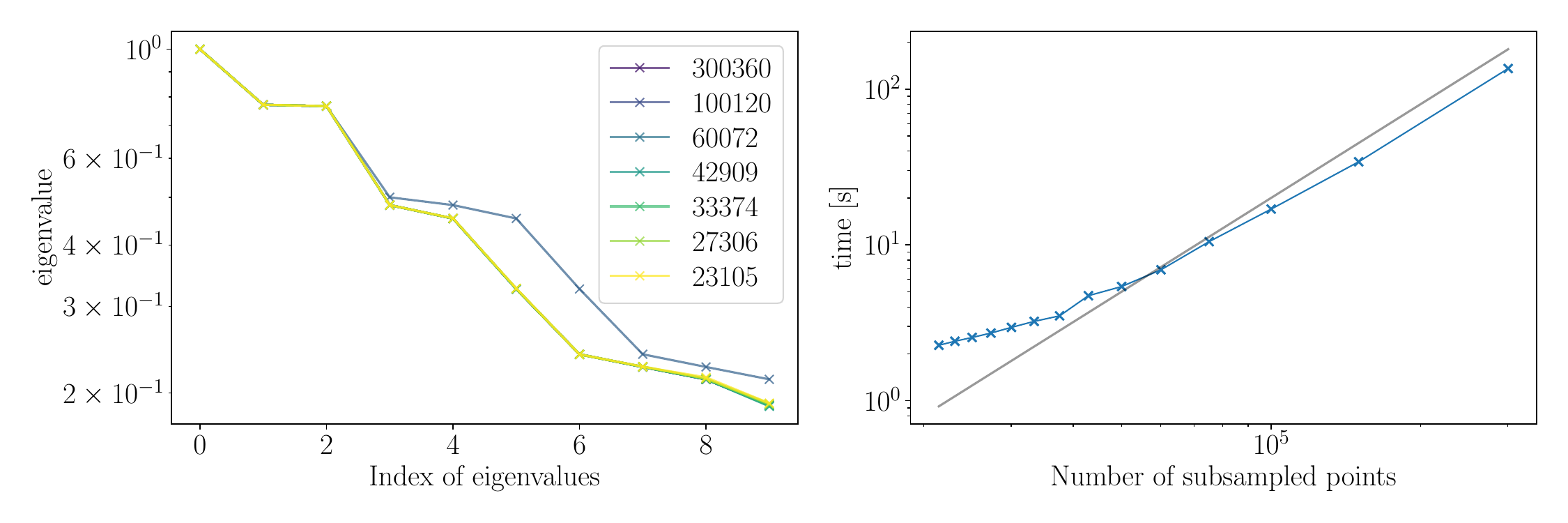}
  \caption{Left: 10 largest eigenvalues for different $N$ represented by color (obtained via $s \in \llbracket 1, 14\rrbracket$ in \eqref{eq:convection_xy_extraction}). All eigenvalues are real. Right: runtimes for the computation of these spectra. Grey line shows $O(N^2)$ trend line. See text for details.}
  \label{fig:embedding_nonrot_large}
\end{figure}

To reduce computational complexity, above we have only computed the entropic transfer operator and its dominant eigenfunctions for $N \approx 5000$ samples.
This allowed us to quickly extract the dominant spectrum for various $\veps$ to get an impression of the relevant length scales of the dataset, to subsequently obtain a reasonable embedding at an appropriate value for $\veps$, and to extend this embedding to the remaining datapoints in a meaningful way.
But to estimate the transfer operator itself, only a small fraction of data was used. It might be that a substantially more accurate picture could be obtained by using the full dataset.
In \cite{Koltai_2020} all samples were used to estimate the transfer operator in a second stage by applying Ulam's method to the diffusion map embedding. 
This was computationally tractable since only a relatively small number of boxes was be used for Ulam's method.
Of course it would be possible to apply the same strategy here.
However, with modern hard- and software the full dataset can also be tackled directly. GPUs are optimized for fast operations on matrices. 
For $N\approx 300\,000$, a matrix of size $N \times N$ in float32 precision would still occupy 360GB of memory and can therefore not be handled in a naive way. 
Fortunately, the matrices involved in representing $T_N^\veps = G_{\nu_N\nu_N}^\veps T_N G_{\mu_N\mu_N}^\veps$ have a very specific structure. 
For the canonical choice of basis on $\Leb^2(\mu_N)$ given by functions $u_i(x_j)=\sqrt{N}\delta_{ij}$ (and likewise on $\Leb^2(\nu_N)$), $T_N$ is represented by the identity matrix and the entropic transport matrices have a structure according to \eqref{eq:EntropicPD} with $c$ in turn being a simple function of the arrays of coordinates $(x_i)_i$ and $(y_i)_i$. 
Such structures can be represented as lazy tensors, e.g.~in the KeOps library \cite{Keops}. 
Their memory footprint only scales linearly in $N$, they can be used efficiently in GPU matrix operations, and they can be efficiently interfaced with the sparse eigenfunction extraction routines of scipy. 
This approach also does not rely on coarse-to-fine strategies as described in \cite{SchmitzerScaling2019} which only work in low dimensions, and it will scale without issues at least to dimensions on the order of 100.
In this way we were able to extract the 10 dominant eigenfunctions of $T_N^\veps$ on the whole dataset ($N=300\,361$) in less than 8 minutes on a MIG 2g.10g partition of an NVIDIA A100-SXM4-40GB GPU. 
(Of course the precise runtime will depend on the available hardware and the numerical precision. 
We ran 20 Sinkhorn iterations per transport kernel, which resulted in relative $\Leb^1$ marginal errors of approximately $10^{-4}$.)
The dominant spectra for various $N$ are shown in Figure \ref{fig:embedding_nonrot_large}, left. 
We observe that the first seven eigenvalues are virtually identical for all $N \geq 5007$ (with the exception of one spurious eigenvalue appearing for $N=60\,073$, see previous paragraph), suggesting a fast and robust convergence of the dominant part of the spectrum, even though the dimension of $\SpX$ is 24.
The runtime seems to scale approximately quadratic in $N$ (indicating that the number of matrix multiplications remained constant with respect to $N$), see Figure \ref{fig:embedding_nonrot_large}, right.
The last two paragraphs indicate that a first robust analysis of the dataset can be performed efficiently on a small subset, and it is also feasible to perform a more complete analysis with the appropriate numerical tools.

\paragraph{Analysis of the rotating tank.}
\begin{figure}[bt]
  \centering
  \includegraphics[width = 0.8\textwidth]{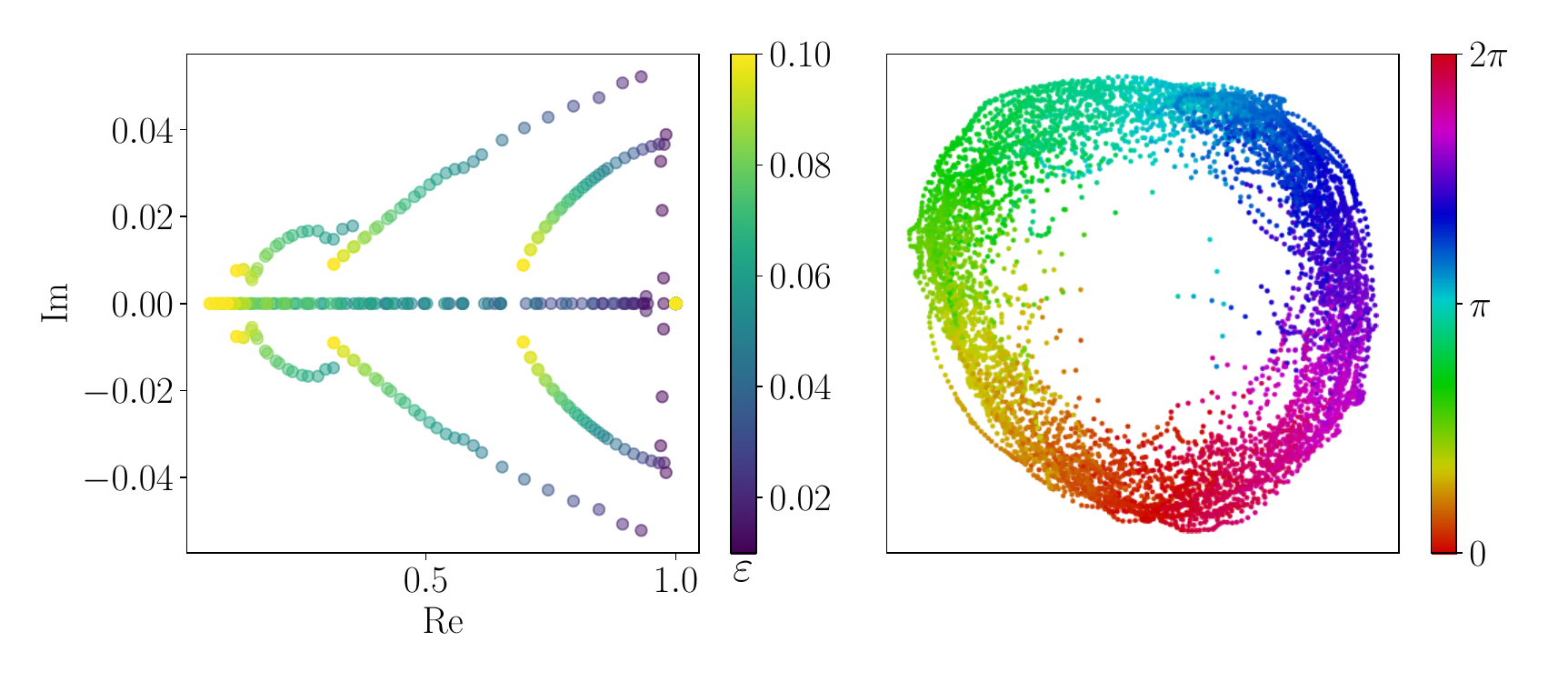}
  \caption{Left: 10 largest eigenvalues for the rotating experiment, color stands for different $\veps$.  Right: spectral embedding based on real and imaginary part of $u_2$ for $\veps=0.1$, with color encoding the roll orientation $\psi_{\tn{m}}$.}
  \label{fig:embedding_rot}
\end{figure}

\begin{figure}[bt]
  \centering
  \includegraphics[width = 0.8\textwidth]{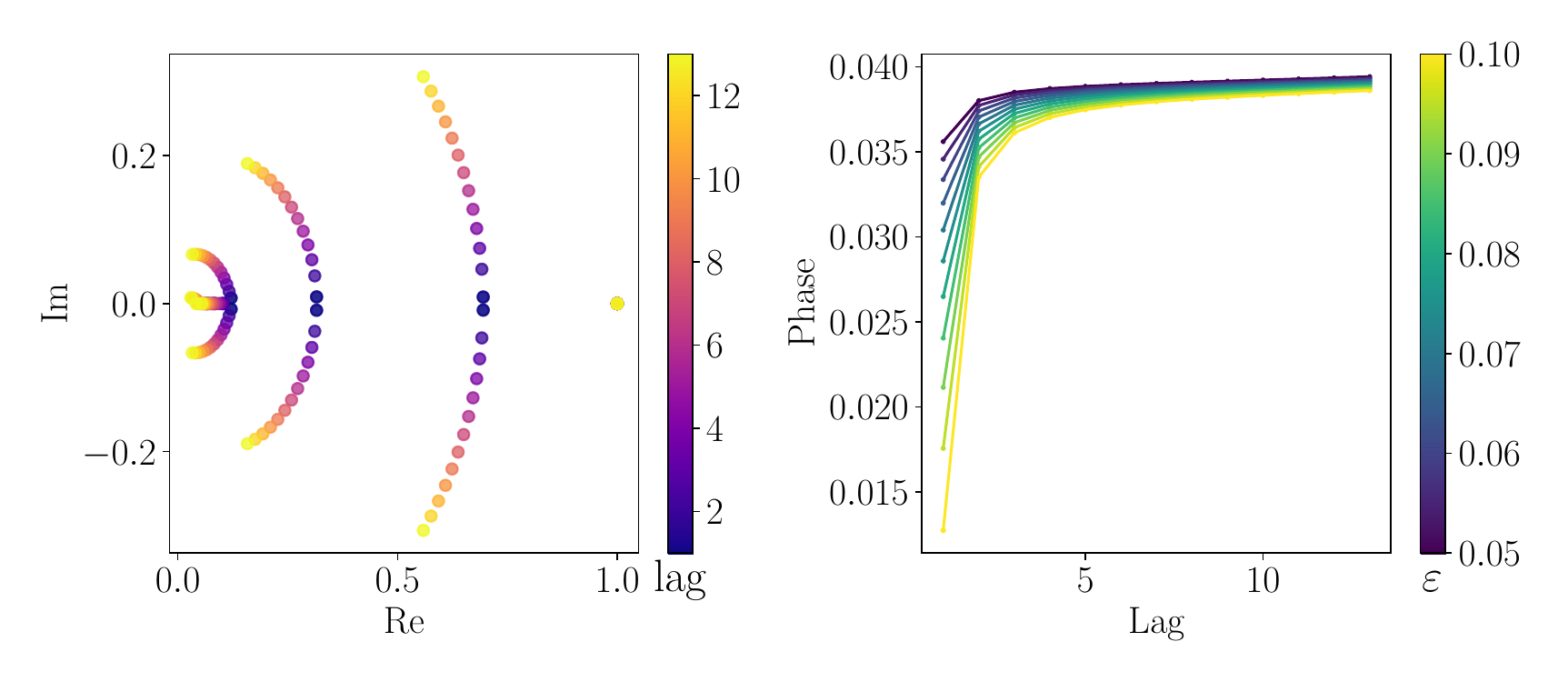}
  \caption{Left: 10 largest eigenvalues for the rotating experiment when $\veps = 0.1, \ s = 1$, and lag $l \in \{1,\ldots,13\}$ (encoded as color). Right: Phase $\arg(\lambda_2^l)/l$ when $s = 1$, for different $\veps$ (encoded as color) and $l \in \{1,\ldots,13\}$ . Color encodes the value used for the regularization parameter $\veps$. The bias decreases with decreasing $\veps$ and increasing $l$.}
  \label{fig:embedding_rot_phase}
\end{figure}

Finally, we consider a dataset of a second experiment (also studied in \cite{Koltai_2020}), where the cylinder was rotated with angular velocity $\omega_{\tn{tank}}=0.88 \tn{rad}/\tn{s}$ (at a temperature difference of $15.9$\textdegree{}C). 
Measurements were again taken approximately with constant frequency of one per $3.4\tn{s}$.
Due to the Coriolis force one expects a relative rotation between the water roll and the tank wall. 
Indeed, fitting cosine profiles and extracting the phase as in \eqref{eq:convection_fit}, we obtain mean and median velocities of $\omega_{\tn{roll,mean}}=0.013 \tn{rad}/\tn{s}$ and $\omega_{\tn{roll,median}}=0.011 \tn{rad}/\tn{s}$ for the middle layer phase $\psi_{\tn{m}}$.
The discrepancy between mean and median is due to asymmetry in the distribution. One reason for this asymmetry is that the drift is linked to the SRS, which occasionally briefly disappears.
On this level of precision the values are consistent with a mean drift of $0.012 \tn{rad}/\tn{s}$ obtained in \cite{Koltai_2020} by a more careful analysis.
Discarding once more the first $2000$ measurements, a total of 7279 datapoints was available, which can easily be analysed in full (i.e.~we set $t_0=2000$, $s=1$ in \eqref{eq:convection_fit}, different values for $\veps$ and $l$ will be used below).

The dominating spectra for lag $l=1$ and various $\veps$ are shown in Figure \ref{fig:embedding_rot}, left. 
Due to the rotation we now expect a non-zero drift of the roll orientation, which is reflected by non-real eigenvalues. 
Since entries of $T_N^\veps$ are real, non-real eigenvalues appear in conjugate pairs.
A spectral embedding based on the real and imaginary part of the eigenvector $u_2$ is shown in Figure \ref{fig:embedding_rot}, right.
The embedding is roughly ring shaped with only few samples lying near the center. 
As before, the angle is in good correspondence with the roll orientation $\psi_{\tn{m}}$. 
We find that $u_{\{3,4\}}$ and $u_{\{5,6\}}$ correspond roughly to higher order Fourier modes on the ring (not shown). 
Hence, the system appears to behave approximately like a stochastic shift on a torus, see Section \ref{subsec:num_1d} and \cite[Section 5]{EntropicTransfer22}.

We observe in Figure \ref{fig:embedding_rot}, left, that the phase of the subdominant eigenvalue $\lambda_2$ of $T_N^\veps$ depends on the regularization $\veps$. 
Apparently the regularization adds some bias towards smaller phases. 
This bias can be reduced by decreasing $\veps$, but this cannot be done arbitrarily, since eventually discretization artefacts emerge \cite[Section 5]{EntropicTransfer22}.
Alternatively one may increase the lag $l$. Let $\lambda_2^l$ be the sub-dominant eigenvalue for the choice $l$. 
It corresponds to transitions over $l$ discrete time steps, so by increasing $l$ the movement of the system becomes larger compared to the regularization strength. 
One might then expect that $\arg(\lambda_2^l)/l$ yields a more robust estimate of the phase of the eigenvalue $\lambda_2=\lambda_2^{l=1}$. 
This is confirmed in Figure \ref{fig:embedding_rot_phase}. 
The phase for the subdominant eigenvalue approaches $\approx 0.039 \tn{rad}$, which corresponds to a phase velocity of $0.011 \tn{rad}/\tn{s}$ (based on the time delta $3.4\tn{s}$ between measurements), consistent with the above estimates for the angular velocity based on the direct estimation of the roll phase drift.
This demonstrates that entropic transfer operators are able to extract information on the dominant features of a dynamical system. 
They can also be applied in scenarios where a direct extraction of meaningful features is not as obvious as in the case of roll orientation.

\section*{Acknowledgements}
This work was supported by the German Research Foundation (DFG) through the CRC 1456, `Mathematics
of Experiment', projects A03 (CS and BS) and C06 (HB and BS), project SCHM 3462/3-1 `Entropic transfer operators for data-driven analysis of dynamical systems' (BS and TS) and the Emmy Noether-Programme (BS).

\bibliographystyle{plain}
\bibliography{ref.bib}

\end{document}